\documentclass[11pt, leqno]{amsart}

\usepackage{html}
\usepackage{url}
\usepackage[dcucite]{harvard}

\usepackage{graphicx,amsmath,amssymb,epsfig,harvard}
\usepackage{graphicx,amsmath,amssymb,epsfig}
\usepackage{colortbl}

\long\def\comment#1{}
\long\def\comment#1{}

\newtheorem{theorem}{Theorem}

\newtheorem{lemma}{Lemma}

\theoremstyle{definition}

\newcommand{\citen}{\citeasnoun}

\newcommand{\be}{\begin{eqnarray}}
\newcommand{\ee}{\end{eqnarray}}

\newcommand{\NN}{\mathbf{N}}

\newcommand{\ba}{\begin{array}}
\newcommand{\ea}{\end{array}}
\newcommand{\bs}{\begin{align}\begin{split}\nonumber}
\newcommand{\bsnumber}{\begin{align}\begin{split}}
\newcommand{\es}{\end{split}\end{align}}

\renewcommand{\(}{\left(}
\renewcommand{\)}{\right)}

\renewcommand{\]}{\right]}
\renewcommand{\hat}{\widehat}

\newcommand{\cc}{\bar{c}}

\newcommand{\Gn}{\mathbb{G}_n}

\newcommand{\Pn}{\mathbb{P}_n}

\newcommand{\Ep}{{\mathrm{E}}}
\newcommand{\En}{{\mathbb{E}_n}}

\newcommand{\conflvl}{\gamma}
\newcommand{\barEp}{\bar \Ep}
\newcommand{\Ena}{{\mathbb{E}_{n_a}}}
\newcommand{\Enb}{{\mathbb{E}_{n_b}}}
\newcommand{\Enk}{{\mathbb{E}_{n_k}}}

\newcommand{\PXkc}{\mathcal{P}_{\hat I^{k^c}}}

\newcommand{\MXkc}{\mathcal{M}_{\hat I^{k^c}}}

\newcommand{\MXa}{\mathcal{M}_{\hat I^{a}}}
\newcommand{\MXb}{\mathcal{M}_{\hat I^{b}}}

\renewcommand{\Pr}{{\mathrm{P}}}

\newcommand{\semin}[1]{\phi_{{\rm min}}(#1)}
\newcommand{\semax}[1]{\phi_{{\rm max}}(#1)}
\renewcommand{\hat}{\widehat}
\renewcommand{\leq}{\leqslant}
\renewcommand{\geq}{\geqslant}

\begin{document}

\title[Treatment Effects with High-Dimensional Controls]{Supplementary Appendix for ``Inference on Treatment Effects After Selection Amongst High-Dimensional Controls"}
\author[Belloni \ Chernozhukov \ Hansen]{A. Belloni \and V. Chernozhukov \and C. Hansen}

\date{First version:  May 2010,  This \sc{version} of  \today.}

\begin{abstract}
In this supplementary appendix we provide additional results, omitted proofs and extensive simulations that complement the analysis of the main text (arXiv:1201.0224).
\end{abstract}

\maketitle

\section{Intuition for the Importance of Double Selection}
To build intuition, we discuss the case where there is only one control; that is, $p=1$.  This scenario provides the simplest possible setting where variable selection might be interesting.  In this case, Lasso-type methods act like conservative $t$-tests which allows the properties of selection methods to be explained easily.

With $p = 1$, the model is
\begin{eqnarray}\label{eq: PL-simple}
&  & y_{i}  = \alpha_0 d_i + \beta_g x_i + \zeta_i, \\
& & d_i  = \beta_m x_i + v_i. \label{eq: PL-simple2}
\end{eqnarray}
For simplicity, all errors and controls are taken as normal,
\begin{eqnarray}
&  &  \(\begin{array}{cc} \zeta_i \\ v_i\end{array}\) \mid x_i \sim N \(0,\(\begin{array}{cc} \sigma^2_\zeta & 0 \\ 0& \sigma^2_{v}\end{array}\)\),  \  \   \ \ x_i \sim N(0, 1),
 \label{Def: errors}
\end{eqnarray}
where the variance of $x_i$ is normalized to be 1.
 The underlying probability space is equipped with probability measure $\Pr$ (referred to as ``dgp'' throughout the paper).
Let $\mathbf{P}$ denote the collection  of all dgps $\Pr$ where (\ref{eq: PL-simple})-(\ref{Def: errors}) hold
with non-singular covariance matrices in (\ref{Def: errors}). Suppose that we have an i.i.d. sample $(y_i, d_i, x_i)_{i=1}^n$ from the dgp $\Pr_n \in \mathbf{P}$. The subscript $n$ signifies that the dgp and all true parameter values may change with $n$ to better model finite-sample phenomena such coefficients being ``close to zero".  As in the rest of the paper, we keep the dependence of the true parameter values on $n$ implicit.  Under the stated assumption, $x_i$ and $d_i$ are jointly normal with
variances $\sigma^2_x =1$ and $\sigma^2_d = \beta_m^2 \sigma_x^2 + \sigma_v^2$ and correlation $ \rho = \beta_m \sigma_x/\sigma_d$.

The standard post-single-selection method for inference proceeds by applying model selection methods -- ranging from standard $t$-tests to Lasso-type selectors -- to the first equation only, followed by applying OLS to the selected model. In the model selection stage, standard selection methods would necessarily omit $x_i$ wp $\to 1$
if
\begin{equation}\label{eq:badPn}
|\beta_g| \leq \frac{\ell_n}{\sqrt{n}}c_n ,    \ c_n:= \frac{\sigma_\zeta}{\sigma_x\sqrt{1 - \rho^2}},   \text{ for some } \ell_n \to \infty,
\end{equation}
where $\ell_n$ is a slowly varying sequence depending only on $\textbf{P}$.  On the other hand, these methods would necessarily include $x_i$ wp $\to 1$, if
\begin{equation}\label{eq:goodPn}
|\beta_g| \geq \frac{\ell'_n}{\sqrt{n}} c_n , \text{ for some } \ell'_n > \ell_n,
\end{equation}
 where $\ell'_n$
is another slowly varying sequence in $n$ depending only on $\textbf{P}$. In most standard model selection
 devices with sensible tuning choices, we shall have $\ell_n =C \sqrt{\log n}$  and $\ell_n' = C' \sqrt{ \log n}$ with constants $C$ and $C'$ depending only on $\textbf{P}$.   In the case of Lasso methods, we prove this in Section 5.   This is also true in the case of  the conservative $t$-test, which omits $x_i$ if the $t$-statistic $|t|= |\hat \beta_g|/\textrm{s.e.}(\hat \beta_g) \geq \Phi^{-1} (1 - 1/(2n))$, where $\hat \beta_g$ is the OLS estimator, and $\textrm{s.e.}(\hat \beta_g)$ is the
corresponding standard error.   In this case, we have
 $\Phi^{-1} (1 - 1/(2n)) = \sqrt{ 2 \log n} (1 + o(1))$ so the test will have power approaching 1 for alternatives of the form (\ref{eq:goodPn}) with $\ell_n' = 2 \sqrt{ \log n}$ and power approaching 0 for alternatives of the form (\ref{eq:goodPn}) with $\ell_n = \sqrt{ \log n}$. \footnote{ This assumes that the canonical estimator of the standard error is used. }


A standard selection procedure would work with the first equation.   Under ``good" sequences of models $\Pr_n$ such that
(\ref{eq:goodPn}) holds, $x_i$ is included wp $\to 1$,  and the estimator becomes the standard OLS estimator with the standard large sample asymptotics under $\Pr_n$
$$
\sigma_n ^{-1} \sqrt{n}(\hat \alpha - \alpha_0) =  \underbrace{ \sigma^{-1}_n \En[v^2_i]^{-1} \sqrt{n} \En[ v_i \zeta_i]}_{=:i}  + o_{P} (1) \rightsquigarrow N(0,1),
$$
where  $ \sigma^2_n = \sigma^2_\zeta (\sigma^2_v)^{-1}$.    On the other hand, when $\beta_g = 0$ or $\beta_g= o(\ell_n/\sqrt{n})$ and $\rho$ is bounded away from 1, we have that
$$
 \sigma^{* -1}_n  \sqrt{n}(\hat \alpha - \alpha_0) =  \underbrace{ \sigma^{* -1}_n \En[d^2_i]^{-1} \sqrt{n} \En[ d_i \zeta_i]}_{:=i^*} + o_{P} (1)
  \rightsquigarrow N(0,1),
$$
where  $\sigma^{*2}_n = \sigma^2_{\zeta} (\sigma^2_d)^{-1}$.  The variance $\sigma^{*2}_n$ is smaller than the variance $\sigma^2_n$ from estimation with $x_i$ included if $\beta_m \neq 0$.  The potential reduction in variance is often used as a ``motivation" for the standard selection procedure.  The estimator is super-efficient, achieving a variance smaller than the semi-parametric efficiency bound under homoscedasticity. That is, the estimator is ``too good.''

The ``too good" behavior of the procedure that looks solely at the first equaiton has its price.  There are plausible sequences of dgps $\Pr_n$ where $\beta_g=\frac{\ell_n'}{\sqrt{n}}c_n$, the coefficient on $x_i$ is not zero but is close too zero,\footnote{Such sequences are very relevant in that they are designed to generate approximations that better capture the fact that one cannot distinguish an estimated coefficient from 0 arbitrarily well in any given finite sample.}
in which the control $x_i$ is dropped wp $\to 1$ and
\begin{equation} \label{eq: diverges}
|\sigma_n ^{*-1}\sqrt{n}(\hat \alpha - \alpha_0) |  \rightsquigarrow  \infty.
\end{equation}
That is, the standard post-selection estimator is not asymptotically normal and even fails to be uniformly consistent
at the rate of $\sqrt{n}$.   This poor behavior occurs because the omitted variable bias created by dropping $x_i$ may be large even when the  magnitude of the regression coefficient, $|\beta_m|$, in the confounding equation (\ref{eq: PL-simple2}) is small but is not exactly zero. To see this, note
$$
\sigma_n^{*-1}\sqrt{n}(\hat \alpha - \alpha_0) = \underbrace{\sigma_n^{*-1} \En[d^2_i]^{-1} \sqrt{n} \En[ d_i \zeta_i]}_{=:i^*}
+ \underbrace{\sigma_n^{*-1} \En[d^2_i]^{-1} \sqrt{n} \En[d_i x_i] \beta_g}_{=:ii} .
$$
The term $i^*$ has standard behavior; namely   $ i^* \rightsquigarrow N(0, 1)$.
 The term $ii$ generates the \textit{omitted variable bias}, and it may be arbitrarily large, since wp $\to 1$,
$$
|ii| \geq \frac{1}{2} \frac{|\rho|}{\sqrt{ 1- \rho^2}} \ell_n \nearrow \infty,
$$
if $\ell_n |\rho| \nearrow \infty$.\footnote{Recall that $\rho = \beta_m \sigma_x/\sigma_d$, so $\ell_n |\rho| \nearrow \infty$ as long as $\ell_n|\beta_m| \nearrow \infty$ assuming that $\sigma_x/\sigma_d$ is bounded away from 0 and $\infty$.} This yields the conclusion (\ref{eq: diverges}) by the triangle inequality.

In contrast to the standard approach, our post-double-selection method for inference proceeds by applying model selection methods, such as standard $t$-tests or Lasso-type selectors, to both equations and taking the selected controls as the union of controls selected from each equation.  This selection is than followed by applying OLS to the selected controls.  Thus,
our approach drops $x_i$ only if the omitted variable bias term $ii$ is small. To see this, note that the double-selection-methods \textit{include} $x_i$ wp $\to 1$ if its coefficient in either (\ref{eq: PL-simple}) or (\ref{eq: PL-simple2}) is not very small.  Mathematically, $x_i$ is included if
\begin{equation}\label{eq:goodPn2}
\text{ either } |\beta_g|  \geq  \frac{\ell'_n}{\sqrt{n}} \left(\frac{\sigma_\zeta}{\sigma_x\sqrt{1 - \rho^2}}\right )  \text{ or }  |\beta_m| \geq \frac{\ell'_n}{\sqrt{n}} \left(\frac{\sigma_v}{\sigma_x}\right )
\end{equation}
 where $\ell'_n$
is a slowly varying sequence in $n$.   As already noted, $\ell_n \propto \ell_n' \propto \sqrt{\log n}$ would be standard for Lasso-type methods as well as for using simple t-tests to do model selection. Considering t-tests and using these rates, we would omit $x_i$
 if both $|t_g|= |\hat \beta_g|/\textrm{s.e.}(\hat \beta_g)
\leq \Phi^{-1} (1 - 1/(2n))$ \textit{and} $|t_m|= |\hat \beta_g|/\textrm{s.e.}(\hat \beta_g)
\leq \Phi^{-1} (1 - 1/(2n)) $ where $\hat \beta_g$ and $\hat \beta_m$ denote the OLS estimator from each equation and $\textrm{s.e.}$
denotes the corresponding estimated standard errors.  Note that the critical value used in the t-tests above is conservative in the sense that the false rejection probability is tending to zero because $\Phi^{-1} (1 - 1/(2n)) = \sqrt{ 2 \log n} (1 + o(1)).$  We note that Lasso-type methods operate similarly.

Given the discussion in the preceding paragraph, it is immediate that the post-double selection estimator satisfies
\begin{equation}\label{eq: regular behavior}
\sigma_n^{-1} \sqrt{n}  (\check{\alpha}- \alpha_0) =   i + o_{P}(1)   \rightsquigarrow  N(0,1)
\end{equation}
under any sequence of $\Pr_n \in \mathbf{P}$.
We get this approximating distribution whether or not $x_i$ is omitted.  That this is the approximate distribution when $x_i$ is included follows as in the single-selection case.  To see that we get the same approximation when  $x_i$ is omitted,
note that we drop $x_i$ only if
\begin{equation}\label{eq:dropping event}
\text{ both } |\beta_g|  \leq  \frac{\ell_n}{\sqrt{n}} c_n  \text{ and }  |\beta_m'| \leq \frac{\ell_n}{\sqrt{n}} (\sigma_v/\sigma_x) ,
\end{equation}
i.e. coefficients in front of $x_i$ in both equations are small.
In this case,
$$
\sigma_n^{*-1}\sqrt{n}(\check \alpha - \alpha_0) = \underbrace{ \sigma_n^{*-1}\En[d^2_i]^{-1} \sqrt{n} \En[ d_i \zeta_i]}_{=i^*} +
 \underbrace{ \sigma_n^{*-1}\En[d^2_i]^{-1} \sqrt{n} \En[d_i x_i] \beta_g}_{=ii}.
$$
Once again, the term $ii$ is due to \textit{omitted variable bias}, and it obeys wp $\to 1$ under  (\ref{eq:dropping event})
$$
|ii| \leq  2 \sigma^{-1}_{\zeta} \sigma_d  \sigma_d^{-2} \sqrt{n}\sigma_x^2 |\beta_m \beta_g | \leq  2 \frac{\sigma_v/\sigma_d}{\sqrt{1- \rho^2}} \frac{\ell_n^2}{\sqrt{n}}  =    2 \frac{\ell_n^2}{\sqrt{n}}   \to 0,
$$
since $(\sigma_v/\sigma_d)^2 = 1- \rho^2$.  Moreover, we can show
$i^* - i = o_P(1)$ under such sequences, so the first order asymptotics of $\check{\alpha}$ is the same whether $x_i$ is included
or excluded.

To summarize, the post-single-selection estimator may not be root-$n$ consistent in sensible models which translates
into bad finite-sample properties.  The potential poor finite-sample performance may be clearly seen in Monte-Carlo experiments.  The estimator
$\hat \alpha$ is thus non-regular: its first-order asymptotic properties depend on the model sequence $\Pr_n$ in a strong way.
In contrast, the post-double selection estimator $\check \alpha$ guards against omitted variables bias which reduces the dependence of the first-order behavior on $\Pr_n$.  This good behavior
under sequences $\Pr_n$ translates into uniform with respect to $\Pr \in \textbf{P}$ asymptotic normality.

We should note, of course, that the post-double-selection estimator is first-order equivalent to the long-regression in this model.\footnote{This equivalence may be a reason double-selection was previously overlooked.  There are higher-order differences between the long regression estimator and our estimator.  In particular, our estimator
has variance that is smaller than the variance of the long regression estimator by a factor that is proportional to $(1-\ell_n/\sqrt{n})$ if $\beta_g = 0$; and when $\beta_g = \ell_n/\sqrt{n}$, there is a small reduction in variance that is traded against a small increase in bias.} This equivalence disappears under approximating sequences with number of controls proportional to the sample size, $p \propto n$, or greater than the sample size, $p \gg n$.  It is these scenarios that motivate the use of selection as  a means of regularization.  In these more complicated settings the intuition from this simple $p=1$ example carries through, and the post-single selection method has a highly non-regular behavior while the post-double selection method continues to be regular.

It is also informative to consider semi-parametric efficiency in this simple example.  The post-single-selection estimator is super-efficient when $\beta_m \neq 0$ and $\beta_g =0$.  The super-efficiency in this case is apparent upon noting that the estimator is root-n consistent and normal with asymptotic variance $\Ep[\zeta^2_i] \Ep[d_i^2]^{-1}$.  This asymptotic variance is generally smaller than the semi-parametric efficiency bound $\Ep[\zeta^2_i]\Ep [v_i^2]^{-1}$.  The price of this efficiency gain is the fact that the post-single-selection estimator breaks down when $\beta_g$ may be small but non-zero.   The corresponding confidence intervals therefore also break down.  In contrast, the post-double-selection estimator remains well-behaved in any case, and confidence intervals
based on the double-selection-estimator and are uniformly valid for this reason.

\section{Extensions: Other Problems and Heterogeneous Treatment Effects}

\subsection{Other Problems}
In order to discuss extensions in a very simple manner, we assume i.i.d sampling as well as assume away approximation errors, namely $g(z_i) = x_i'\beta_{g0}$ and $m(z_i) = x'_i \beta_{m0}$, where parameters
$\beta_{g0}$ and $\beta_{m0}$ are high-dimensional and that $x_i = P(z_i)$ as before. In this paper we considered a moment condition:
\begin{equation}
\Ep[\varphi(y_i- d_i\alpha_0- x_i'\beta) v_i] =0,
\end{equation}
where $\varphi(u) = u$ and $v_i$ are measurable functions of $z_i$, and the target parameter is $\alpha_0$.   We selected the  instrument $v_i$ such that the equation is first-order insensitive to the parameter $\beta$ at $\beta = \beta_{g0}$:
\begin{equation}\label{immune}
\left.\frac{\partial}{\partial \beta} \Ep[\varphi(y_i -d_i\alpha_0-x_i'\beta) v_i] \right|_{\beta= \beta_{g0}} =0.
\end{equation}
Note that $\varphi(u)= u$ and  $v_i =d_i - m(z_i)$  implement
this condition.  If (\ref{immune}) holds, the estimator of $\alpha_0$ gets ``immunized" against nonregular estimation of $\beta_0$, for example, via a post-selection procedure or other regularized estimators.   Such immunization ideas are in fact behind the classical Frisch-Waugh-Robinson  partialling out technique in the linear setting and the \citeasnoun{Neyman1979}'s $C(\alpha)$ test in the nonlinear setting.   Our contribution here
is to recognize the importance of this immunization in the context of post-selection inference,\footnote{To the best of our knowledge, all prior theoretical and empirical work uses the standard post-selection approach based on the outcome
equation alone, which is highly non-robust way of conducting inference, as shown in extensive monte-carlo, in Section 2.4, and in a sequence of fundamental critiques by Leeb and Potscher, see \citen{leeb:potscher:review} .}  to develop
a robust post-selection approach to inference on the target parameter, and characterize the uniformity regions of this procedure. Our approach uses modern selection methods to estimate $g$ and the function $m$ defining the intstrument $v$.  In an ongoing work, we explore other regularization methods, such as the ridge method or combination of ridge method with Lasso methods, and characterize uniformity regions of the resulting procedures.  Also, generalizations to nonlinear models, where $\varphi$ is non-linear and can correspond to a likelihood score or quantile check function are given in \citeasnoun{BCK-LAD} and \citeasnoun{BCY-honest}; in these generalizations achieving (\ref{immune}) is also critical.

Within the context of this paper, a potentially important extension is to consider a general treatment
effect model, where $d_i$ is interacted with transformations of $z_i$.
As long as the interest lies in a particular regression coefficient, the current framework covers
this implicitly since  $x_i$ could contain interactions of $d_i$ with transformations of controls $z_i$.
In the case  a fixed number of such regression coefficients is of interest,
we can estimate each of the coefficients by re-labeling the corresponding regressor as $d_i$ and other regressors
as $x_i$ and then applying our procedure the fixed number of times.  Such component-wise procedure is valid as long as our regularity conditions hold for each of the resulting regression models in this manner.

A related research direction being pursued is the study of estimation of average treatment effects when
treatment effects are fully heterogeneous. When the treatment variable $d_i \in \{0,1\}$ is binary (or discrete more generally) our approach is readily amenable
to this problem.  In this case the parameter of interest is the average treatment effect
$$\alpha_0 = \Ep[ g(1,z_i) - g(0,z_i)],$$
where $g(d_i,z_i) = \Ep[y_i|d_i,z_i]$. We can write,  assuming again no approximation errors for simplicity, $g(d_i,z_i) = d_ix_i'\beta_{g0,1} + (1-d_i) x_i'\beta_{g0,0}.$
 Suppose that the  propensity score $\Pr(d_i=1\mid z_i)$ is $m(z_i) = \Lambda(x_i'\beta_{m0})$, where $\Lambda(u)$ is a link such
 as logit or linear, then we can use the moment equations
of \citeasnoun{hahn-pp}:
\begin{eqnarray}\label{eq:Hahn}
 &  \Ep\left [ \varphi(\alpha_0, y_i, d_i, g(0, z_i), g(1, z_i), m(z_i)) \right] =0,
 \end{eqnarray}
 where $\varphi(\alpha, y, d,  g_0, g_1, m) =   \alpha -  \frac{d (y -g_1) } {m} + \frac{(1-d) (y -g_0)}{1-m} -    (g_1 - g_0).$
It is straightforward to check that for each $j  \in \{1, 2, 3\}$:
\begin{equation}
\frac{\partial }{\partial \beta_{j}} \Ep\left [ \varphi_i(\alpha_0, y_i, d_i, x_i'\beta_{1},  x_i\beta_{2}, x_i' \beta_{3}) \right] =0,  \ \  \
\end{equation}
when $(\beta_1, \beta_2, \beta_3) = (\beta_{g0,0}, \beta_{g0,1}, \beta_{m0})$ (i.e., evaluated at the true values).   Therefore, by using \citeasnoun{hahn-pp}'s equation we obtain immunization against the crude estimation of either $\beta_{g0}$ or $\beta_{m0}$, just like we do in the partially linear case. Hence we can use the selection approach to regularization and estimate the parameter
of interest $\alpha_0$. Note that in this case the resulting procedure is a double selection method, where terms explaining propensity score
and the regression function are selected. Here too we can use the ``union" approach in fitting each regression function involved, as it gives the best finite-sample
performance in extensive computational experiments.   Using the results of this paper,
it is not difficult to show for the case that $\Lambda(u) =u$ that under the sparsity assumption imposed on both $g$ and $m$,
and additional assumptions needed to guarantee consistency of post-Lasso estimators $\hat g$ and $\hat m$ in the uniform norm, that the post-double-selection estimator $\check \alpha$ that solves $ \En \left [ \varphi(\check \alpha, y_i, d_i, \hat g(0, z_i), \hat g(1, z_i), \hat m(z_i)) \right] =0,$
has  the following large sample behavior:
\begin{equation}
\sigma^{-1}_n \sqrt{n} (\check \alpha - \alpha_0) \rightsquigarrow  N(0,  1), \ \  \sigma^2_n =\Ep[\varphi^2(\alpha_0, y_i, d_i, g(0, z_i), g(1, z_i), m(z_i))],
\end{equation}
where the latter is
the semiparametric efficiency bound of \citeasnoun{hahn-pp}.  Formal
result is proven and stated below, and other formal results
along these lines are given in an ongoing work that studies this as well as other types of  effects.

\subsection{Theoretical Results on ATE with Heterogeneity} Consider i.i.d.
sample $(y_i,d_i,z_i)_{i=1}^n$ on the probability space $(\Omega , \mathfrak{F}, \Pr)$,
where we  call $\Pr$ the data-generating process.
Consider the case where treatment variable is binary $d_i \in \{0,1\}$, and
the outcome and propensity equations, as before,
 \begin{eqnarray}\label{eq: PL1}
 & y_{i}  = g(d_i, z_i) + \zeta_i,  &  \Ep[\zeta_i \mid z_i, d_i]= 0,\\
  & d_i  = m(z_i) + v_i, \label{eq: PL2}  &   \Ep[v_i \mid z_i] = 0.
\end{eqnarray}
The first target parameter  is the average treatment effect:
$\alpha_0 = \Ep[ g(1,z_i) - g(0,z_i)],$ defined above, which is implicitly indexed by $\Pr$, like other parameters.  In this model  $d_i$ is \textit{not} additively separable.  The purpose  of this section is to show that our analysis easily extends to this case, using
 our techniques.

The confounding factors $z_i$ affect the policy variable via the propensity score $m(z_i)$ and the outcome variable via the function $g(d_i, z_i)$. Both of these functions are unknown and potentially complicated. As in the main text, we use linear combinations of control terms $x_i = P(z_i)$ to approximate $g(z_i)$ and $m(z_i)$, writing (\ref{eq: PL1}) and (\ref{eq: PL2}) as
\begin{eqnarray}\label{eq: appPL1}
& & y_{i}  = \underbrace{\tilde x_i'\beta_{g0} + r_{gi}}_{g(d_i,z_i)} + \zeta_i, \\
& & d_i  = \underbrace{\Lambda(x_i'\beta_{m0}) + r_{mi}}_{m(z_i)} + v_i \label{eq: appPL2},
\end{eqnarray}
where $r_{gi}$ and $r_{mi}$ are the  approximation errors, and
\begin{eqnarray}
&&  \tilde x_i := (d_i  x_i', (1-d_i) x_i')', \  \ \beta_{g0} := (\beta_{g0,1}', \beta_{g0,0}'),  \ x_i := P(z_i),
\end{eqnarray}
where $ x_i'\beta_{g0,1}$, $ x_i'\beta_{g0,0}$, and $x_i'\beta_{m0}$ are approximations to $g(1, z_i)$, $g(0, z_i)$, and $m(z_i)$, and $\Lambda(u) = u$ for the case of  linear link and $\Lambda(u) = e^u/(1+ e^u)$
 for the case of the logistic link. In order to allow for a flexible specification and incorporation of pertinent confounding factors, the vector of controls, $x_i = P(z_i)$, we can have a dimension $p=p_n$ which can be large relative to the sample size.

The efficient moment condition, derived by \citeasnoun{hahn-pp}, for parameter $\alpha_0$ is as follows:
\begin{eqnarray}\label{eq:Hahn}
 &  \Ep\left [ \varphi(\alpha_0, y_i, d_i, g(0, z_i), g(1, z_i), m(z_i)) \right] =0,
 \end{eqnarray}
 where
 $$\varphi(\alpha, y, d,  g_0, g_1, m) =   \alpha -  \frac{d (y -g_1) } {m} + \frac{(1-d) (y -g_0)}{1-m} -    (g_1 - g_0).$$

 The post-double-selection estimator $\check \alpha$ that solves
\begin{equation}
 \En \left [ \varphi(  \check \alpha, y_i, d_i, \hat g(0, z_i), \hat g(1, z_i), \hat m(z_i)) \right] =0,
\end{equation}
where $\hat g(d_i,z_i)$ and $\hat m(z_i)$ are post-Lasso estimators of functions $g$ and $m$ based
upon equations (\ref{eq: appPL1})-(\ref{eq: appPL2}). In case of the logistic link $\Lambda$,
Lasso for logistic regression is as defined in \citen{vdGeer} and \citen{Bach2010}, and the associated post-Lasso estimators are as those defined in \citeasnoun{BCY-honest}.

In what follows, we use $\| w_i\|_{\Pr,q}$ to denote the $L^q(\Pr)$
norm of a random variable $w_i$ with law determined by $\Pr$,
and we $\| w_i\|_{\Pn,q}$ to denote the empirical
$L^q(\Pn)$ norm of a random variable with law determined by the emprical
measure $\Pn = n^{-1} \sum_{i=1}^n \delta_{w_i}$, i.e., $\| w_i\|_{\Pn,q} = (n^{-1} \sum_{i=1}^n \|w_i\|^q)^{1/q}$.

Consider fixed positive sequences $\delta_n \searrow 0$ and $\Delta_n \nearrow 0$ and constants
$C>0, c>0, 1/2 > c'>0$, which will not vary with $\Pr$. \\

\textbf{Condition HTE} ($\Pr$) .  \textbf{Heterogeneous Treatment Effects}. \textit{ Consider i.i.d.
sample $(y_i,d_i,z_i)_{i=1}^n$ on the probability space $(\Omega , \mathfrak{F}, \Pr)$,
where we shall call $\Pr$ the data-generating process, such that equations (\ref{eq: appPL1})-(\ref{eq: appPL2})
holds, with $d_i \in \{0,1\}$.  (i) Approximation errors satisfy $\|r_{gi}\|_{\Pr,2} \leq \delta_n n^{-1/4}$, $\|r_{gi}\|_{\Pr,\infty} \leq \delta_n $, and $\|r_{mi}\|_{\Pr,2} \leq \delta_n n^{-1/4}$, $\|r_{mi}\|_{\Pr,\infty} \leq \delta_n $. (ii)
With $\Pr$-probability  no less than $1- \Delta_n$, estimation errors satisfy
$\|\tilde x_i'(\hat \beta_{g} - \beta_{g0})\|_{\Pn,2} \leq \delta_n n^{-1/4}$,  $\|x_i'(\hat \beta_{m} - \beta_{m0})\|_{\Pn,2} \leq \delta_n n^{-1/4}$, $K_n \|\hat \beta_{m} - \beta_{m}\|_1 \leq \delta_n $,  $K_n \|\hat \beta_{m} - \beta_{m0}\|_1 \leq \delta_n $,
estimators and approximations are sparse, namely $\| \hat \beta_{g}\|_0 \leq Cs $, $\|\hat \beta_{m}\|_0 \leq Cs$,
 and $\|  \beta_{g0}\|_0 \leq Cs $, $\|\beta_{m0}\|_0 \leq Cs$
and the empirical and populations norms are equivalent on sparse subsets, namely $\sup_{\|\delta\|_{0} \leq 2 C s} \left | \|\tilde x_i'\delta\|_{\Pn,2}/\|\tilde x_i'\delta\|_{\Pr,2} -1 \right | \leq \delta_n$. (iii)  The following  boundedness conditions hold: $\|x_{ij}||_{\Pr, \infty} \leq K_n$ for each $j$,  $\|g\|_{\Pr,\infty} \leq C$, $\|y_i\|_{\Pr,\infty} \leq C$, $\Pr( c' \leq m(z_i) \leq 1-c') =1$, and $\|\zeta_i^2\|_{\Pr,2} \geq c$. (iv) The sparsity index obeys the following growth condition, $(s \log (p \vee n))^2/n \leq \delta_n$. }\\

These conditions are simple high-level conditions, which encode
both the approximate sparsity of the models as well as impose some
reasonable behavior on the post-selection estimators of $m$ and $g$ (or other sparse estimators). These conditions are implied by other more primitive conditions in the literature. Sufficient conditions for the equivalence between population and empirical sparse eigenvalues are given in \citen{RudelsonZhou2011} and \citen{RudelsonVershynin2008}. The boundedness conditions are made to simplify arguments,
and they could be dealt away with more complicated proofs, under more stringent side conditions.

\begin{theorem}[\textbf{Uniform Post-Double Selection Inference on ATE}]\label{ATE-theorem} Consider the set $\mathbf{P}_n$ of data generating processes $\Pr$ such that equations
(\ref{eq: PL1})-(\ref{eq: PL2}) and Condition ATE (P) holds.  (1) Then under any sequence $\Pr \in \mathbf{P}_n$,
\begin{equation}
\sigma^{-1}_n \sqrt{n} (\check \alpha - \alpha_0) \rightsquigarrow  N(0,  1), \ \  \sigma^2_n = \Ep[\varphi^2(\alpha_0,y_i, d_i, g(0, z_i), g(1, z_i), m(z_i))],
\end{equation}
(2) The result continues to hold with $\sigma^2_n$ replaced by $\hat \sigma^2_n:=
\En[\varphi^2( \alpha_0, y_i, d_i, \hat g(0, z_i), \hat g(1, z_i), \hat m(z_i))]$.   (3) Moreover, the confidence regions based upon
post-double selection estimator $\check \alpha$ have the uniform asumptotic validity,
$$
\lim_{n \to \infty} \sup_{\Pr \in \mathbf{P}_n}| \Pr \left ( \alpha_0 \in [ \check \alpha \pm \Phi^{-1} (1-\xi/2) \hat \sigma_n /\sqrt{n}]\right) - (1- \xi)| = 0 .
$$
\end{theorem}

The next target parameter  is the average treatment effect on the treated:
$$\gamma_0 = \Ep[ g(1,z_i) - g(0,z_i)|d_i=1].$$
The efficient moment condition, derived by \citeasnoun{hahn-pp}, for parameter $\gamma_0$ is as follows:
\begin{eqnarray}\label{eq:Hahn2}
 &  \Ep\left [ \tilde \varphi(\gamma_0, y_i, d_i, g(0, z_i), g(1, z_i), m(z_i), \mu) \right] =0,
 \end{eqnarray}
 where $\mu = \Ep[m(z_i)] = \Pr(d_i=1)$,  and
 $$\tilde \varphi(\gamma, y, d,  g_0, g_1, m, \mu) =    \frac{d (y -g_1) } {\mu} - \frac{m(1-d) (y -g_0)}{(1-m)\mu}
 +    \frac{d( g_1 - g_0)}{\mu }  - \gamma \frac{d}{\mu}.$$
In this case the post-double-selection estimator $\check \gamma$ that solves
\begin{equation}
 \En \left [ \tilde \varphi(  \check \gamma, y_i, d_i, \hat g(0, z_i), \hat g(1, z_i), \hat m(z_i), \hat \mu) \right] =0,
\end{equation}
where $\hat g(d_i,z_i)$ and $\hat m(z_i)$ are post-Lasso estimators (or other sparse estimators
obeying the regularity conditions posed in HTE) of functions $g$ and $m$ based
upon equations (\ref{eq: appPL1})-(\ref{eq: appPL2}), and $\hat \mu = \En[d_i]$. As mentioned before, in case of the logistic link $\Lambda$, the post-Lasso estimators are as those defined in \citeasnoun{BCY-honest}.

\begin{theorem}[\textbf{Uniform Post-Double Selection Inference on ATT}]\label{ATT-theorem} Consider the set $\mathbf{P}_n$ of data generating processes $\Pr$ such that equations
(\ref{eq: PL1})-(\ref{eq: PL2}) and Condition HTE (P) holds.  (1) Then under any sequence $\Pr \in \mathbf{P}_n$,
\begin{equation}
\sigma^{-1}_n \sqrt{n} (\check \gamma - \gamma_0) \rightsquigarrow  N(0,  1), \ \  \sigma^2_n = \Ep[\tilde \varphi^2(\gamma_0,y_i, d_i, g(0, z_i), g(1, z_i), m(z_i), \mu)],
\end{equation}
(2) The result continues to hold with $\sigma^2_n$ replaced by $\hat \sigma^2_n:=
\En[\tilde \varphi^2( \gamma_0, y_i, d_i, \hat g(0, z_i), \hat g(1, z_i), \hat m(z_i),  \mu)]$.   (3) Moreover, the confidence regions based upon
post-double selection estimator $\check \alpha$ have the uniform asumptotic validity,
$$
\lim_{n \to \infty} \sup_{\Pr \in \mathbf{P}_n}| \Pr \left ( \gamma_0 \in [ \check \gamma \pm \Phi^{-1} (1-\xi/2) \hat \sigma_n /\sqrt{n}]\right) - (1- \xi)| = 0 .
$$
\end{theorem}

\subsection{Proof of Theorems \ref{ATE-theorem} and \ref{ATT-theorem}. }

The two results have identical structure and have nearly the same proof, and so we present the proof of the  Proof of Theorems \ref{ATE-theorem} only.

  In the proof $a \lesssim b$ means that $a \leq A b$, where the constant
$A$ depends on the constants  in Condition HT only, but not on $n$ once $n \geq n_0=\min\{j: \delta_j \leq 1/2\}$, and not on $\Pr \in \mathbf{P}_n$. For the proof of claims (1) and (2) we consider a sequence $\Pr_n$ in $\mathbf{P}_n$, but for simplicity, we write  $\Pr$ throughout the proof, omitting the index $n$.  Since the argument is asymptotic,
we can just assume that $n \geq n_0$ in what follows.

Step 1.  In this step we establish claim (1).

(a) We begin with a preliminary observation.
Define, for $t= (t_1, t_2, t_3)$,
 $$\psi(y, d,  t) =  \frac{d (y -t_2) } {t_3} - \frac{(1-d) (y -t_1)}{1-t_3} +  ( t_2 - t_1).$$
The derivatives of this function with respect to $t$ obey for all $k =(k_j)_{j=1}^3 \in\mathbb{N}:  0 \leq |k| \leq 3$,
\begin{equation}\label{Lip}
|\partial^k_t \psi (y, d, t) | \leq L, \ \  \forall (y,d,t):  |y| \leq C, |t_1| \leq C,  |t_2| \leq C, c'/2 \leq |t_3| \leq 1-c'/2,
\end{equation}
where $L$ depends only on $c'$ and $C$,  $|k| = \sum_{j=1}^3 k_j,$ and $$\partial^k_t := \partial^{k_1}_{t_1}\partial^{k_2}_{t_2}\partial^{k_3}_{t_3}.$$

 (b).  Let $$
\hat h (z_i) := ( \hat g(0,z_i), \hat g(1,z_i), \hat m(z_i))', \  \ h_0(z_i) :=
(  g(0,z_i),  g(1,z_i), m(z_i))',
$$$$
f_{\hat h} (y_i, d_i, z_i) := \psi (y_i, d_i, \hat h(z_i)) , \  \  f_{h_0} (y_i, d_i, z_i) := \psi(y_i, d_i, h_0(z_i)).
$$
We observe that with probability no less than $1- \Delta_n$,  $$\hat g(0,\cdot)  \in \mathcal{G}_0, \ \ \hat g(1,\cdot)  \in \mathcal{G}_1 \text{ and  } \hat m  \in \mathcal{M},$$ \begin{eqnarray*}
\mathcal{G}_d:=\{ z \mapsto x '\beta :  \|\beta\|_0 \leq s C,  \| x_i '\beta - g(d,z_i)\|_{\Pr,2} \lesssim \delta_n n^{-1/4},
  \| x_i '\beta - g(d,z_i)\|_{\Pr, \infty}
 \lesssim  \delta_n\},
   \end{eqnarray*} \begin{eqnarray*}
\mathcal{M} :=\{z \mapsto \Lambda(x '\beta) : \|\beta\|_0 \leq s C,  \|  \Lambda(x_i '\beta) - m(z_i)\|_{\Pr,2} \lesssim  \delta_n n^{-1/4},  \| \Lambda(x_i '\beta) - m(z_i)\|_{\Pr, \infty}  \lesssim \delta_n\}.
   \end{eqnarray*}
To see this note, that under assumption HT ($\Pr$), under condition (i)-(ii), under the event occurring under condition (ii) of that assumption: for $n \geq n_0=\min\{j: \delta_j \leq 1/2\}$:
 \begin{eqnarray*}
&& \| \tilde x_i '\beta - g(d_i,z_i)\|_{\Pr,2} \leq \| \tilde x_i '(\beta - \beta_{g0})\|_{\Pr,2}+ \| r_{gi} \|_{\Pr,2}  \leq
 2 \| \tilde x_i '(\beta - \beta_{g0})\|_{\Pn,2}+ \| r_{gi} \|_{\Pr,2} \leq 4 \delta_n n^{-1/4}, \\
&& \| \tilde x_i '\beta - g(d_i,z_i)\|_{\Pr,\infty} \leq \| \tilde x_i '(\beta - \beta_{g0})\|_{\Pr,\infty}+ \| r_{gi} \|_{\Pr,\infty}  \leq   K_n \|\beta- \beta_{0g}\|_{1} + \delta_n \leq 2 \delta_n, \end{eqnarray*}
for $\beta=\hat \beta_g$, with evaluation after computing the norms, and noting that for any $\beta$
 $$ \| x_i '\beta - g(1,z_i)\|_{\Pr,2} \vee  \| x_i '\beta - g(0,z_i)\|_{\Pr,2} \lesssim \| \tilde x_i '\beta - g(d_i,z_i)\|_{\Pr,2}$$
 under condition (iii). Furthermore,
for $n \geq n_0=\min\{j: \delta_j \leq 1/2\}$:
 \begin{eqnarray*}
 \| \Lambda(x_i '\beta) - m(z_i) \|_{\Pr,2} && \leq \|  \Lambda(x_i '\beta) - \Lambda(x_i'\beta_{m0}) \|_{\Pr,2}+ \| r_{mi} \|_{\Pr,2}  \\ && \lesssim \|\partial \Lambda\|_{\infty} \| \tilde x_i '(\beta - \beta_{m0})\|_{\Pr,2}+ \| r_{mi} \|_{\Pr,2} \\
&& \lesssim \|\partial \Lambda\|_{\infty} \| \tilde x_i '(\beta - \beta_{m0})\|_{\Pn,2}+ \| r_{mi} \|_{\Pr,2} \lesssim \delta_n n^{-1/4} \\
 \| \Lambda(x_i '\beta) - m(z_i) \|_{\Pr,\infty} && \leq \|\partial \Lambda\|_{\infty} \| \tilde x_i '(\beta - \beta_{g0})\|_{\Pr,\infty}+ \| r_{mi} \|_{\Pr,\infty} \\
 &&  \lesssim   K_n \|\beta- \beta_{m0}\|_{1} + \delta_n \leq 2 \delta_n,
\end{eqnarray*}
for $\beta=\hat \beta_{m0}$, with evaluation after computing the norms.

Hence with probability at least $1- \Delta_n$,
   $$\hat h \in \mathcal{H}_n:=\{ h =  (\bar g(0,z), \bar g(1,z), \bar m(z)) \in \mathcal{G}_{0} \times \mathcal{G}_1 \times \mathcal{M}\}.$$

(c) We have that $$\alpha_0 = \Ep[f_{h_0}] \text{
and } \check \alpha = \En[f_{\hat h}],$$ so that
\begin{eqnarray*}
\sqrt{n} (\check \alpha - \alpha_0) =  \underbrace{\Gn[f_{h_0}]}_{i} + \underbrace{ (\Gn[f_{h}]- \Gn[f_{h_0}])}_{ii} +  \underbrace{ \sqrt{n} (\Ep[f_{h} - f_{h_0}])}_{iii},
\end{eqnarray*}
with $h$ evaluated at $h = \hat h$.  By Liapunov central limit theorem,
$$\sigma_n^{-1} i \rightsquigarrow N(0,1).$$

(d) Note that for $\Delta_i = h(z_i) - h_0(z_i)$,
\begin{eqnarray*}
iii & = & \sqrt{n} \sum_{|k|=1} \Ep[\partial^k_t \psi(y_i,d_i, h_0(z_i)) \Delta_i ^k]  \\
 & + & \sqrt{n} \sum_{|k|=2} 2^{-1} \Ep[\partial^k_t \psi(y_i,d_i, h_0(z_i)) \Delta_i ^k ] \\
 & + & \sqrt{n} \sum_{|k|=3} \int_0^1 6^{-1} \Ep[\partial^k_t \psi(y_i,d_i, h_0(z_i) + \lambda\Delta_i) \Delta_i ^k ] d\lambda, \\
 & =: & iii_a + iii_b + iii_c,
\end{eqnarray*}
(with $h$ evaluated at $h = \hat h$). By the law of iterated expectations and because
$$\Ep[\partial^k_t \psi(y_i,d_i, h_0(d_i, z_i))|d_i, z_i]= 0 \ \  \forall m \in \mathbb{N}^3 : |k| =1,$$
we have that $$iii_a = 0.$$ Moreover, uniformly for any
$h \in \mathcal{H}_n$ we have that
$$|iii_b| \lesssim \sqrt{n} \| h - h_0 \|^2_{\Pr,2} \lesssim \sqrt{n} (\delta_n n^{-1/4})^2 \leq \delta_n^2,$$ $$|iii_c|  \lesssim \sqrt{n} \|h - h_0\|^2_{\Pr,2}\| h-h_0\|_{\Pr,\infty} \lesssim \sqrt{n} (\delta_n n^{-1/4})^2 \delta_n \leq \delta_n^3.$$  Since $\hat h \in \mathcal{H}_{n}$ with probability $1- \Delta_n$,
we have that once $n \geq n_0$,
$$
\Pr( |iii| \lesssim \delta^2_n) \geq 1- \Delta_n.
$$

(e). Furthermore, we have that
$$
|ii| \leq \sup_{h \in \mathcal{H}_n} |\Gn[f_{h}]- \Gn[f_{h_0}]|.
 $$

 The class of functions $\mathcal{G}_d$ for $d \in \{0,1\}$ is   a union of at most $\binom{p}{Cs}$ VC-subgraph classes of functions with VC indices bounded by $C's$.  The class of functions $
\mathcal{M}$ is  a union of at most $\binom{p}{Cs}$ VC-subgraph classes of functions with VC indices bounded by $C's$
(monotone transformation $\Lambda$ preserve the VC-subgraph property). These classes are uniformly bounded and their entropies  therefore satisfy
\begin{equation*}
\log N( \varepsilon, \mathcal{M}, \|\cdot\|_{\Pn,2})+ \log N( \varepsilon, \mathcal{G}_0, \|\cdot\|_{\Pn,2}) +  \log N( \varepsilon, \mathcal{G}_1, \|\cdot\|_{\Pn,2})\lesssim  s \log p + s \log (1/\varepsilon).
\end{equation*}
Finally, the class $\mathcal{F}_n =\{f_h - f_{h_0}: h \in \mathcal{H}_n\}$ is a Lipschitz transform of $\mathcal{H}_n$  with bounded Lipschitz coefficients and with a constant envelope. Therefore, we have that
\begin{equation*}
\log N( \varepsilon, \mathcal{F}_{n}, \|\cdot\|_{\Pn,2}) \lesssim  s \log p + s \log (1/\varepsilon).
\end{equation*}

We shall invoke the following lemma derived in \citeasnoun{Belloni:Chern}.
\begin{lemma}[A Self-Normalized Maximal inequality]\label{expo3} Let $\mathcal{F}$ be a measurable function class
on a sample space. Let $F = \sup_{f \in \mathcal{F}} |f|$, and suppose that there exist some constants $\omega_n>3$
and $\upsilon>1$, such hat
\begin{equation*}
\log N(\epsilon \| F \|_{\mathbb{P}_{n},2}, \mathcal{F},\| \cdot \|_{\mathbb{P}_n,2}) \leq \upsilon m ( \log (
n\vee \omega_n) + \log (1/\epsilon )),  \ 0 < \epsilon < 1.
\end{equation*}
Then for every $\delta \in (0,1/6)$  we have
$$
\sup_{f \in \mathcal{F}} | \mathbb{G}_n(f)| \leq  C_\upsilon \sqrt{2/\delta} \sqrt{m \log (n  \vee
\omega_n)}  ( \sup_{f \in \mathcal{F}}
\|f\|_{\Pr,2} \vee \ \sup_{f \in \mathcal{F}}\|f\|_{\Pn,2}),
$$
with probability at least $1-\delta$ for some constant that $C_\upsilon$.
\end{lemma}

Then by Lemma \ref{expo3} together and some simple calculations,  we have that
\begin{eqnarray*}
|ii| &\leq & \sup_{f \in \mathcal{F}_n} | \Gn (f) | =O_{\Pr}(1) \sqrt{s \log (p \vee n)} ( \sup_{f \in \mathcal{F}_n}\|   f\|_{\Pn,2} \vee \sup_{f \in \mathcal{F}_n}\|  f\|_{\Pr,2}) \\
&\leq & O_\Pr(1) \sqrt{s \log (p \vee n)} ( \sup_{h \in \mathcal{H}_n} \|   h - h_0\|_{\Pn,2} \vee \sup_{h \in \mathcal{H}_n} \|  h - h_0\|_{\Pr,2}) = o_{\Pr}(1).
\end{eqnarray*}
The last conclusion follows because  $\sup_{h \in \mathcal{H}_n} \|  h - h_0\|_{\Pr,2} \lesssim  \delta_n n^{-1/4}$
by definition of the norm, and
\begin{eqnarray*}
 \sup_{h \in \mathcal{H}_n} \|   h - h_0\|_{\Pn,2} &&\leq O_{\Pr}(1) \cdot \( \sup_{h \in \mathcal{H}_n} \|  h - h_0\|_{\Pr,2} + \|r_{gi}\|_{\Pr,2}+ \|r_{mi}\|_{\Pr,2}\),
 \end{eqnarray*}
 where the last conclusion follows from the same argument as in step (b) but in a reverse
 order,  switching from empirical norms to population norms, using
 equivalence of norms over  sparse sets imposed in condition (ii) , and also using an application of Markov inequality
 to argue that  $\|r_{gi}\|_{\Pn,2}+ \|r_{mi}\|_{\Pn,2} = O_\Pr(1) ( \|r_{gi}\|_{\Pr,2}+ \|r_{mi}\|_{\Pr,2}).$

Step 2.  Claim (2) follows from consistency: $\hat \sigma_n/\sigma_n = 1+ o_{\Pr}(1)$,
which follows
from  $\hat \sigma_n$ being a Lipschitz transform of $\hat h$ with respect
to $\|\cdot\|_{\Pn,2}$, once $\hat h \in \mathcal{H}_n$
and the consistency of $\hat h$ for $h$ under $\|\cdot\|_{\Pn,2}$.

Step 3. Claim (3) is immediate from claims (2) and (3) by the way of contradiction.
\qed

\section{Deferred Proofs: Proof of Lemma 1}

We establish the result for Lasso (the proof for other feasible Lasso estimators is similar).

By Lemma 7 in \citen{BellChenChernHans:nonGauss}, under our choice of penalty level and loadings, we have that the condition $\lambda/n \geq 2c\|\hat\Psi^{-1}\En[\tilde x_{i}\epsilon_i]\|_\infty$ holds with probability $1-o(1)$. Thus, the conclusion of Lemma 11 of \citen{BellChenChernHans:nonGauss} holds with probability $1-o(1)$, namely
for $c_s = (\En[r_i^2])^{1/2}$ \begin{equation}\label{bound:hats}\hat s \leq s + \left( \min_{m\in \mathcal{H}} \semax{m} \right) \| \hat \Psi^{-1}\|_\infty \left( \frac{2\cc}{\kappa_{\cc}} + \frac{4\cc nc_s}{\lambda \sqrt{s}} \right)^2 \end{equation}
where $\cc = (c+1)/(c-1)$, $$\mathcal{H} = \left\{ m \in \NN :  m \geq 2 s \semax{m}  \| \hat \Psi^{-1}\|_\infty \left( \frac{2\cc}{\kappa_{\cc}} + \frac{4\cc nc_s}{\lambda \sqrt{s}} \right)^2\right\},$$ $$\kappa_{\cc} \geq \max_{m\in \NN}\frac{\sqrt{\semin{m+s}}}{\|\hat \Psi \|_\infty} \( 1 -  \sqrt{\frac{\semax{m+s}}{\semin{m+s}}} \cc\sqrt{s/m}\).$$

By Condition SE, with probability $1-o(1)$ for $n$ sufficiently large we have $\kappa_{\cc} > \kappa'/2\| \hat \Psi\|_\infty$ so that with the same probability
\begin{equation}\label{bound:first}\frac{2\cc}{\kappa_{\cc}} \lesssim 1.\end{equation} Moreover, by condition RF we have with probability $1-o(1)$ that \begin{equation}\label{bound:second}\max\{ \|\hat \Psi\|_\infty, \ \| \hat \Psi^{-1}\|_\infty\} \lesssim 1.\end{equation}  Finally, since $\lambda \gtrsim \sqrt{n\log(p\vee n)}$ we have
\begin{equation}\label{bound:third}\frac{4\cc nc_s}{\lambda \sqrt{s}} \lesssim \frac{\sqrt{n} c_s}{\sqrt{s \log(p\vee n)}}\lesssim 1 \ \ \ \mbox{with probability $1-o(1)$}\end{equation}
since $c_s \lesssim_P \sqrt{s/n}$ by condition ASM and Chebyshev inequality.

Therefore, for some constant $\tilde C$,  we have $\tilde C s \in \mathcal{H}$, so that $\min_{m \in \mathcal{H}} \semax{m} \leq \kappa''$ for $n$ sufficiently large with probability $1-o(1)$ by Condition SE. In turn combining this bound with (\ref{bound:first}), (\ref{bound:second}) and (\ref{bound:third}) into (\ref{bound:hats}) we have that $ \hat s \lesssim s$ holds with probability $1-o(1)$ which is the first statement of (i).

To show the second statement in (i), note that
$$
\min_{\beta \in \Bbb{R}^p: \ \beta_j = 0 \  \forall j \not \in \widehat T} \sqrt{\En[ f(\tilde z_i) -  \tilde x_i'\beta]^2} \leq \sqrt{\En[ f(\tilde z_i) -  \tilde x_i'\hat \beta]^2}
$$ where $\hat \beta$ is the Lasso estimator. Again by Lemma 7 in \citen{BellChenChernHans:nonGauss} we have that the assumptions of Lemma 6 in \citen{BellChenChernHans:nonGauss} hold with probability $1-o(1)$. Using Condition SE to bound $\kappa_{\cc}$ from below and Condition RF to bound $\|\hat\Psi\|_\infty$ from above with probability $1-o(1)$ as before, and $\lambda \lesssim \sigma \sqrt{n\log(p\vee n)}$, it follows from  Lemma 6 in \citen{BellChenChernHans:nonGauss} that with probability $1-o(1)$ that
$$ \sqrt{\En[ f(\tilde z_i) -  \tilde x_i'\hat \beta]^2}
 \lesssim \sigma \sqrt{ \frac{ s \log (p \vee n)}{n} }.$$

The results regarding Post-Lasso in (ii) follow similarly by invoking Lemma 8 in \citen{BellChenChernHans:nonGauss}.

\section{Split-Sample Estimation and Inference}
In this section we discuss a variant of the double selection estimator based on sample splitting. The motivation for the split-sample estimator is that its use allows us to relax the requirement $s^2\log^2 (p\vee n) = o(n)$ that is assumed in the full-sample counterpart to the milder condition
$$ s \log (p\vee n) = o(n).$$

To define the estimator, divide the sample randomly into (approximately) equal parts $a$ and $b$ with sizes $n_a = \lceil n/2 \rceil$ and  $n_b = n - n_a$.  We use superscripts $a$ and $b$ for variables in the first and second subsample respectively. We let the index $k=a,b$ refer to one of the subsamples and let $k^c=\{a,b\}\setminus \{k\}$ refer to the other.

For each subsample $k=a,b$, the model $\widehat I^k$ is selected based on the subsample $k$ independently from the subsample $k^c$. In what follows the model $\widehat I^k$ is used to fit the subsample $k^c$. A constructive way to obtain $\widehat I^a$ and $\widehat I^b$ is to apply the double selection method for each subsample to select the sets of controls $\widehat I^a:= \widehat I_1^a \cup \widehat I_2^a \cup \widehat I_3^a$ and $\widehat I^b:= \widehat I_2^b \cup \widehat I_2^b \cup \widehat I_3^b$.

Then we form estimates in the two subsamples
$$
(\check \alpha^a, \check \beta^a) = \underset{ \alpha \in \Bbb{R}, \beta \in \Bbb{R}^p}{\rm argmin}\{ \Ena[(y_{i} - d_i \alpha - x_i'\beta)^2] \ : \ \beta_j = 0, \forall j \not \in \hat I^b \}, \mbox{and}
$$
$$
(\check \alpha^b, \check \beta^b) = \underset{ \alpha \in \Bbb{R}, \beta \in \Bbb{R}^p}{\rm argmin}\{ \Enb[(y_{i} - d_i \alpha - x_i'\beta)^2] \ : \ \beta_j = 0, \forall j \not \in \hat I^a \}.$$
For an index $i$ in the subsample $k$, we  define the residuals
\begin{align}\label{SSzetaResid0}
\hat \zeta_i^o &:= [y_i - d_i\check \alpha_{k} - x_i'\check \beta_{k}]\{n_k/(n_k - \hat s^{k^c}-1)\}^{1/2}  \\
\label{SSvResid}
\hat v_i &:= d_i - x_i'\hat\beta_{k}\ \textnormal{and}\\
\label{SSzetaResid}
\hat \zeta_i &:= \hat\zeta_i^o \ 1\{|\hat\zeta_i^o|\vee|\hat v_i|\leq Cn^{1/2}/[(\hat s^{k^c}\vee n^{1/2})\log n]^{1/2} \}
\end{align}
where $\hat \beta_k \in \arg\min_\beta \{ \Enk[(d_i-x_i'\beta)^2]:\beta_j=0, \forall j\notin \widehat I^{k^c}\}$ and $\widehat s^{k^c} = |\widehat I^{k^c}|$.

Finally, we combine the estimates into the split-sample estimator based on $\widehat I^a$ and $\widehat I^b$ is defined as
 \begin{equation}\label{Def:SplitDScombined}
\check \alpha_{ab} = \{ (n_a/n)  \Upsilon^a +  (n_b/n) \Upsilon^b  \}^{-1} \{(n_a/n)  \Upsilon^a \check \alpha_a + (n_b/n) \Upsilon^b \check \alpha_b\},
 \end{equation} where $\Upsilon^k = D^k{}'\MXkc D^k/n_k$.

We state below sufficient conditions for the analysis of the
split-sample method.

\textbf{Condition ASTESS ($\Pr$)}. \textit{(i) $\{(y_i,d_i,z_i),i=1,\ldots,n\}$ are i.n.i.d. vectors on $(\Omega,\mathcal{F},\Pr)$ that obey the model
(2.2)-(2.3), and the vector $x_i=P(z_i)$ is a dictionary of transformations of $z_i$, which may depend on $n$ but not on $\Pr$.  (ii)  The true parameter value $\alpha_0$, which may depend on $\Pr$, is bounded,  $\|\alpha_0\| \leq C$. (iii) Functions $m$ and $g$ admit
an approximately sparse form.  Namely there exists $s \geq 1$ and $\beta_{m0}$ and $\beta_{g0}$, which depend on $n$ and $\Pr$, such that
\begin{eqnarray}
&&m(z_i) = x_i' \beta_{m0} + r_{mi},  \ \ \|\beta_{m0}\|_0 \leq s, \ \  \{ \barEp [r_{mi}^2]\}^{1/2} \leq C \sqrt{s/n}, \\
&&g(z_i) = x_i' \beta_{g0} + r_{gi},  \ \ \ \ \ \|\beta_{g0}\|_0 \leq s, \ \ \   \{ \barEp[r_{gi}^2]\}^{1/2} \leq  C \sqrt{s/n}.
\end{eqnarray}
(iv) The sparsity index obeys $s \log (p\vee n)/n \leq C\delta_n$.
(v) For each subsample $k=a,b$, the model $\widehat I^{k^c}$ satisfies condition HLMS. (vi)  We have $\barEp[|v_i^q|+| \zeta_i^q|] \leq C$ for some $q>4$ and $n^{2/q}s\log(p\vee n) / n \leq C\delta_n$.}

The Conditions ASTESS(i)-(iii) agree with the corresponding conditions in ASTE.
The remaining conditions ASTESS(iv)-(v) are implied by Condition ASTE.
We note that Condition ASTESS(vi) is needed only for obtaining consistent estimates
of the asymptotic variance. Such conditions are mild since they do not require uniform
estimation of the functions $g$ and $m$.

The next result establishes that the split-sample
estimator $\check \alpha_{ab}$ has similar large sample properties
to the full-sample double-selection estimator under weaker growth condition.

\begin{theorem}[Inference on Treatment Effects, Split Sample]\label{theorem:inferenceSS} Let $\{\Pr_n\}$ be a sequence of data-generating processes. Assume conditions ASTESS($\Pr$)(i-v), SM($\Pr$), and SE($\Pr$) hold for $\Pn=\Pr_n$ for each $n$ and each subsample. The split sample estimator $\check \alpha_{ab}$ based on $\widehat I^a$ and $\widehat I^b$ obeys,
$$
([\barEp v_i^2]^{-1}\barEp[ v_i^2\zeta_i^2] [\barEp v_i^2]^{-1})^{-1/2} \sqrt{n} (\check \alpha_{ab} - \alpha_0) \rightsquigarrow N(0,1).
$$
Moreover, if Condition ASTESS($\Pr$)(vi) also holds, the result continues to apply if
$\barEp[v_i^2]$ and $\barEp[v_i^2\zeta_i^2]$ are replaced by  $\En[\hat v_i^2]$
and $\En[\hat v_i^2\hat \zeta_i^2]$ for $\hat \zeta_i$ and $\hat v_i$ defined in (\ref{SSzetaResid}) and (\ref{SSvResid}).
\end{theorem}
\begin{proof}
We use the same notation as in the proof of Theorem 1 
with the addition of sub/superscripts indicating the appropriate subsample $k=a,b$, where $k^c = \{a,b\}\setminus \{k\}$.

Step 0.(Combining) In this step we combine both subsample estimators.
Letting $\Upsilon^k = D^k{}'\MXkc D^k/n_k$, for $k=a,b$, so that we have
\begin{eqnarray*}
\sqrt{n}(\check \alpha_{ab} - \alpha_0) & = &
( (n_a/n)  \Upsilon^a  +  (n_b/n) \Upsilon^b   )^{-1} \times \\
& \times &  ( (n_a/n)  \Upsilon^a  \sqrt{n} (\check \alpha_a - \alpha_0) +   (n_b/n) \Upsilon^b \sqrt{n} (\check \alpha_b - \alpha_0)   ) \\
& = &
( V'V/n + o_P(1)   )^{-1} \times \\
& \times &  ( (n_a/n)  \Upsilon^a  \sqrt{n} (\check \alpha_a - \alpha_0) +   (n_b/n) \Upsilon^b \sqrt{n} (\check \alpha_b - \alpha_0)   ) + o_P(1) \\
& = & \{ V'V/n \}^{-1} \times \{(1/\sqrt{2}) \times  {{\Gn}_a} [v_i \zeta_i{}]
+ (1/\sqrt{2}) {{\Gn}_b} [v_i \zeta_i {}]\}+ o_P(1) \\
& = & \{ V'V/n\}^{-1} \times {{\Gn}} [v_i \zeta_i] + o_P(1) \\
 \end{eqnarray*}
 where we are also using the fact that
$$
\Enk [\hat v_i^2 ] -  \Enk [ v_i^2] = o_P(1), \ \ k=a,b
$$
which follows similarly to the proofs given in Step 5.

For $\sigma_n^2 := [\barEp v_i^2]^{-1}\barEp[ v_i^2\zeta_i^2] [\barEp v_i^2]^{-1}$, define
$$
Z_n = \sigma_n^{-1} \sqrt{n}(\check \alpha_{ab} - \alpha_0)  = \Gn[z_{i,n}] + o_{\mathrm{P}}(1),$$
\noindent where $z_{i,n} = \sigma_n^{-1}v_i\zeta_i/\sqrt{n}  $ are i.n.i.d. with mean zero. We have
that for some small enough $\delta>0$
$$
\barEp|z_{i,n}|^{2+\delta} \lesssim \barEp\left[ | v_i|^{2+\delta}|\zeta_i |^{2+\delta}  \right] \lesssim  \sqrt{\barEp |v_i|^{4+2\delta}} \sqrt{\barEp |\zeta_i|^{4 + 2\delta}} \lesssim 1,
$$
by Condition SM(ii).

This condition verifies the Lyapunov condition and thus 
implies that $Z_n \to_d N(0,1)$.

Step 1.(Main) For the subsample $k=a,b$ write
$
\check \alpha_k = \[D^k{}'\MXkc D^k/n_k\]^{-1}[D^k{}'\MXkc Y^k/n_k]
$
so that
$$
\sqrt{n_k}(\check \alpha_k - \alpha_0) = \[D^k{}'\MXkc D^k/n_k\]^{-1}[D^k{}'\MXkc (g^k + \zeta^k)/\sqrt{n_k}] =: ii^{-1}_k \cdot i_k.
$$
By Steps 2 and 3, $ii_k = V^k{}'V^k/n_k + o_P(1)$ and $i_k = V^k{}'\zeta^k/\sqrt{n_k} + o_P(1)$. Next note that
$V^k{}'V^k/n_k = \Ep[V^k{}'V^k/n_k] + o_P(1)$ by Chebyshev, and we have that $\barEp_k[v_i^2\zeta_i^2]$ and $\Ep[V^k{}'V^k/n_k]$ are bounded from above and away from zero by assumption.

Step 2. (Behavior of $i_k$.) Decompose
\begin{eqnarray*}
i_k = V^k{}'\zeta^k/\sqrt{n_k} + \underset{=:i_{k,a}}{m^k{}' \MXkc g^k/\sqrt{n_k}} + \underset{=:i_{k,b}}{m^k{}'\MXkc \zeta^k/\sqrt{n_k}} + \underset{=:i_{k,c}}{V^k{}'\MXkc g^k/\sqrt{n_k}} - \underset{=:i_{k,d}}{ V^k{}'\PXkc \zeta^k/\sqrt{n_k}}.
\end{eqnarray*}
First, note that by Condition ASTESS we have
$$
|i_{k,a}| = |m^k{}'\MXkc g^k/\sqrt{n_k}| \leq \|\MXkc m^k\| \ \|\MXkc g^k\|/\sqrt{n_k} = o_P(1).
$$
Second, by the split sample construction, we have that $\hat I^{k^c}$ is independent from $\zeta^k$, and by assumption of the model $m^k$ is also independent of $\zeta^k$. Thus by Chebyshev inequality
$$
|i_{k,b}| \lesssim_P \| \MXkc m^k /\sqrt{n_k}\| = o_P(1),
$$
where the last relation follows by ASTESS.

Third, using similar independence arguments, by Chebyshev and Condition ASTESS, conclude
$$
|i_{k,c}| \lesssim_P \| \MXkc g^k /\sqrt{n_k}\| = o_P(1).
$$
Fourth, using that $\hat s^{k^c} \lesssim_P s$ by ASTESS so that $ \phi^{-1}_{\min}(\hat s^{k^c}) \lesssim_P 1$ by condition SE, we have that
$$
|i_{k,d}| \leq |\tilde \beta_{V^k}(\hat I^{k^c})' X^k{}'\zeta^k/\sqrt{n_k}|  \lesssim_P   \sqrt{ s / n } = o_P(1),
$$
by Chebyshev since $ \|X^k \tilde \beta_{V^k}(\hat I^{k^c})/\sqrt{n_k}\| \lesssim_P \sqrt{s/n_k}$ because of  the independence of the two subsamples $k$ and $k^c$.

Step 3.(Behavior of $ii_k$.) Since $ii_k = (m^k+V^k)'\MXkc (m^k+V^k)/n_k$, decompose
$$
ii_k =  V^k{}'V^k/n_k + \underset{=:ii_{k,a}}{m^k{}' \MXkc m^k/n_k} +
\underset{=:ii_{k,b}}{ 2 m^k{}'\MXkc V^k/n_k } - \underset{=:ii_{k,c}}{V^k{}'\PXkc V^k/n_k}.
$$
Then $|ii_{k,a}| = o_P(1)$ by Condition ASTESS,
$|ii_{k,b}| =o_P(1)$ by reasoning similar to deriving
the bound for $|i_{k,b}|$, and $|ii_{k,c}| =o_P(1)$
by reasoning similar to deriving the bound for $|i_{k,d}|$.

Step 4.(Auxiliary Bounds.)
Note that
$$\begin{array}{rl}
\|g^k-X^k \check \beta_k \| &  = \|g^k-\mathcal{P}_{\widehat I^{k^c}}(Y^k-D^k\check\alpha_k)\| \\
& \leq \|\MXkc g^k\| +|\check\alpha_k-\alpha_0|\|\mathcal{P}_{\widehat I^{k^c}}D^k\| + \|\mathcal{P}_{\widehat I^{k^c}}\zeta^k\|.\\
\end{array}
$$
By condition ASTESS $\|\MXkc g^k\|=o_P(n^{1/4})$ and by condition SM(ii) we have $\|\mathcal{P}_{\widehat I^{k^c}}D^k/\sqrt{n_k}\|\leq \|D^k/\sqrt{n_k}\|\lesssim_P 1$, and by Step 1 we have $|\check\alpha_k-\alpha_0|\lesssim_P n^{-1/2}$. Moreover,
$$\begin{array}{rl}
\|\mathcal{P}_{\widehat I^{k^c}}\zeta^k\|& = \| X^k[\widehat I^{k^c}](X^k[\widehat I^{k^c}]'X^k[\widehat I^{k^c}])^{-1}X^k[\widehat I^{k^c}]'\zeta^k \|\\
& \leq [\sqrt{\phi_{\max,k}(\hat s^{k^c})}/\phi_{\min,k}(\hat s^{k^c})]\|X^k[\widehat I^{k^c}]'\zeta^k/\sqrt{n_k}\|.\\
\end{array}
$$
We have $\sqrt{\phi_{\max,k}(\hat s^{k^c})}/\phi_{\min,k}(\hat s^{k^c})\lesssim_P 1$ by condition SE, and $\|X^k[I^{k^c}]'\zeta^k/\sqrt{n_k}\|\lesssim_P \sqrt{\hat s^{k^c}}$ by condition SM(ii), the independence between the selected components $\widehat I^{k^c}$ and $\zeta^k$ since they are based on different subsamples, and applying Chebyshev inequality.

Finally, collecting terms we have
$$ \|g^k-X^k \check \beta_k \|/\sqrt{n^k} \lesssim_P o(n^{-1/4})+ \sqrt{\hat s^{k^c}/n_k} $$

Similarly, we have $ \|m^k-X^k \hat \beta_k \|/\sqrt{n^k} \lesssim_P o(n^{-1/4})+ \sqrt{\hat s^{k^c}/n_k}$.

Step 5.(Variance Estimation.) Since $\hat s^k \lesssim_P s = o(n)$, $(n_k-\hat s^k - 1)/n_k = o_P(1)$, so
we can use $n$ as the denominator. Recall the definitions  $\hat\zeta_i^o = y_i - d_i\check\alpha_k - x_i'\check\beta_k$, $\hat v_i = d_i-x_i'\hat\beta_k$ and $\hat\zeta_i = \hat\zeta_i^o 1\{ |\hat\zeta_i^o|\vee |\hat v_i|\leq H_k\}$ if $i$ belongs to subsample $k$ where $H_k = C\sqrt{n/[(\hat s^{k^c}\vee n^{1/2})\log n]}$. For notational convenience let $A_i = \{  |\hat\zeta_i^o|\vee |\hat v_i|\leq H_k \}$. Since $q>4$, $\hat s^{k^c}\lesssim_P s$, and $n^{2/q}s\log(n\vee p) = o(n)$, we have $n^{1/q} = o_P(H_k)$. Hence consider
$$\begin{array}{rl}
 \En[ \hat v_i^{2} ] & = (n_a/n)D^a{}'\MXb D^a/n_a + (n_b/n)D^b{}'\MXa D^b/n_b\\
 & = (n_a/n) ii_a + (n_b/n) ii_b = V'V/n + o_P(1) = \barEp[v_i^2] + o_P(1)\end{array}$$
by Step 3 and $\barEp[|v_i|^q] \lesssim 1$ for some $q > 4$ by condition SM(ii).


By Condition ASTESS(vi), for each subsample $k=a,b$, we have  $\Enk[ v_i^2 \zeta_i^2 ] - \barEp_k[  v_i^2 \zeta_i^2 ] \to_P 0$ by Vonbahr-Esseen's inequality in \citen{vonbahr:esseen} since $\barEp_k[ | v_i \zeta_i|^{2+\delta}]\leq (\barEp_k[ | v_i|^{4+2\delta}]\barEp_k[ |\ \zeta_i|^{4+2\delta}])^{1/2}$ is uniformly bounded for $4+2\delta\leq q$.
Thus it suffices to show that $\Enk[ \hat v_i^2\hat \zeta_i^2] - \Enk[  v_i^2 \zeta_i^2 ] \to_P 0$. By the triangular inequality
$$\begin{array}{rl}
| \Enk[ \hat v_i^2\hat \zeta_i^2 -  v_i^2 \zeta_i^2 ] | & \leq | \Enk[ (\hat v_i^2\hat \zeta_i^2 - v_i^2\zeta_i^2)1\{A_i\} ] | + | \Enk[ (\hat v_i^2\hat \zeta_i^2 -  v_i^2 \zeta_i^2)1\{A_i^c\} ] | \\
& \leq | \Enk[ (\hat v_i^2 -  v_i^2) \zeta_i^2 1\{A_i\}] | + | \Enk[   v_i^2(\hat \zeta_i^2 -  \zeta_i^2) 1\{A_i\}]|+\\
& + | \Enk[ (\hat v_i^2 -  v_i^2)(\hat \zeta_i^2 -  \zeta_i^2)1\{A_i\} ]|+o_P(1)\end{array}$$
since $| \Enk[ (\hat v_i^2\hat \zeta_i^2 -  v_i^2 \zeta_i^2)1\{A_i^c\} ] | = o_P(1)$ by Step 6.
Then,
$$\begin{array}{rl}
|\Enk[  v_i^2(\hat \zeta_i^2 - \zeta_i^2)1\{A_i\} ]| & \leq \underset{=:iii_1 } {2\Enk[\{d_i(\alpha_0-\check\alpha_k)\}^2  v_i^2]}+\underset{=:iii_2 } {2\Enk[\{x_i'\check\beta_k - g_i\}^2  v_i^2]}\\
& +\underset{=:iii_3 } {2\max_{i\leq n}|v_i|\{\Enk[ \zeta_i^2 v_i^2]\}^{1/2}\{\Enk[ d_i^2(\alpha_0-\check\alpha_k)^2]\}^{1/2}}\\
& + \underset{=:iii_4 } {2\max_{i\leq n}|v_i|\{\Enk[ \zeta_i^2 v_i^2]\}^{1/2}\{\Enk[ (g_i-x_i'\check\beta_k)^2]\}^{1/2}}\\
\end{array}$$

As a consequence of Condition SM(ii) we have $\Ep[\max_{i\leq n}d_i^2] \lesssim n^{2/q}$, $\Ep[\max_{i\leq n} \zeta_i^2]\lesssim n^{2/q}$, $\Ep[\max_{i\leq n} v_i^2] \lesssim n^{2/q}$, thus by Markov inequality we have $\max_{i\leq n} |d_i| + |\zeta_i| + | v_i|\lesssim_P n^{1/q}$.

We have the following relations:
$$
\begin{array}{rl}
iii_1 & \leq |\alpha_0-\check\alpha_k|^2\Enk[d_i^2] \max_{i\leq n} v_i^2\lesssim_P n^{-1} n^{2/q} = o_P(1),\\
iii_2 &  \leq \max_{i\leq n}  v_i^2\Enk[\{x_i'\check\beta_k - g_i\}^2] \lesssim_P n^{2/q} \{o(n^{-1/4})+\sqrt{\hat s^{k^c}/n}\}^2=o_P(1),\\
iii_3 &  \lesssim_P n^{1/q}\sqrt{1/n} = o_P(1), \\
iii_4& \lesssim_P n^{1/q} \{o(n^{-1/4})+\sqrt{\hat s^{k^c}/n}\} = o_P(1), \\
\end{array}
$$ since $\Enk[ \zeta_i^2 v_i^2] \lesssim_P 1$, $\Enk[\{x_i'\check\beta_k-g_i\}^2]\lesssim_P \{o(n^{-1/4})+\sqrt{\hat s^{k^c}/n}\}^2$ by  Step 4, $\hat s^{k^c}\lesssim_P s$, and $|\check\alpha_k - \alpha_0|^2\lesssim_P 1/n$ by Step 1.

Similarly, $\Enk[ (\hat v_i^2 -  v_i^2) \zeta_i^2 ] =o_P(1)$.
{}

Finally, since $\max_{i\leq n} \|1\{ A_i\}( \hat v_i, \hat \zeta_i, \zeta_i, v_i)'\|_\infty^2\lesssim_P ( H_k^2\vee n^{2/q})\lesssim_P H_k^2$,  we have
$$\begin{array}{rl}
| \Enk[ (\hat v_i^2   -  v_i^2)(\hat \zeta_i^2 -  \zeta_i^2)1\{ A_i\} ]| & \leq  \{\Enk[ (\hat v_i^2   - v_i^2)^21\{A_i\}] \Enk[ (\hat \zeta_i^2   - \zeta_i^2)^21\{A_i\}]\}^{1/2} \\
& \leq  \{\Enk[ 2(\hat v_i^2 + v_i^2)(\hat v_i   - v_i)^21\{ A_i\}] \Enk[ 2(\hat\zeta_i^2 + \zeta_i^2) (\hat \zeta_i   - \zeta_i)^21\{ A_i\}]\}^{1/2} \\
& \lesssim_P ( H_k^2\vee n^{2/q}) \{\Enk[ (\hat v_i   - v_i)^2] \Enk[(\hat \zeta_i   -  \zeta_i)^2]\}^{1/2} \\
& \lesssim_P  H_k^2 \{ o(n^{-1/4}) + \sqrt{\hat s^{k^c}/n}\}^2 \\
&\lesssim \frac{n}{(\hat s^{k^c}\vee n^{1/2})\log n}\{ o(n^{-1/2}) + \hat s^{k^c}/n\}=o(1).\\
\end{array}$$

Step 6.(Controlling large terms) By definition of the event $A_i$ we have
$$ \begin{array}{rl}
H_k^2\Enk[1\{A_i^c\}]& \leq \Enk[\hat\zeta_i^{o2}1\{A_i^c\} ] \\
& \leq 4\Enk[\zeta_i^21\{A_i^c\}] + 4\Enk[d_i^2(\check\alpha_k-\alpha_0)^21\{A_i^c\}] +4\Enk[\{x_i'\check\beta_k-g_i\}^21\{A_i^c\}] \\
&\lesssim_Pn^{2/q}\Enk[1\{A_i^c\}] + n^{2/q-1}\Enk[1\{A_i^c\}]+\Enk[\{x_i'\check\beta_k-g_i\}^2]. \end{array}
$$
Since $n^{1/q} = o_P(H_k)$, and $\Enk[\{x_i'\check\beta_k-g_i\}^2]\lesssim_P o(n^{-1/2}) + \hat s^{k^c}/n$, we have
$$ \Enk[1\{A_i^c\}] \lesssim_P \{o(n^{-1/2}) + \hat s^{k^c}/n\}/H_k^2.$$

Therefore,
$$ \Enk[ \tilde \zeta_i^2\tilde v_i^2 1\{A_i^c\}] \lesssim_P n^{4/q}\Enk[1\{A_i^c\}] \lesssim_P n^{4/q}\{o(n^{-1/2}) + \hat s^{k^c}/n\}/H_k^2.$$

Finally note that
$$\begin{array}{rl}\frac{n^{4/q}\{o(n^{-1/2}) + \hat s^{k^c}/n\}}{H_k^2} & \lesssim \frac{n^{2/q}}{n^{1/2}}\frac{n^{2/q} (\hat s^{k^c}\vee n^{1/2})\log n}{n} +  \frac{n^{2/q} \hat s^{k^c}\log n}{n}\frac{n^{2/q}(\hat s^{k^c}\vee n^{1/2})}{n} = o_P(1)\end{array} $$
since $q > 4$, $\hat s^{k^c}\lesssim_P s$, and $n^{2/q}s\log(n\vee p) = o(n)$ by ASTESS.
Also, by construction, we have $ \Enk[ \hat \zeta_i^2\hat v_i^2 1\{A_i^c\}]=0$.

\end{proof}

\section{Additional Simulation Results}

In this section, we present additional simulation results.  All of the simulation results are based on the structural model
\begin{equation}\label{ModelMCPLMy}
y_i = d_i'\alpha_0 +  x_i'(c_y \beta_0) + \sigma_{y}(d_i,x_i) \zeta_i, \ \  \zeta_i \sim N(0,1)
\end{equation}
where $p = \dim(x_i) = 200$, the covariates $ x \sim N(0,\Sigma)$ with $\Sigma_{kj} = (0.5)^{|j-k|}$, $\alpha_0 = .5$, and
the sample size $n$ is set to $100$.  In each design, we generate
\begin{equation}\label{ModelMCPLMd}
d_i^* =  x_i'(c_d \beta_1)  + \sigma_{d}(x_i) v_i, \ \  v_i \sim N(0,1)
\end{equation}
with E[$\zeta_i v_i] = 0 $.
Inference results for all designs are based on conventional t-tests with standard errors calculated using the heteroscedasticity consistent jackknife variance estimator discussed in \citen{mackinnon:white}.
We set  $\lambda$ according to the algorithm outlined in Appendix A with $1-\conflvl = .95$.  We draw new $x$'s, $\zeta$'s and $v$'s at every replication and draw new $\beta_0$'s and $\beta_1$'s at every replication in the random coefficient designs.

In the first thirteen designs, $\beta_1 = \beta_0$.  We set the constants $c_y$ and $c_d$ to generate desired population values for the reduced form $R^2$'s, i.e. the $R^2$'s for equations  (\ref{ModelMCPLMy}) and (\ref{ModelMCPLMd}).
Let $R^2_y$ be the desired $R^2$ for the regression of $y$ on $x$ and $R^2_d$ be the desired $R^2$ from the regression of $d$ on $x$.  For each equation, we choose $c_y$ and $c_d$ to generate $R^2 = 0, .2, .4, .6,$ and $.8$.  In the heteroscedastic and binary designs discussed below, we choose $c_y$ and $c_d$ based on $R^2$ as if (\ref{ModelMCPLMy}) held with $d_i = d_i^*$ and $v_i$ and $\zeta_i$ were homoscedastic with variance equal to the average variance and label the results by $R^2$ as in the other cases.  In the homoscedastic cases, we set $\sigma_y = \sigma_d = 1$; and in the heteroscedastic cases, the average of $\sigma_d(x_i)$ and the average of $\sigma_y(d_i,x_i)$ are both one.  We set
\begin{align*}
c_d &= \sqrt{\frac{R^2_d}{(1-R^2_d)\beta_0'\Sigma\beta_0}} \\
c_y &= \frac{-(1-R^2_y)\alpha_0 c_d \beta_0'\Sigma\beta_0 + \sqrt{(1-R^2_y)R^2_y \beta_0'\Sigma\beta_0 (\alpha_0^2+1)}}{(1-R^2_y)\beta_0'\Sigma\beta_0}
\end{align*}

\begin{itemize}
\item Design 1. $d_i = d_i^*$, $\beta_{0} = (1,1/2,1/3,1/4,1/5,0,0,0,0,0,1,1/2,1/3,1/4,1/5,0,...,0)'$, $\sigma_y = \sigma_d = 1$.
\item Design 2. $d_i = d_i^*$, $\beta_{0} = (1,1/4,1/9,1/16,1/25,0,0,0,0,0,1,1/4,1/9,1/16,1/25,0,...,0)'$, $\sigma_y = \sigma_d = 1$.
\item Design 22. $d_i = d_i^*$, $\beta_{0,j} = (1/j)^2$, $\sigma_y = \sigma_d = 1$.
\item Design 3.  $d_i = d_i^*$, $\beta_{0} = (1,1/2,1/3,1/4,1/5,0,0,0,0,0,1,1/2,1/3,1/4,1/5,0,...,0)'$, $\sigma_d = \sqrt{\frac{(1+x_i'\beta_0)^2}{\frac{1}{n}\sum_{i=1}^{n}(1+x_i'\beta_0)^2}}$, $\sigma_y = \sqrt{\frac{(1+\alpha_0 d_i + x_i'\beta_0)^2}{\frac{1}{n}\sum_{i=1}^{n}(1+\alpha_0 d_i+x_i'\beta_0)^2}}$.
\item Design 4.  $d_i = d_i^*$, $\beta_{0} = (1,1/4,1/9,1/16,1/25,0,0,0,0,0,1,1/4,1/9,1/16,1/25,0,...,0)'$, $\sigma_d = \sqrt{\frac{(1+x_i'\beta_0)^2}{\frac{1}{n}\sum_{i=1}^{n}(1+x_i'\beta_0)^2}}$, $\sigma_y = \sqrt{\frac{(1+\alpha_0 d_i + x_i'\beta_0)^2}{\frac{1}{n}\sum_{i=1}^{n}(1+\alpha_0 d_i+x_i'\beta_0)^2}}$.
\item Design 44. $d_i = d_i^*$, $\beta_{0,j} = (1/j)^2$, $\sigma_d = \sqrt{\frac{(1+x_i'\beta_0)^2}{\frac{1}{n}\sum_{i=1}^{n}(1+x_i'\beta_0)^2}}$, $\sigma_y = \sqrt{\frac{(1+\alpha_0 d_i + x_i'\beta_0)^2}{\frac{1}{n}\sum_{i=1}^{n}(1+\alpha_0 d_i+x_i'\beta_0)^2}}$.
\item Design 5. $d_i = \mathbf{1}\{d_i^* > 0\}$, $\beta_{0} = (1,1/2,1/3,1/4,1/5,0,0,0,0,0,1,1/2,1/3,1/4,1/5,0,...,0)'$, $\sigma_y = \sigma_d = 1$.
\item Design 6. $d_i = d_i^*$, $\beta_{0,j} \sim N(0,1)$, $\sigma_y = \sigma_d = 1$.
\item Design 7. $d_i = d_i^*$, $\widetilde\beta_{0} = (1,1/2,1/3,1/4,1/5,0,0,0,0,0,1,1/2,1/3,1/4,1/5,0,...,0)'$, $\beta_{0,j} \sim N(0,\widetilde\beta_{0,j}^2)$, $\sigma_y = \sigma_d = 1$.
\item Design 72. $d_i = d_i^*$, $\widetilde\beta_{0} = (1,1/4,1/9,1/16,1/25,0,0,0,0,0,1,1/4,1/9,1/16,1/25,0,...,0)'$, $\beta_{0,j} \sim N(0,\widetilde\beta_{0,j}^2)$, $\sigma_y = \sigma_d = 1$.
\item Design 722. $d_i = d_i^*$, $\widetilde\beta_{0,j} = (1/j)^2$, $\beta_{0,j} \sim N(0,\widetilde\beta_{0,j}^2)$, $\sigma_y = \sigma_d = 1$.
\item Design 8.  $d_i = d_i^*$, $\widetilde\beta_{0,j} = u_j z_{1,j} + (1-u_j) z_{2,j}$, $u_j \sim \textnormal{Bernoulli}(.05)$, $z_{1,j} \sim N(0,25)$, $z_{2,j} \sim N(0,.0025)$, $\sigma_y = \sigma_d = 1$
\item Design 1001. $d_i = d_i^*$, $\beta_{0,j} = \mathbf{1}\{j \in \{2,4,6,...,38,40\}\}$, $\sigma_y = \sigma_d = 1$.
\end{itemize}

In the last thirteen designs, we set the constants $c_y$ and $c_d$ according to
\begin{align*}
c_d &= \sqrt{\frac{R^2_d}{(1-R^2_d)\beta_1'\Sigma\beta_1}} \\
c_y &= \sqrt{\frac{R^2_d}{(1-R^2_d)\beta_0'\Sigma\beta_0}}
\end{align*}
for $R^2_d = 0, .2, .4, .6,$ and $.8$ and $R^2_y = 0, .2, .4, .6,$ and $.8$.

\begin{itemize}
\item Design 1a. $d_i = d_i^*$, $\beta_{0} = (1,1/2,1/3,1/4,1/5,0,0,0,0,0,1,1/2,1/3,1/4,1/5,0,...,0)'$, $\beta_1 = (1,1/2,1/3,1/4,1/5,1/6,1/7,1/8,1/9,1/10,0,...,0)'$, $\sigma_y = \sigma_d = 1$.
\item Design 2a. $d_i = d_i^*$, $\beta_{0} = (1,1/4,1/9,1/16,1/25,0,0,0,0,0,1,1/4,1/9,1/16,1/25,0,...,0)'$, $\beta_1 = (1,1/4,1/9,1/16,1/25,1/36,1/49,1/64,1/81,1/100,0,...,0)'$, $\sigma_y = \sigma_d = 1$.
\item Design 22a. $d_i = d_i^*$, $\beta_{0,j} = (1/j)^2$, $\beta_{1,j} = (1/j)^2$, $\sigma_y = \sigma_d = 1$.
\item Design 3a.  $d_i = d_i^*$, $\beta_{0} = (1,1/2,1/3,1/4,1/5,0,0,0,0,0,1,1/2,1/3,1/4,1/5,0,...,0)'$, $\beta_1 = (1,1/2,1/3,1/4,1/5,1/6,1/7,1/8,1/9,1/10,0,...,0)'$, $\sigma_d = \sqrt{\frac{(1+x_i'\beta_1)^2}{\frac{1}{n}\sum_{i=1}^{n}(1+x_i'\beta_1)^2}}$, $\sigma_y = \sqrt{\frac{(1+\alpha_0 d_i + x_i'\beta_0)^2}{\frac{1}{n}\sum_{i=1}^{n}(1+\alpha_0 d_i+x_i'\beta_0)^2}}$.
\item Design 4a.  $d_i = d_i^*$, $\beta_{0} = (1,1/4,1/9,1/16,1/25,0,0,0,0,0,1,1/4,1/9,1/16,1/25,0,...,0)'$, $\beta_1 = (1,1/4,1/9,1/16,1/25,1/36,1/49,1/64,1/81,1/100,0,...,0)'$, $\sigma_d = \sqrt{\frac{(1+x_i'\beta_1)^2}{\frac{1}{n}\sum_{i=1}^{n}(1+x_i'\beta_1)^2}}$, $\sigma_y = \sqrt{\frac{(1+\alpha_0 d_i + x_i'\beta_0)^2}{\frac{1}{n}\sum_{i=1}^{n}(1+\alpha_0 d_i+x_i'\beta_0)^2}}$.
\item Design 44a. $d_i = d_i^*$, $\beta_{0,j} = (1/j)^2$, $\beta_{1,j} = (1/j)^2$, $\sigma_d = \sqrt{\frac{(1+x_i'\beta_1)^2}{\frac{1}{n}\sum_{i=1}^{n}(1+x_i'\beta_1)^2}}$, $\sigma_y = \sqrt{\frac{(1+\alpha_0 d_i + x_i'\beta_0)^2}{\frac{1}{n}\sum_{i=1}^{n}(1+\alpha_0 d_i+x_i'\beta_0)^2}}$.
\item Design 5a. $d_i = \mathbf{1}\{d_i^* > 0\}$, $\beta_{0} = (1,1/2,1/3,1/4,1/5,0,0,0,0,0,1,1/2,1/3,1/4,1/5,0,...,0)'$, $\beta_1 = (1,1/2,1/3,1/4,1/5,1/6,1/7,1/8,1/9,1/10,0,...,0)'$, $\sigma_y = \sigma_d = 1$.
\item Design 6a. $d_i = d_i^*$, $\beta_{0,j} \sim N(0,1)$, $\beta_{1,j} \sim N(0,1)$, E$[\beta_{0,j}\beta_{1,j}] = .8$, $\sigma_y = \sigma_d = 1$.
\item Design 7a. $d_i = d_i^*$, $\widetilde\beta_{0} = (1,1/2,1/3,1/4,1/5,0,0,0,0,0,1,1/2,1/3,1/4,1/5,0,...,0)'$, $\widetilde\beta_1 = (1,1/2,1/3,1/4,1/5,1/6,1/7,1/8,1/9,1/10,0,...,0)'$, $\beta_{0,j} = \widetilde\beta_{0,j} z_{0,j}$, $\beta_{1,j} = \widetilde\beta_{1,j} z_{1,j}$, $z_{0,j} \sim N(0,1)$, $z_{1,j} \sim N(0,1)$, E$[z_{0,j} z_{1,j}] = .8$, $\sigma_y = \sigma_d = 1$.
\item Design 72a. $d_i = d_i^*$, $\widetilde\beta_{0} = (1,1/4,1/9,1/16,1/25,0,0,0,0,0,1,1/4,1/9,1/16,1/25,0,...,0)'$, $\widetilde\beta_1 = (1,1/4,1/9,1/16,1/25,1/36,1/49,1/64,1/81,1/100,0,...,0)'$, $\beta_{0,j} = \widetilde\beta_{0,j} z_{0,j}$, $\beta_{1,j} = \widetilde\beta_{1,j} z_{1,j}$, $z_{0,j} \sim N(0,1)$, $z_{1,j} \sim N(0,1)$, E$[z_{0,j} z_{1,j}] = .8$, $\sigma_y = \sigma_d = 1$.
\item Design 722a. $d_i = d_i^*$, $\widetilde\beta_{0,j} = (1/j)^2$, $\widetilde\beta_{1,j} = (1/j)^2$, $\beta_{0,j} = \widetilde\beta_{0,j} z_{0,j}$, $\beta_{1,j} = \widetilde\beta_{1,j} z_{1,j}$, $z_{0,j} \sim N(0,1)$, $z_{1,j} \sim N(0,1)$, E$[z_{0,j} z_{1,j}] = .8$, $\sigma_y = \sigma_d = 1$.
\item Design 8a.  $d_i = d_i^*$, $\widetilde\beta_{0,j} = 5 u_j z_{11,j} + .05 (1-u_j) z_{12,j}$, $\widetilde\beta_{1,j} = 5 u_j z_{21,j} + .05 (1-u_j) z_{22,j}$, $u_j \sim \textnormal{Bernoulli}(.05)$, $z_{11,j} \sim N(0,1)$, $z_{12,j} \sim N(0,1)$, $z_{21,j} \sim N(0,1)$, $z_{22,j} \sim N(0,1)$, $\sigma_y = \sigma_d = 1$
\item Design 1001a. $d_i = d_i^*$, $\beta_{0,j} = \mathbf{1}\{j \in \{2,4,6,...,38,40\}\}$, $\beta_{1,j} = \mathbf{1}\{j \in \{1,3,5,...,37,39\}\}$, $\sigma_y = \sigma_d = 1$.
\end{itemize}

Results are summarized in figures and tables below.  In the tables, we report results for the four estimators considered in the main text (Oracle, Double-Selection Oracle, Post-Lasso, and Double-Selection).  We also report results for regular Lasso (Lasso), the union of the Double-Selection interval with the Post-Lasso interval (Double-Selection Union ADS), using the union of the set of variables selected by Double-Selection and the set of variables selected by running Lasso of $y$ on $d$ and $x$ without penalizing $d$ (Double-Selection + I3), and the split-sample procedure discussed in the text (Split-Sample).  For Double-Selection Union ADS, the point estimate is taken as the midpoint of the union of the intervals.

\pagebreak

\begin{figure}
\includegraphics[width=\textwidth]{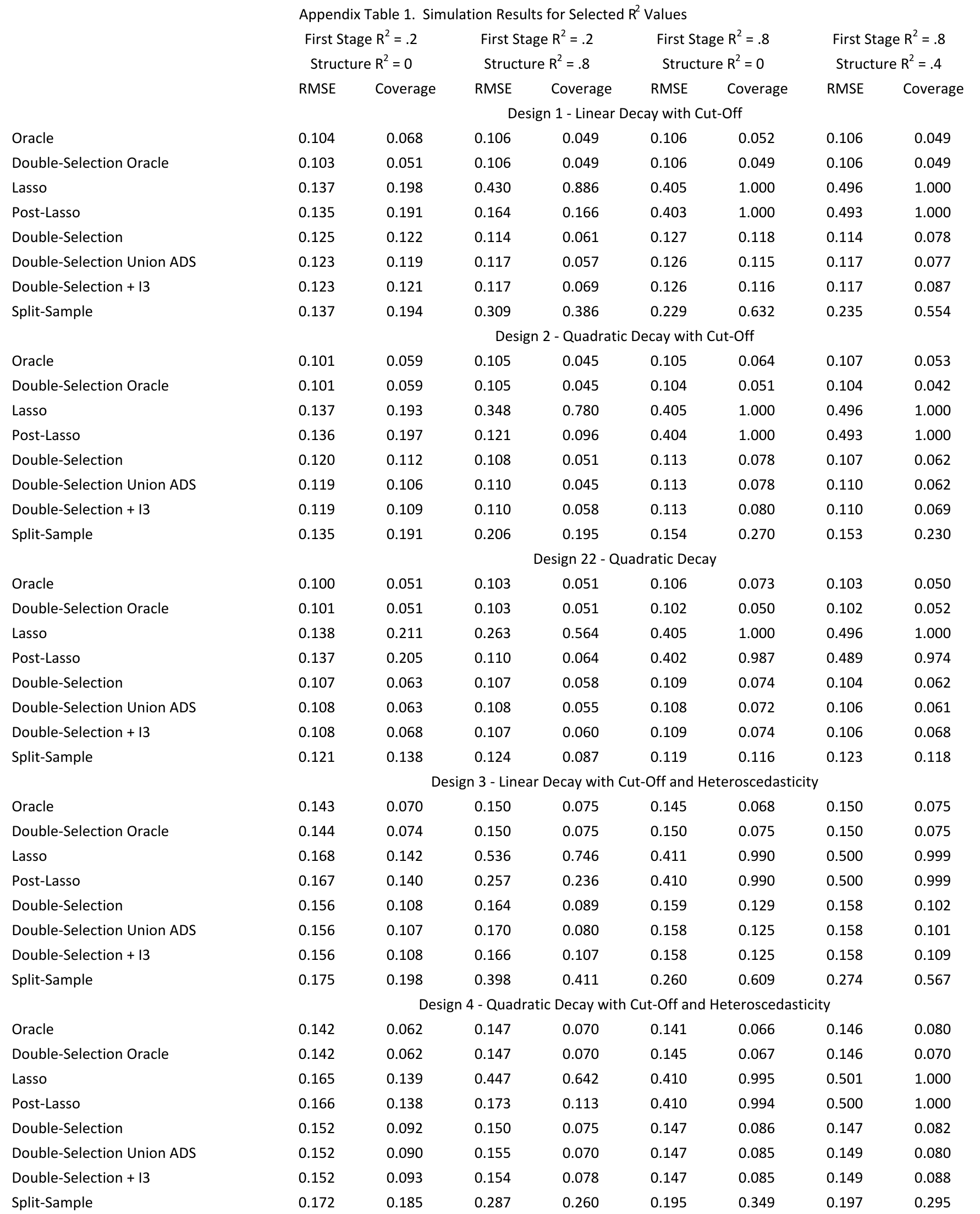}
	\label{fig:Table1p1}
\end{figure}

\pagebreak

\begin{figure}
\includegraphics[width=\textwidth]{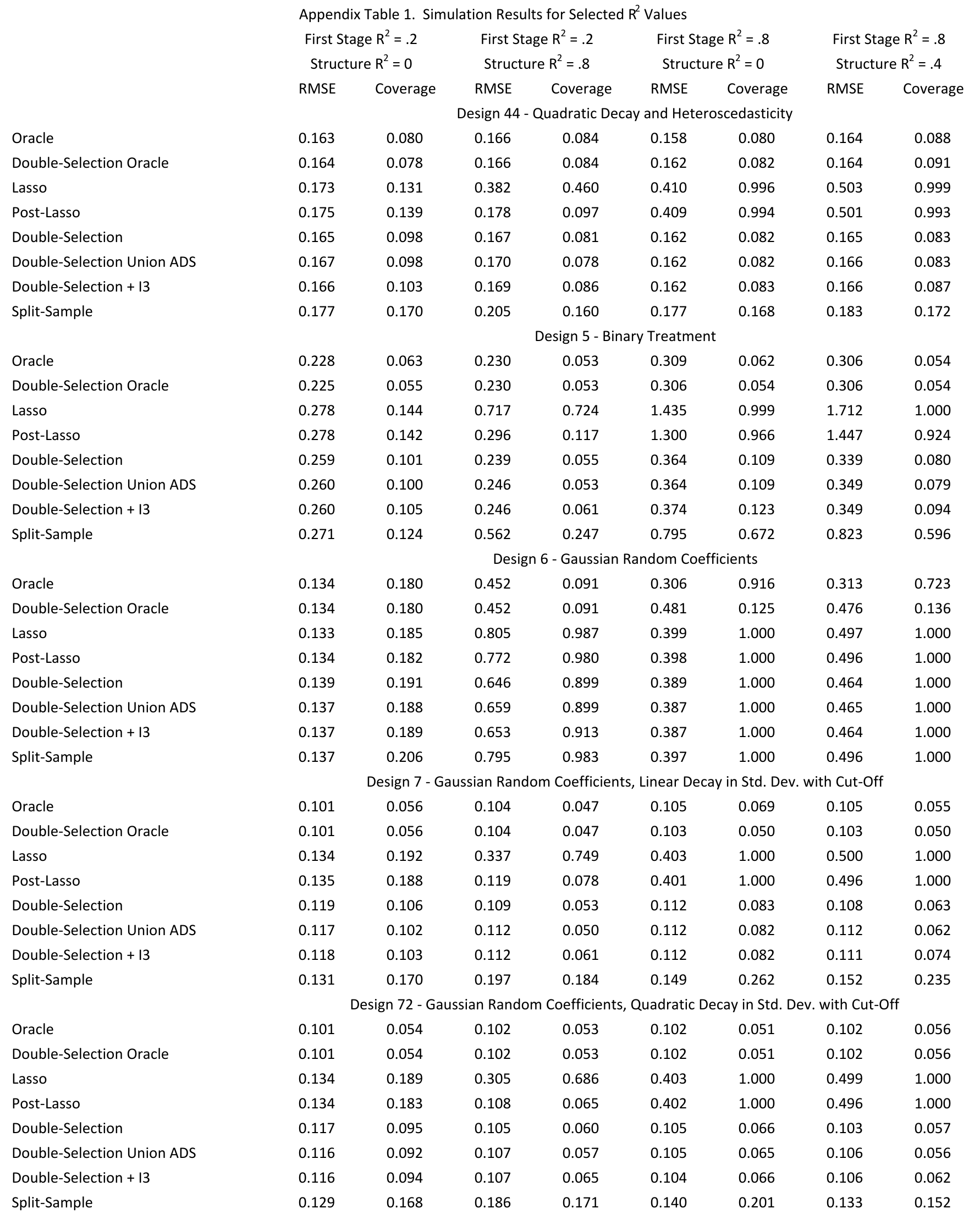}
	\label{fig:Table1p2}
\end{figure}

\pagebreak

\begin{figure}
\includegraphics[width=\textwidth]{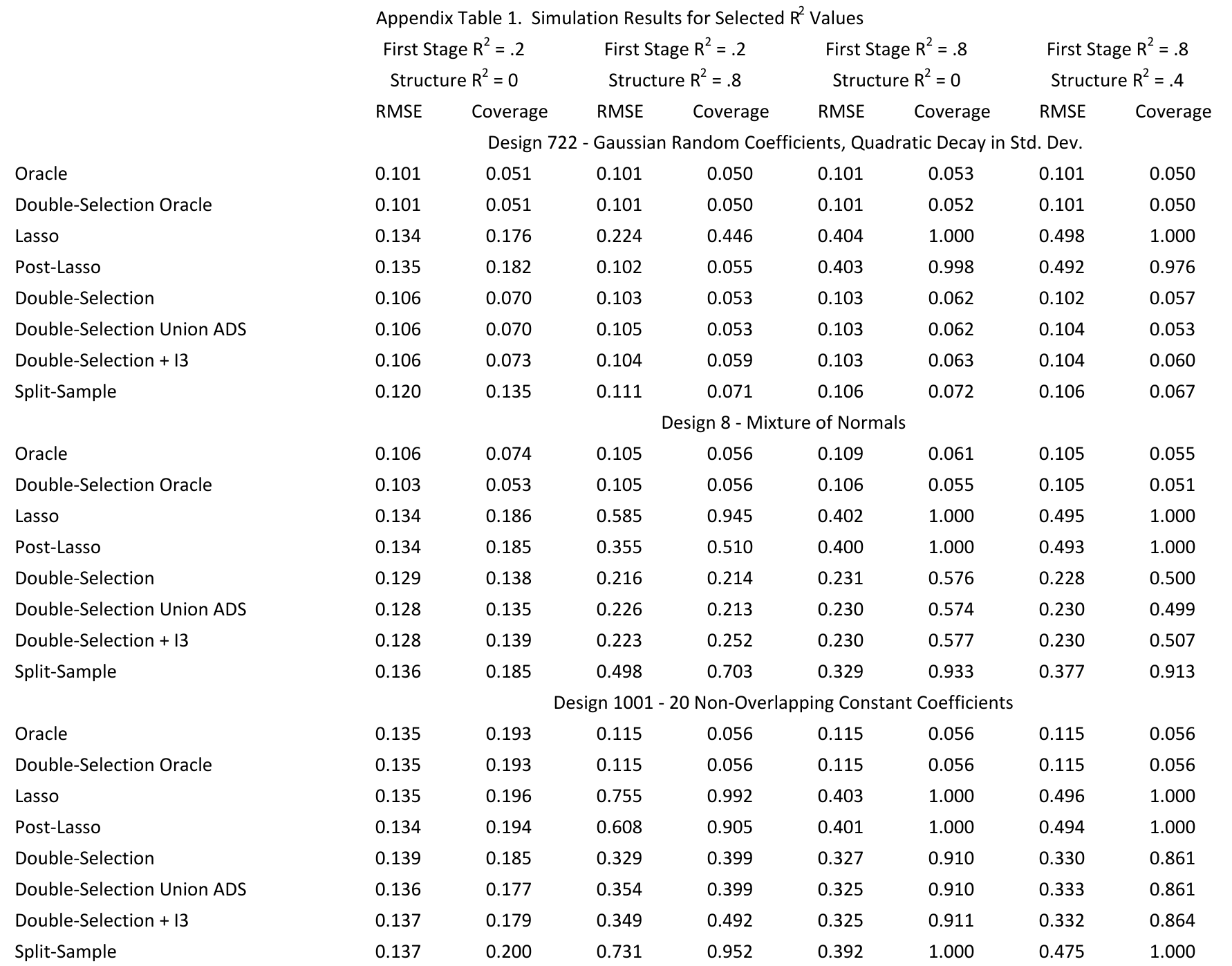}
	\label{fig:Table1p3}
\end{figure}

\pagebreak

\begin{figure}
\includegraphics[width=\textwidth]{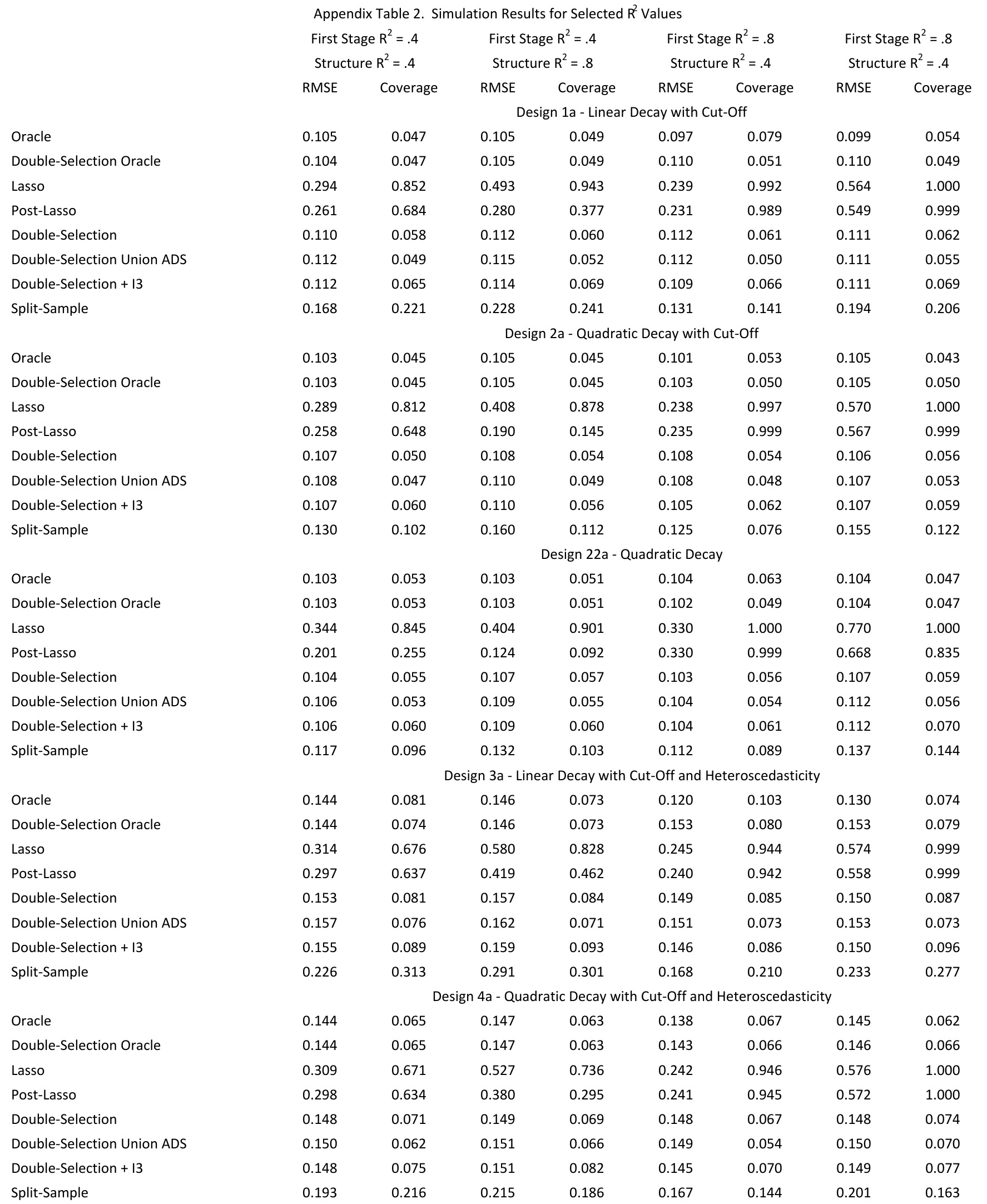}
	\label{fig:Table2p1}
\end{figure}

\pagebreak

\begin{figure}
\includegraphics[width=\textwidth]{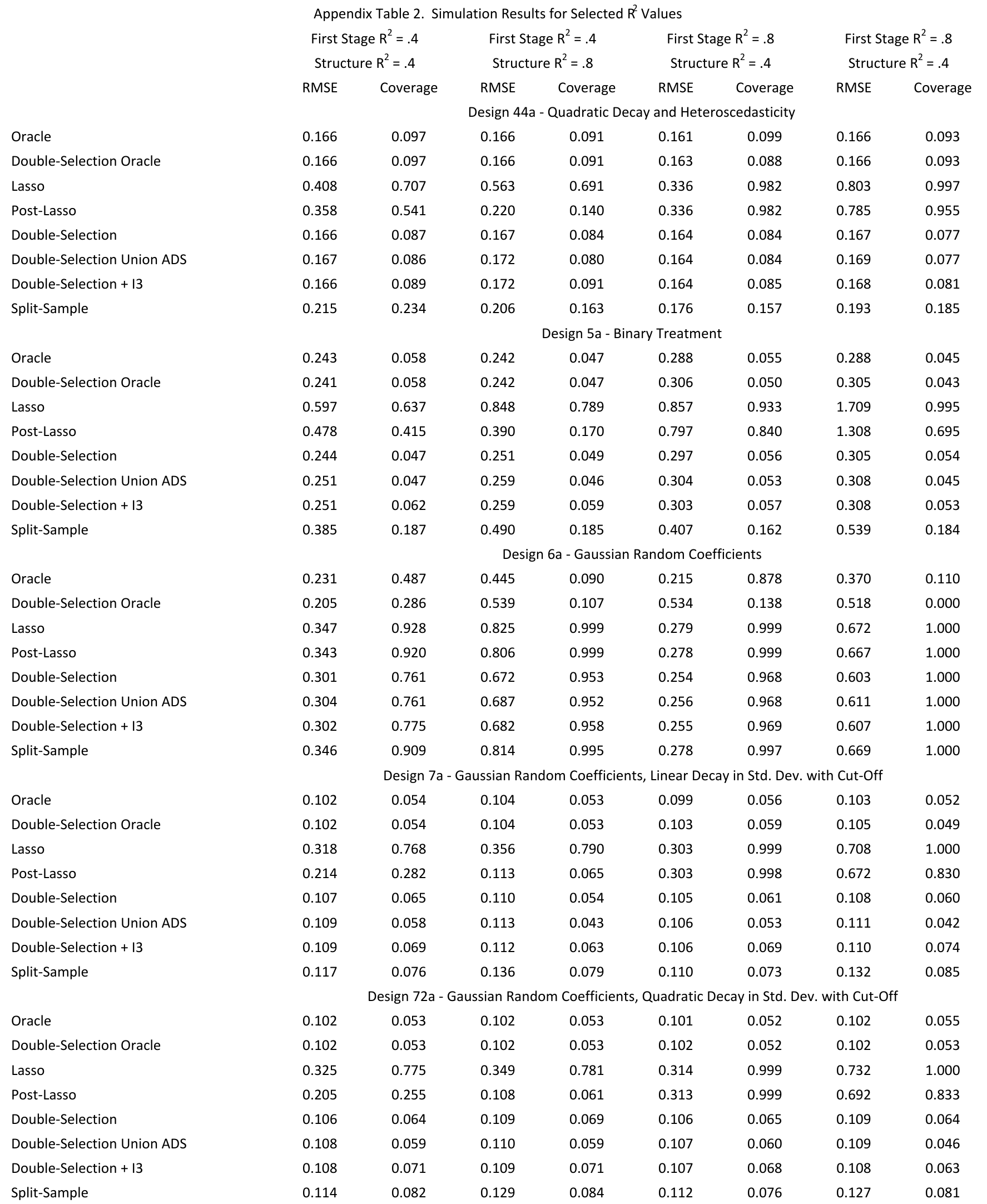}
	\label{fig:Table2p2}
\end{figure}

\pagebreak

\begin{figure}
\includegraphics[width=\textwidth]{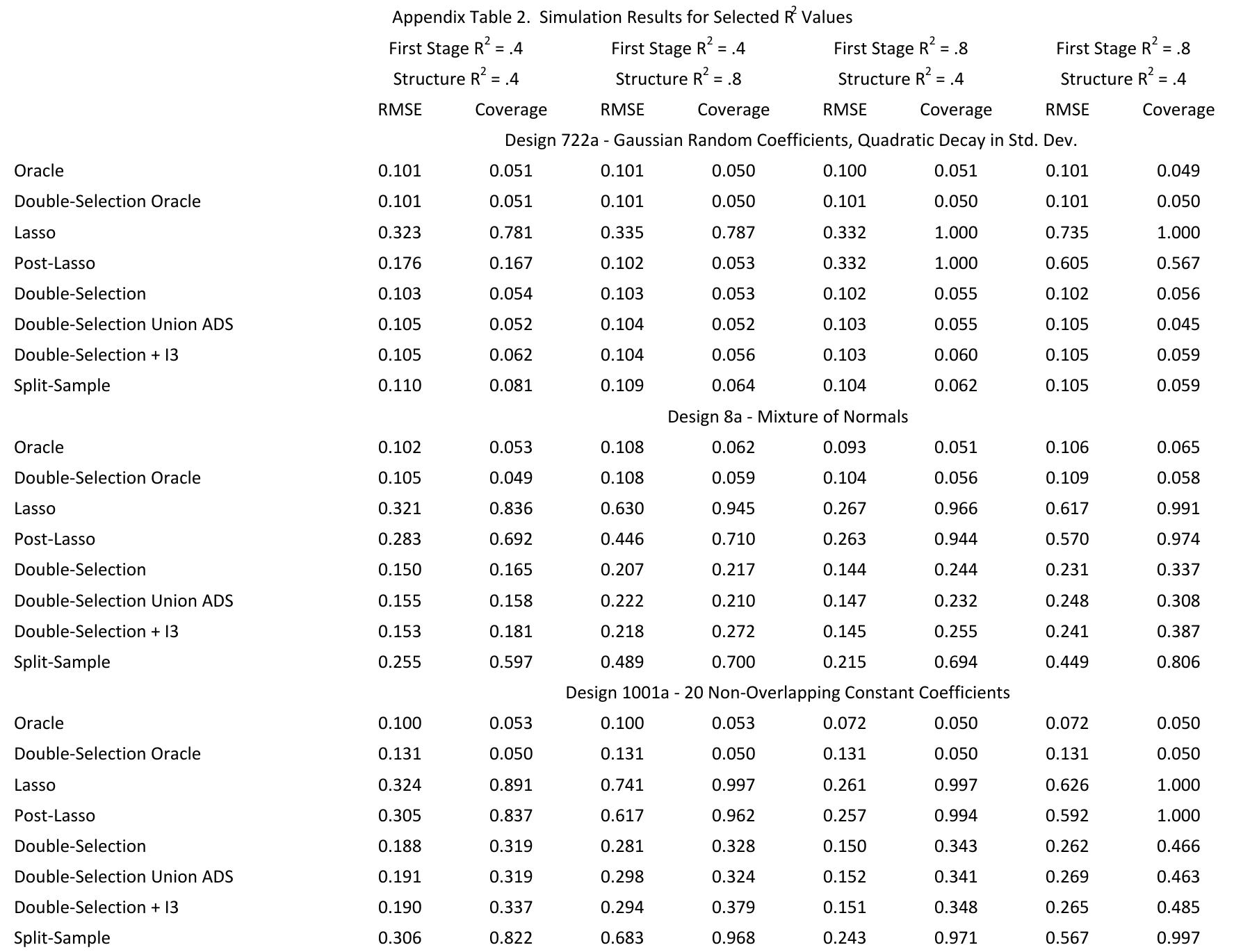}
	\label{fig:Table2p3}
\end{figure}

\pagebreak

\begin{figure}
\includegraphics[width=\textwidth]{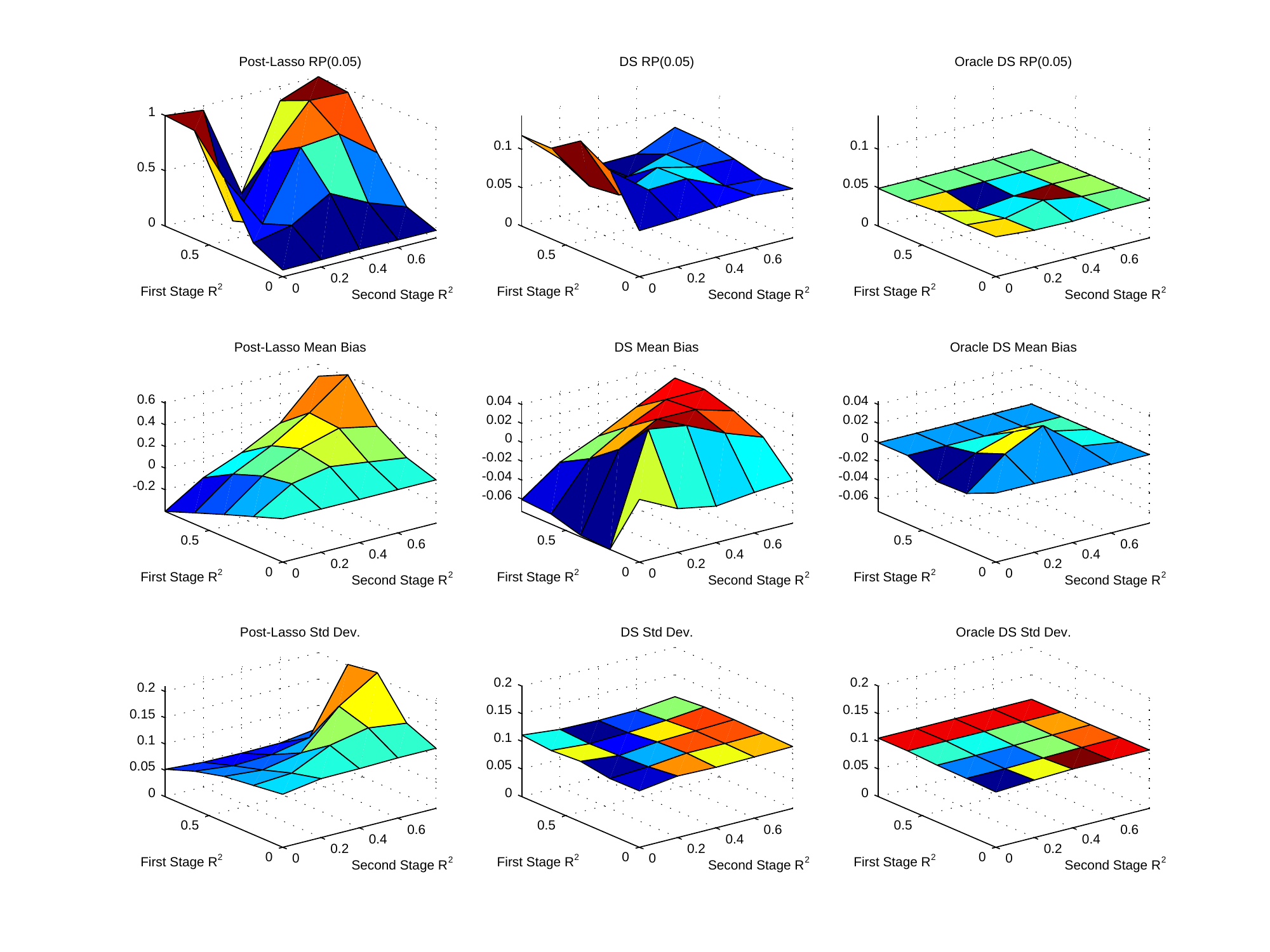}
	\label{fig:figure1}
\caption{Design 1}
\end{figure}

\pagebreak

\begin{figure}
\includegraphics[width=\textwidth]{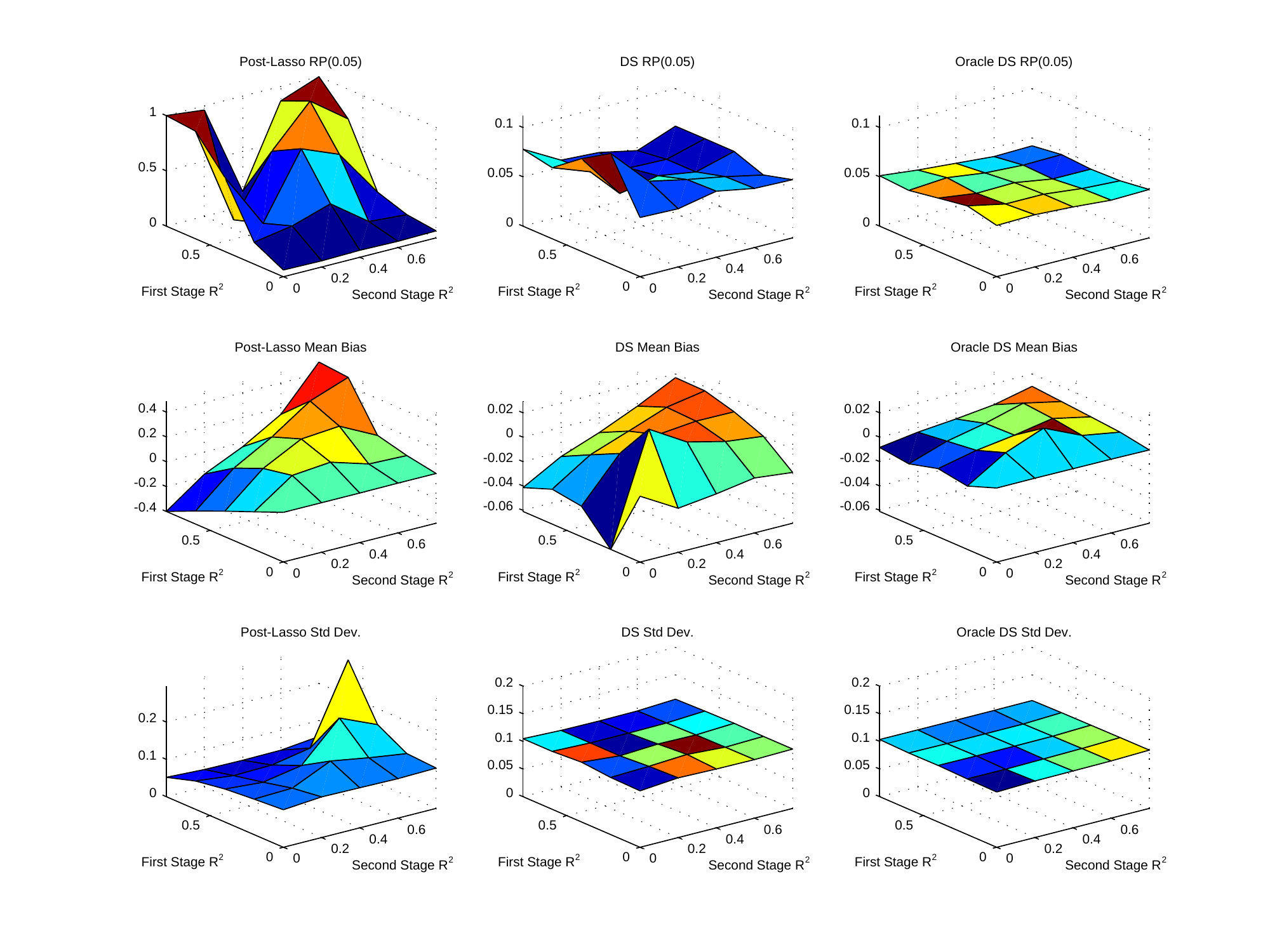}
	\label{fig:figure2}
\caption{Design 2}
\end{figure}

\pagebreak

\begin{figure}
\includegraphics[width=\textwidth]{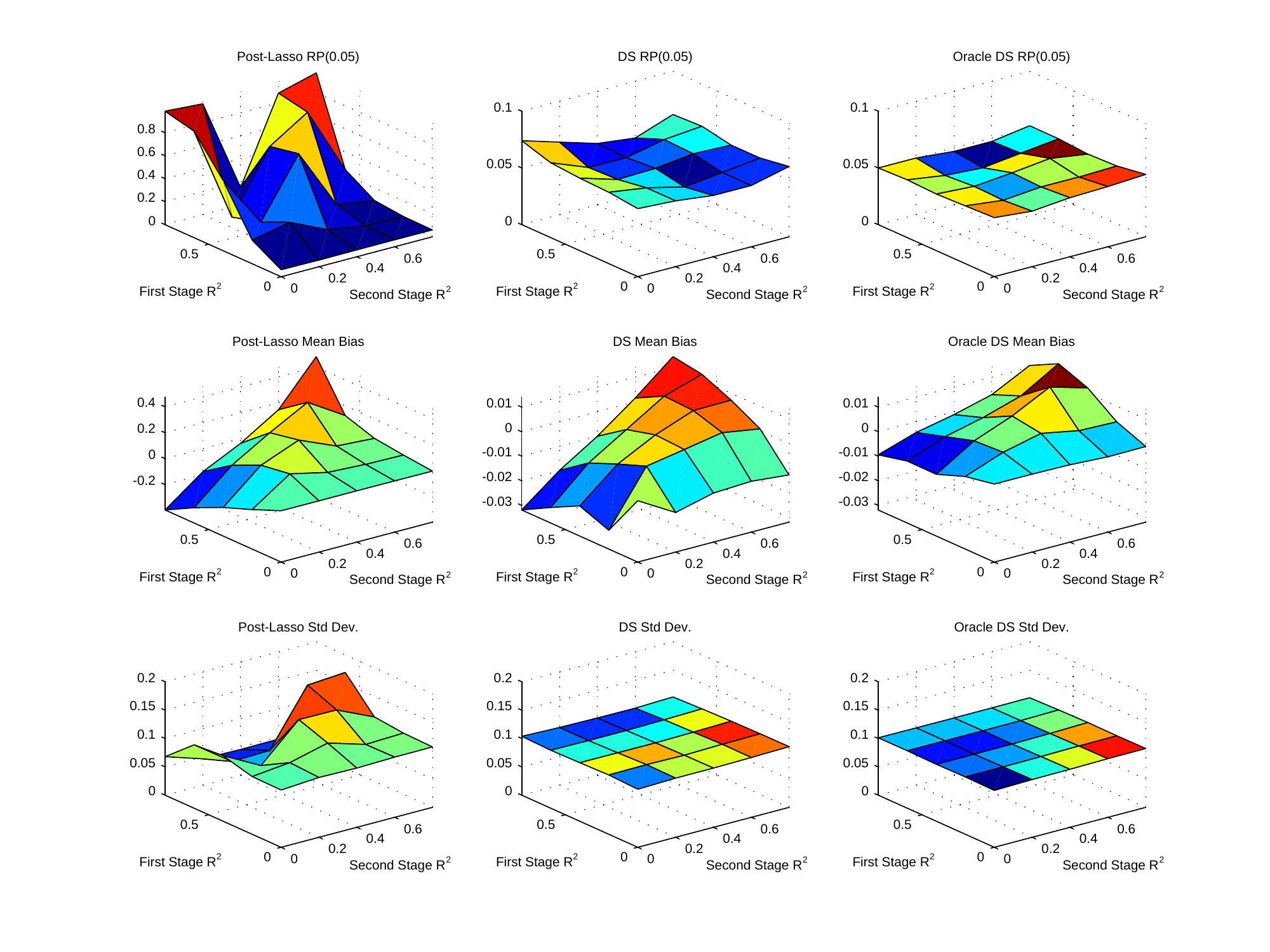}
	\label{fig:figure3}
\caption{Design 22}
\end{figure}

\pagebreak

\begin{figure}
\includegraphics[width=\textwidth]{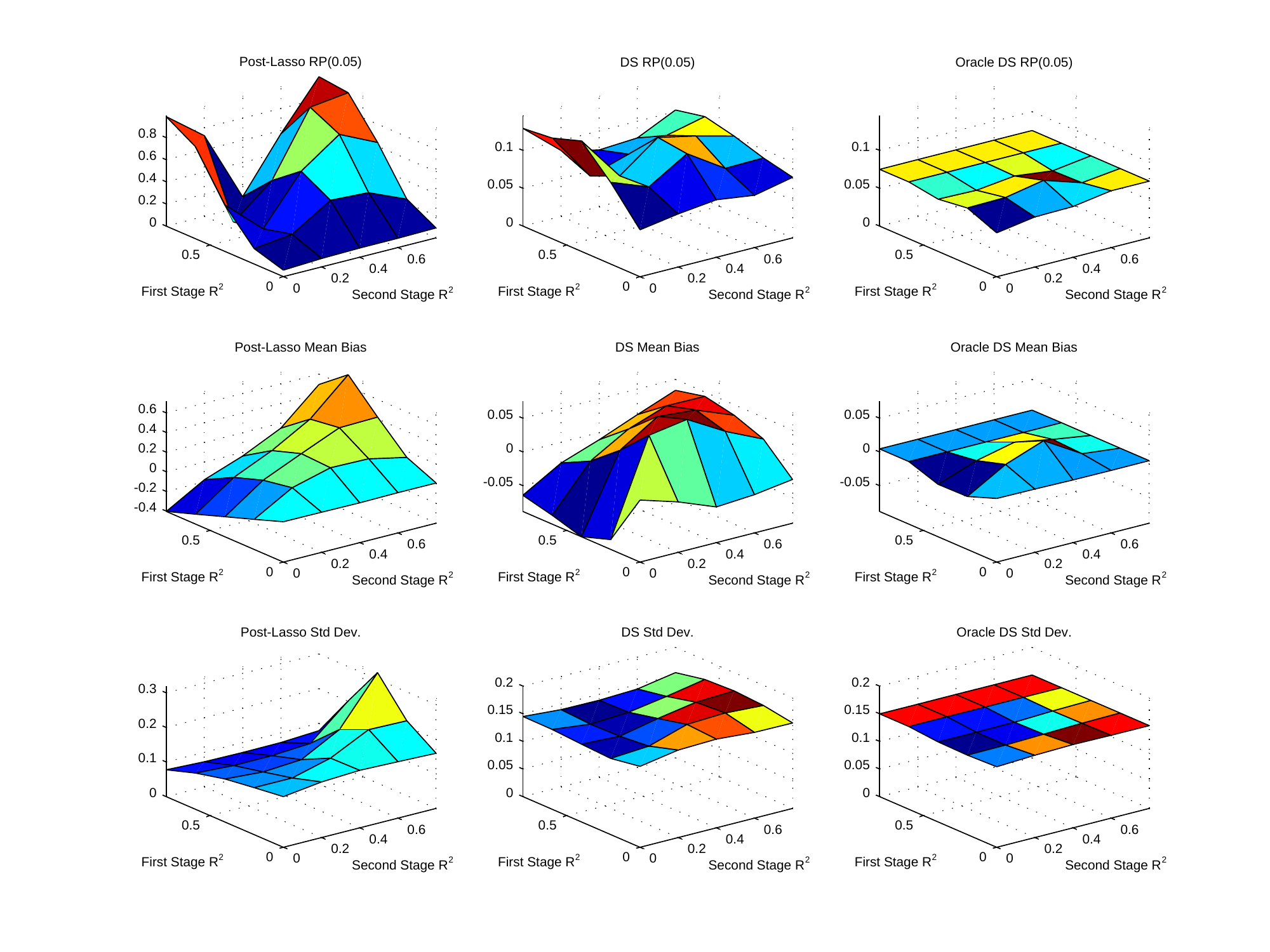}
	\label{fig:figure4}
\caption{Design 3}
\end{figure}

\pagebreak

\begin{figure}
\includegraphics[width=\textwidth]{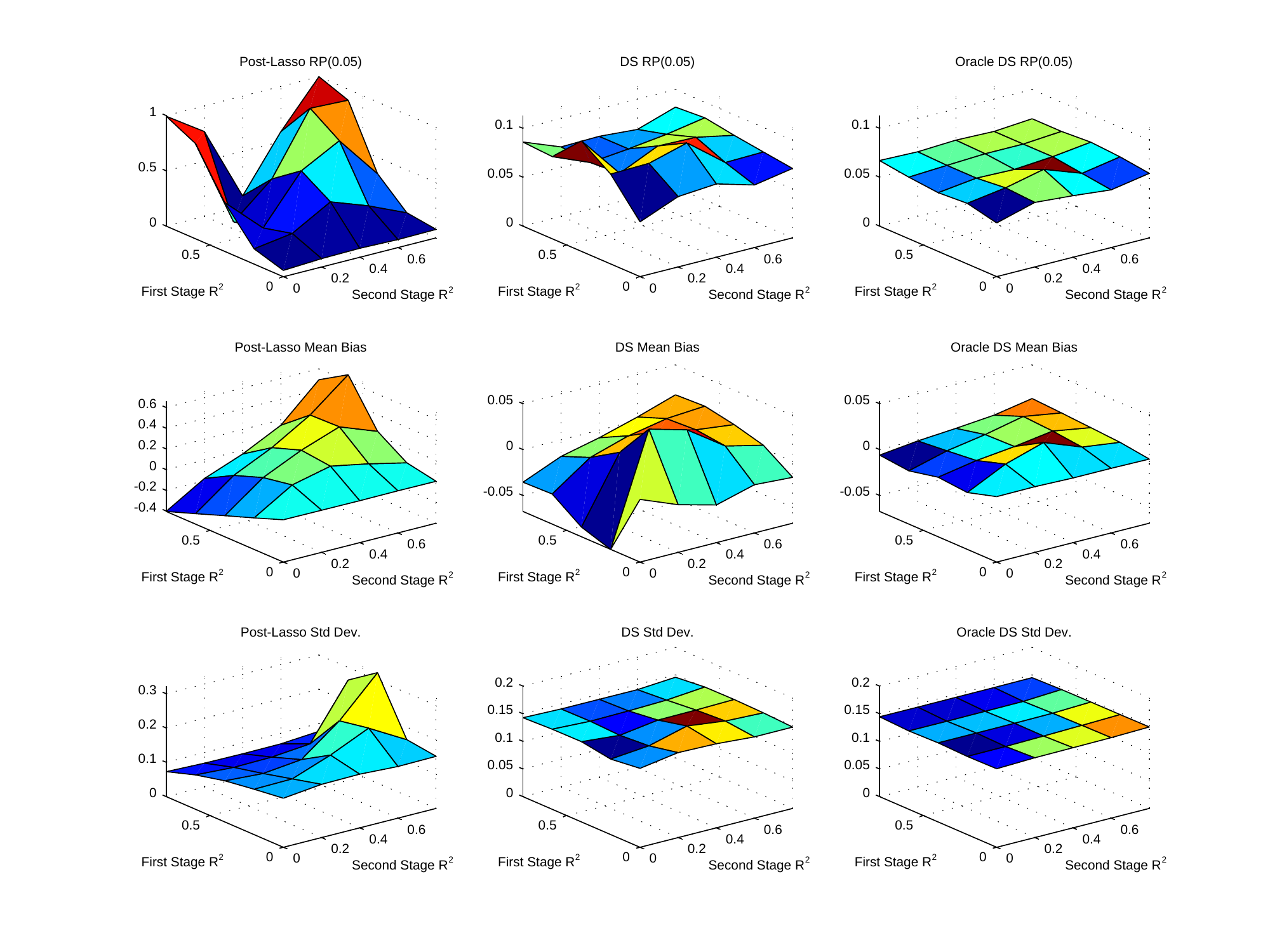}
	\label{fig:figure5}
\caption{Design 4}
\end{figure}

\pagebreak

\begin{figure}
\includegraphics[width=\textwidth]{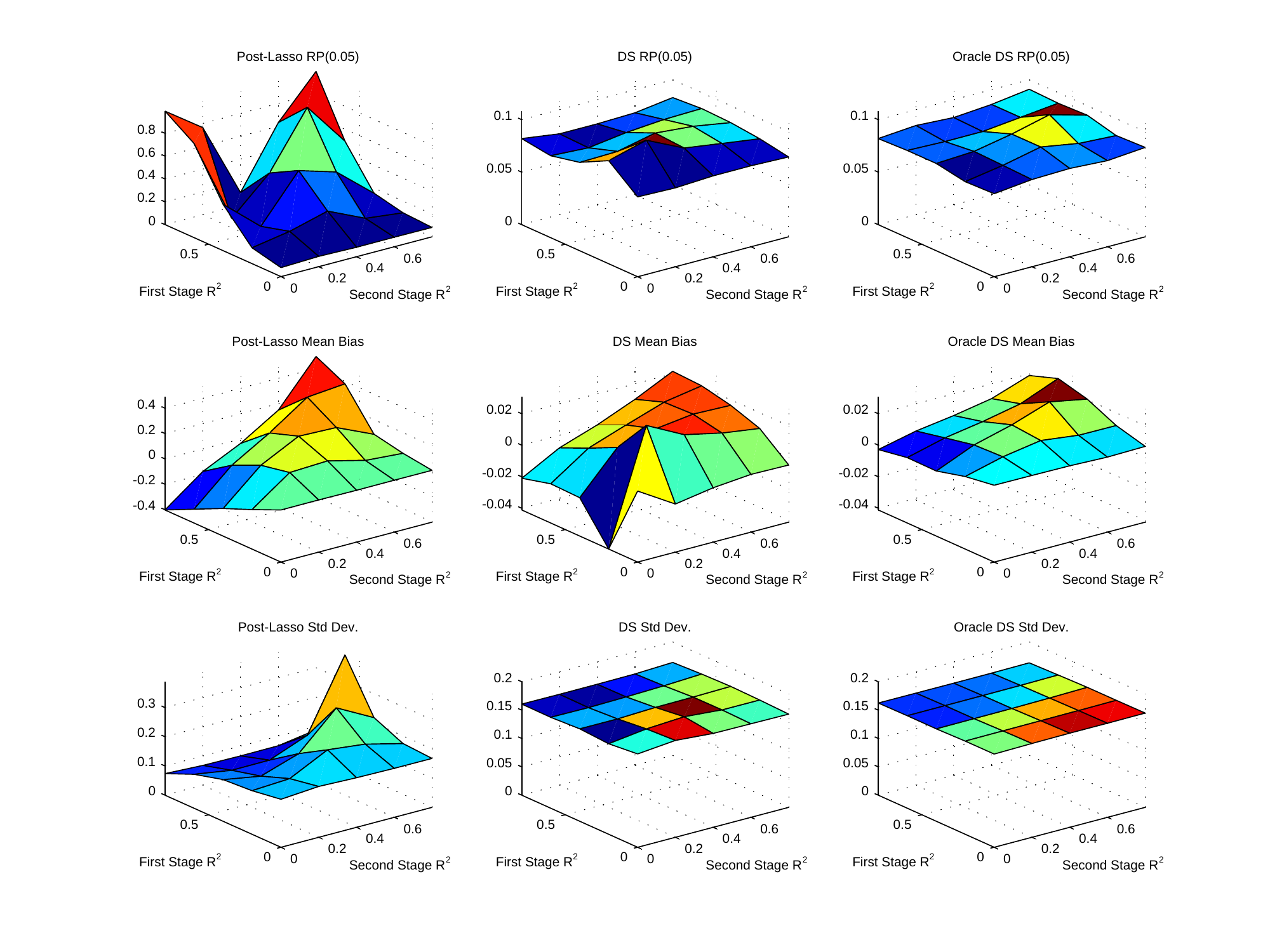}
	\label{fig:figure6}
\caption{Design 44}
\end{figure}

\pagebreak

\begin{figure}
\includegraphics[width=\textwidth]{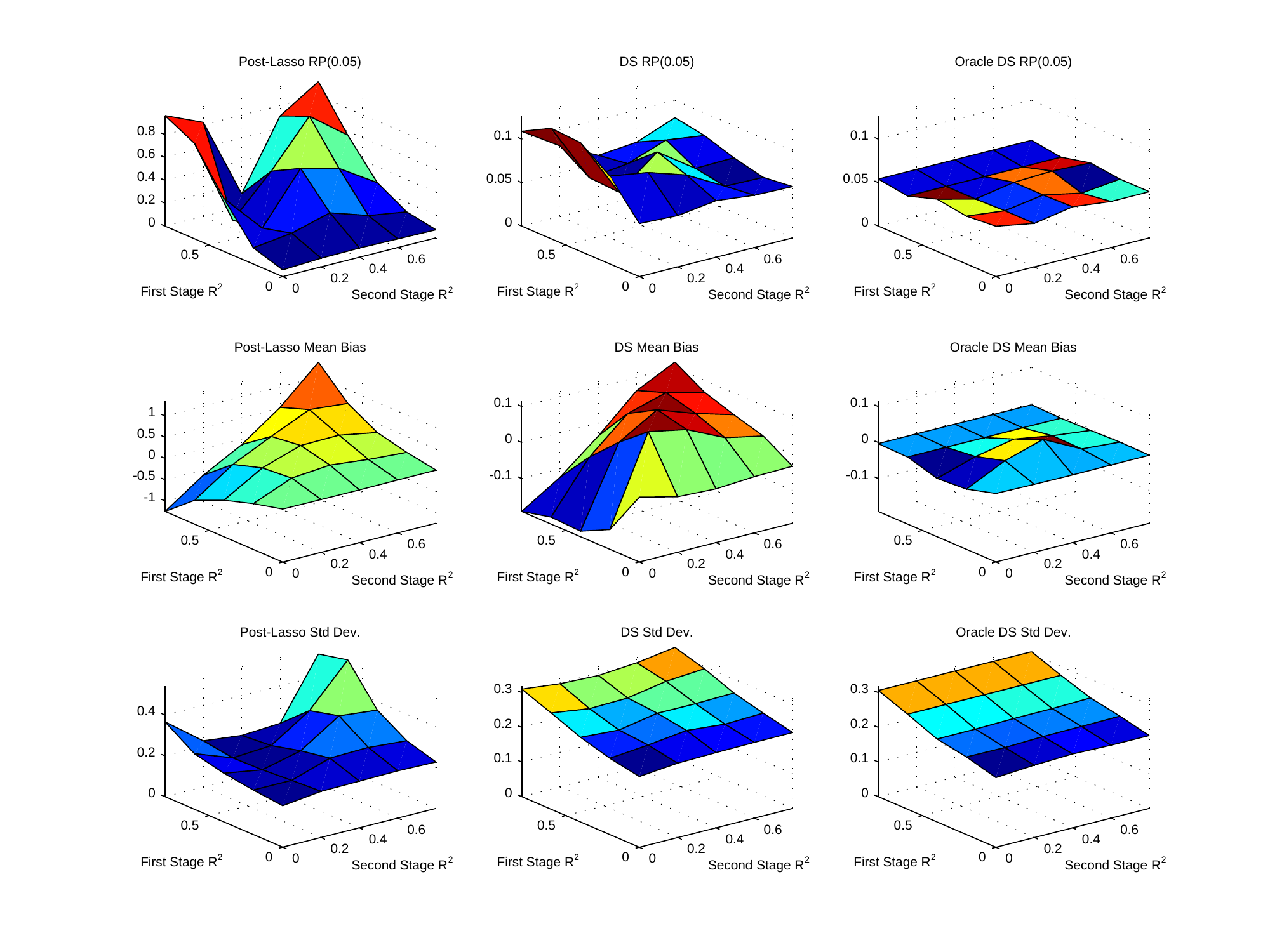}
	\label{fig:figure7}
\caption{Design 5}
\end{figure}

\pagebreak

\begin{figure}
\includegraphics[width=\textwidth]{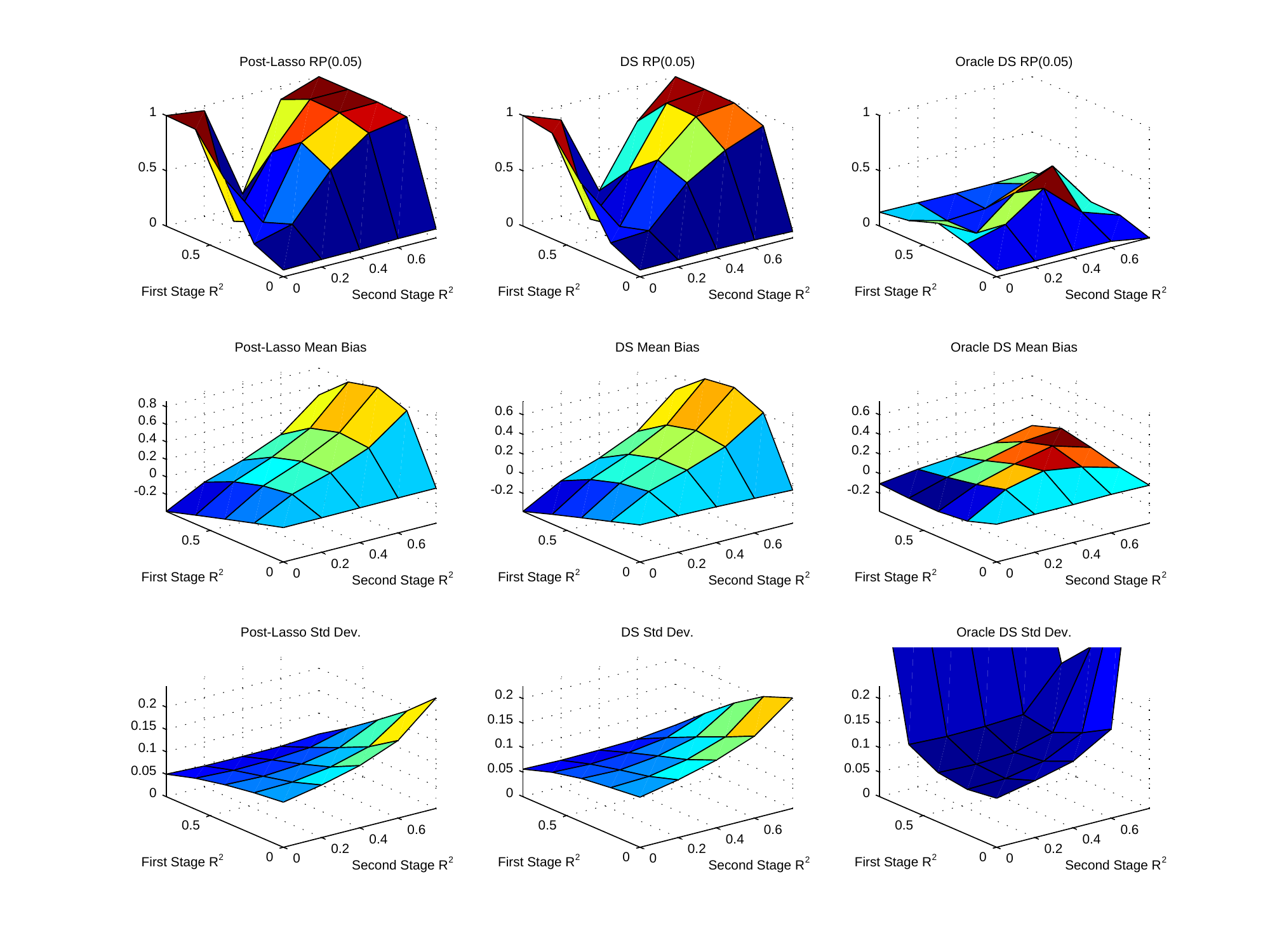}
	\label{fig:figure8}
\caption{Design 6}
\end{figure}

\pagebreak

\begin{figure}
\includegraphics[width=\textwidth]{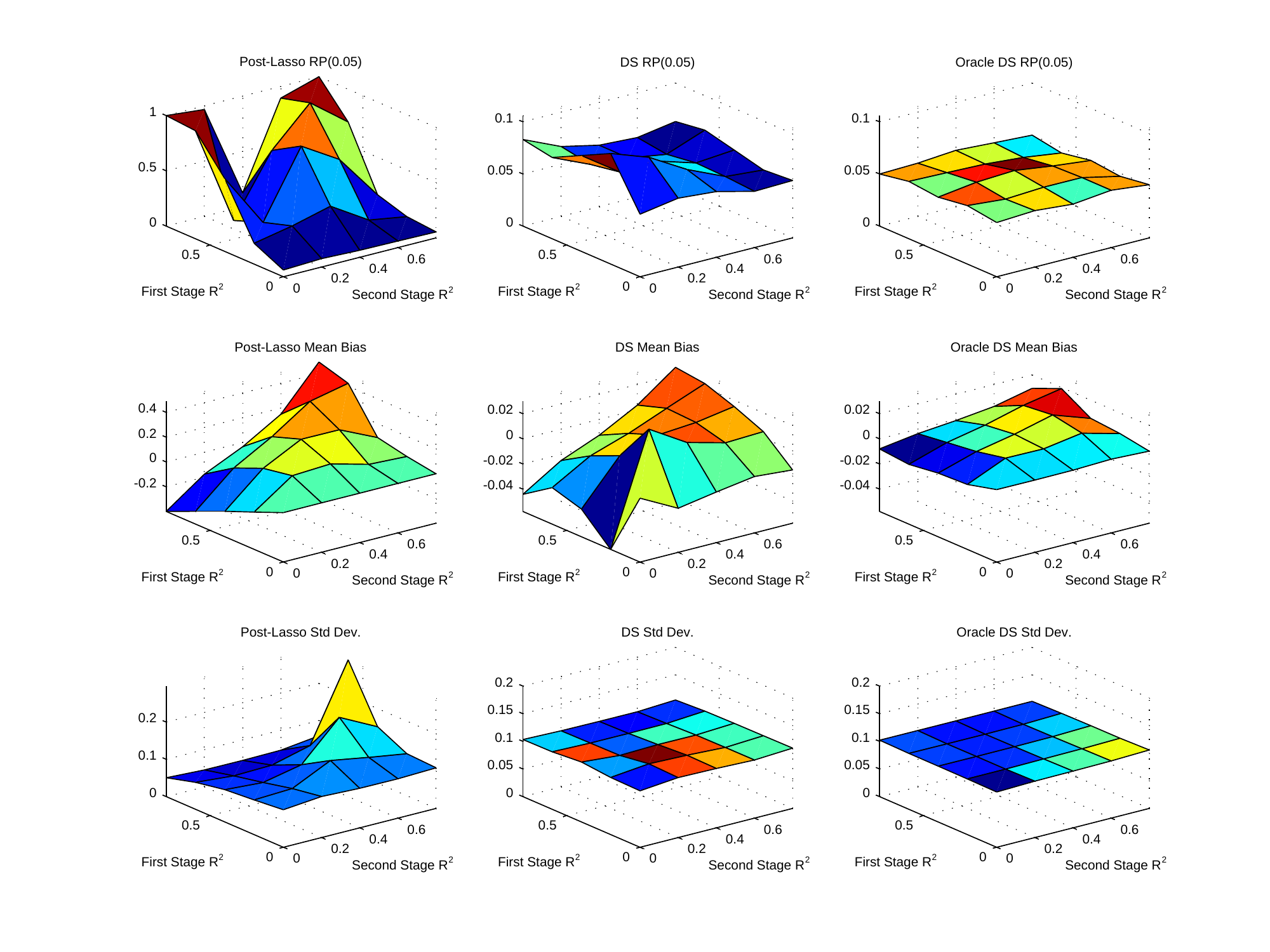}
	\label{fig:figure9}
\caption{Design 7}
\end{figure}

\pagebreak

\begin{figure}
\includegraphics[width=\textwidth]{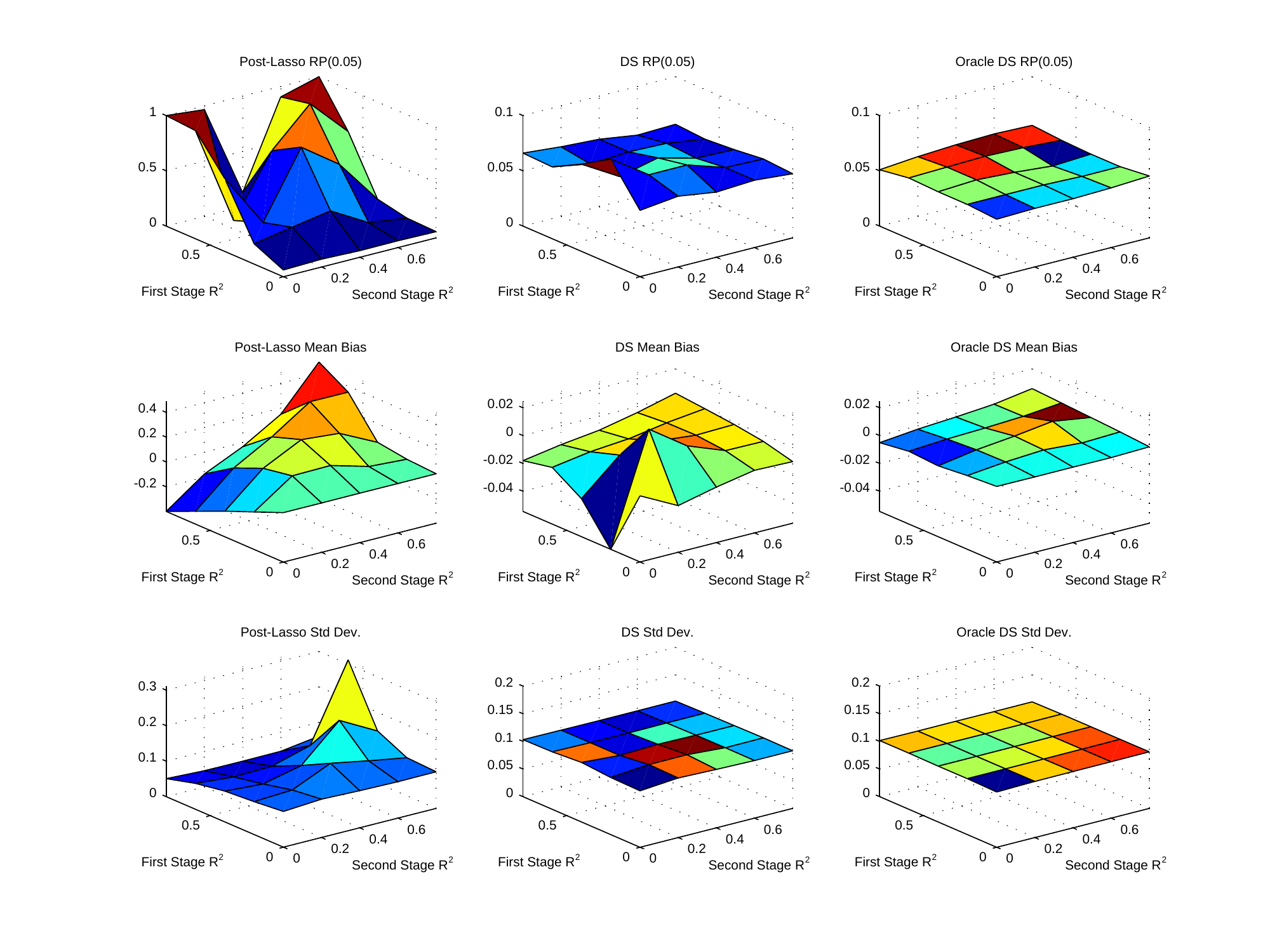}
	\label{fig:figure10}
\caption{Design 72}
\end{figure}

\pagebreak

\begin{figure}
\includegraphics[width=\textwidth]{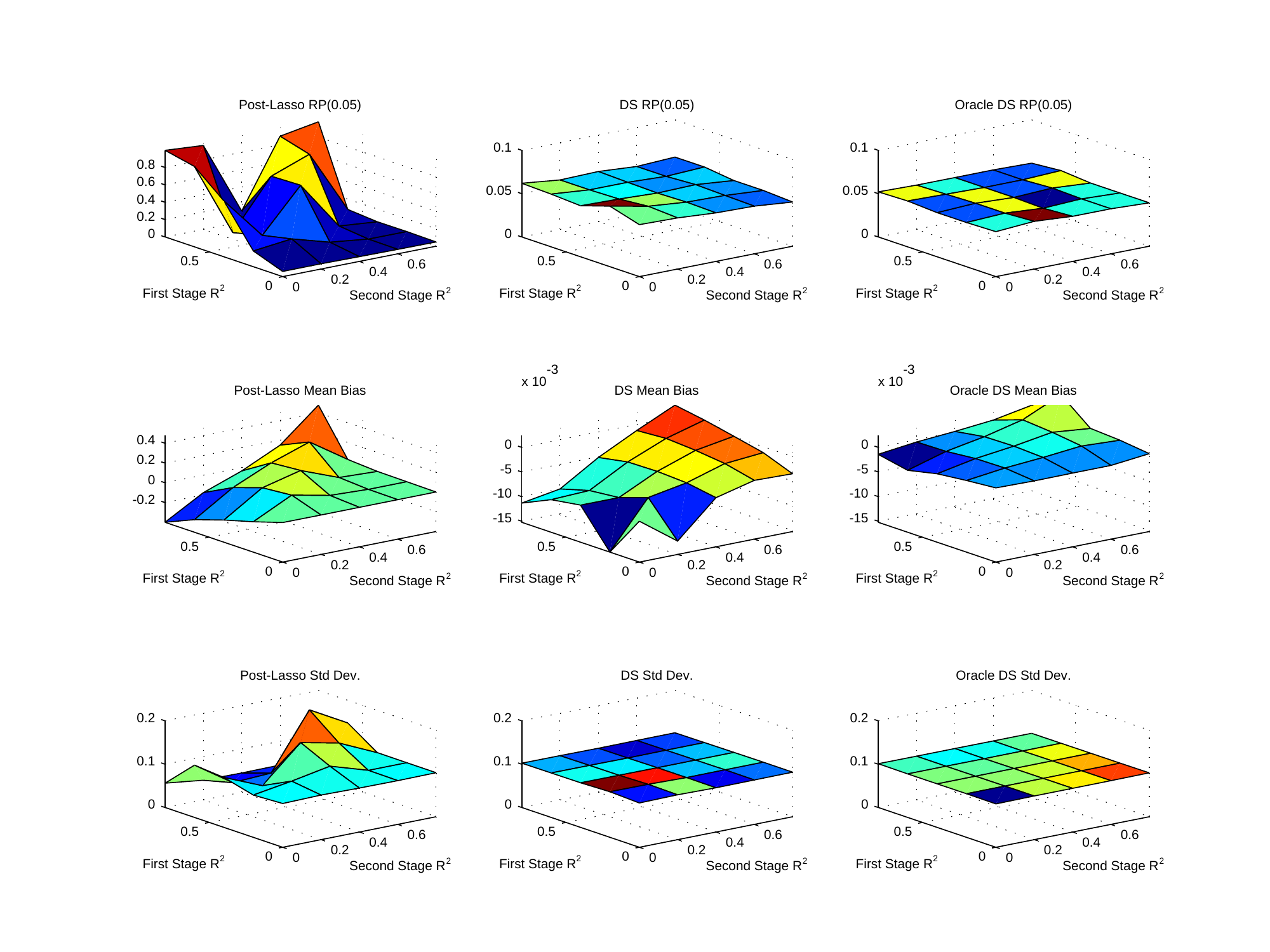}
	\label{fig:figure11}
\caption{Design 722}
\end{figure}

\pagebreak

\begin{figure}
\includegraphics[width=\textwidth]{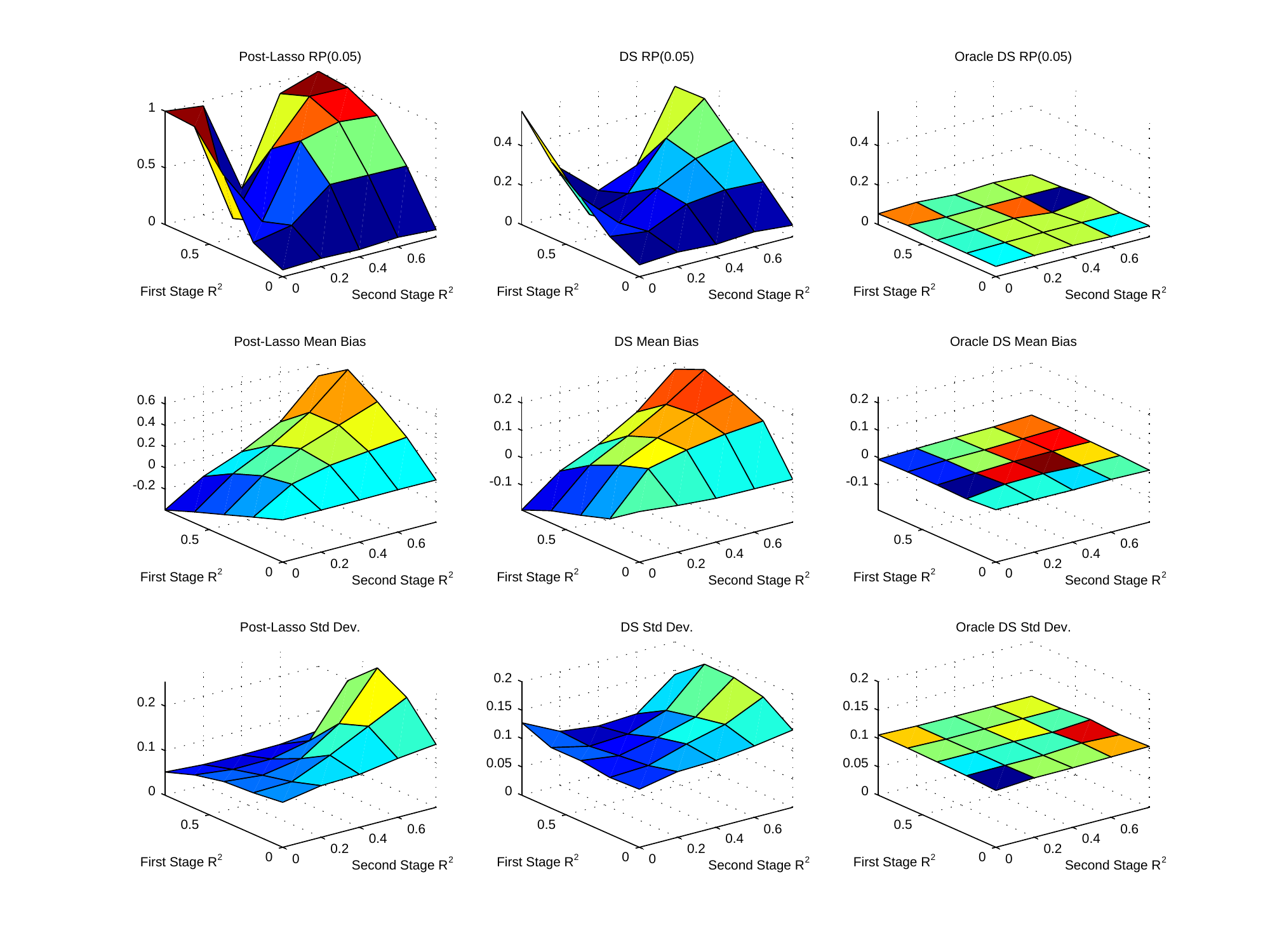}
	\label{fig:figure12}
\caption{Design 8}
\end{figure}

\pagebreak
\clearpage

\begin{figure}
\includegraphics[width=\textwidth]{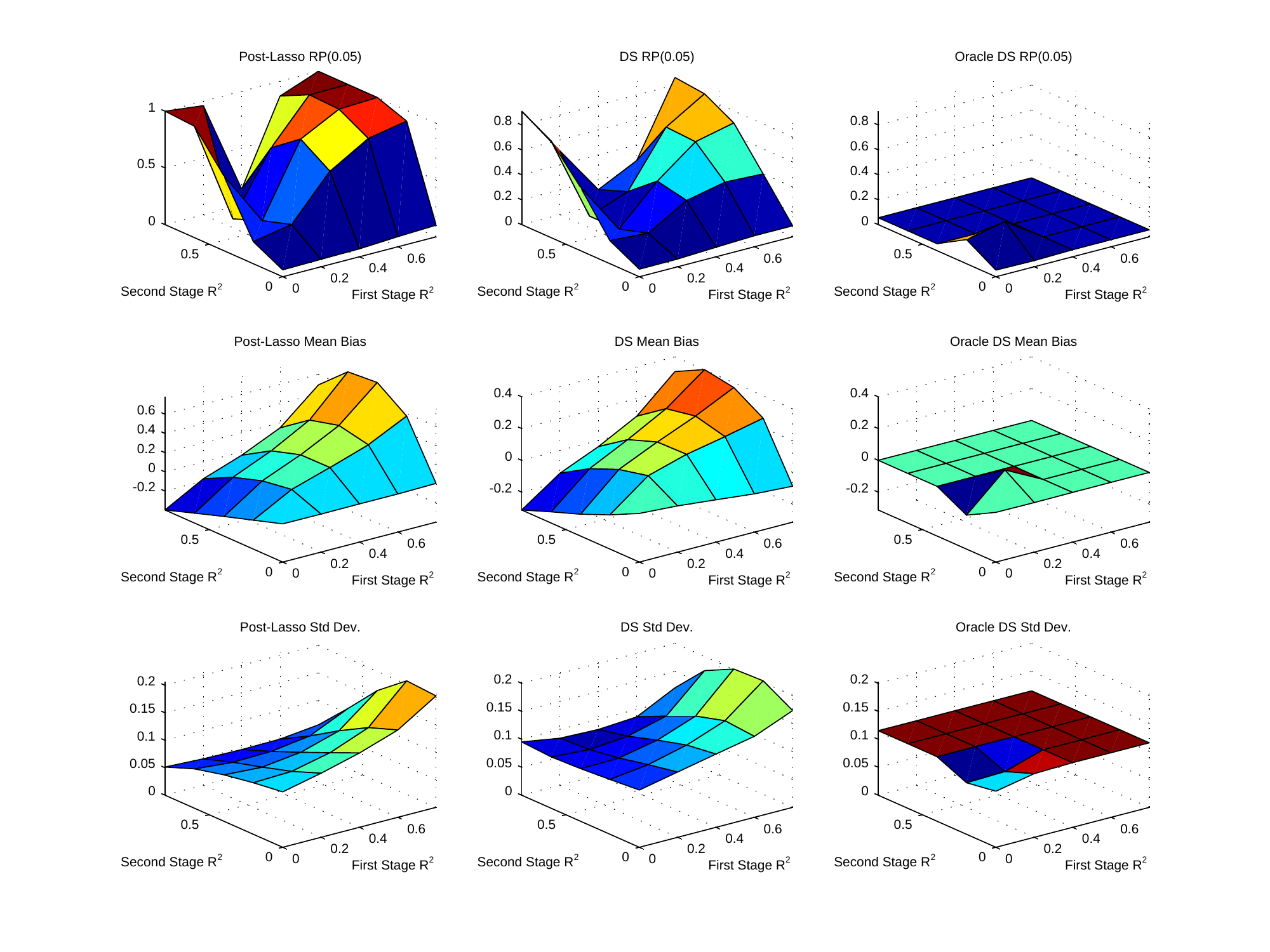}
	\label{fig:figure13}
\caption{Design 1001}
\end{figure}

\pagebreak

\begin{figure}
\includegraphics[width=\textwidth]{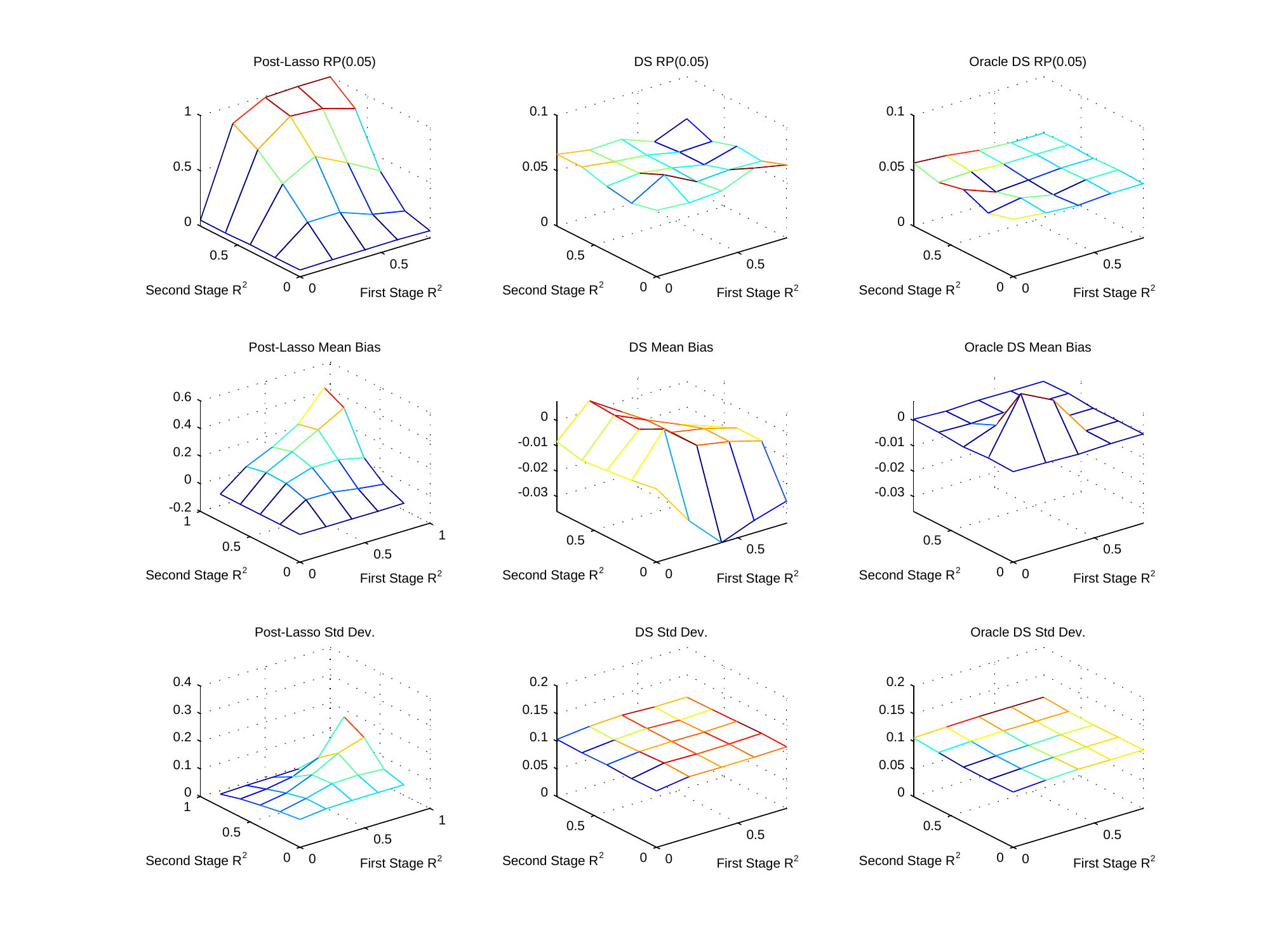}
	\label{fig:figure14}
\caption{Design 1a}
\end{figure}

\pagebreak

\begin{figure}
\includegraphics[width=\textwidth]{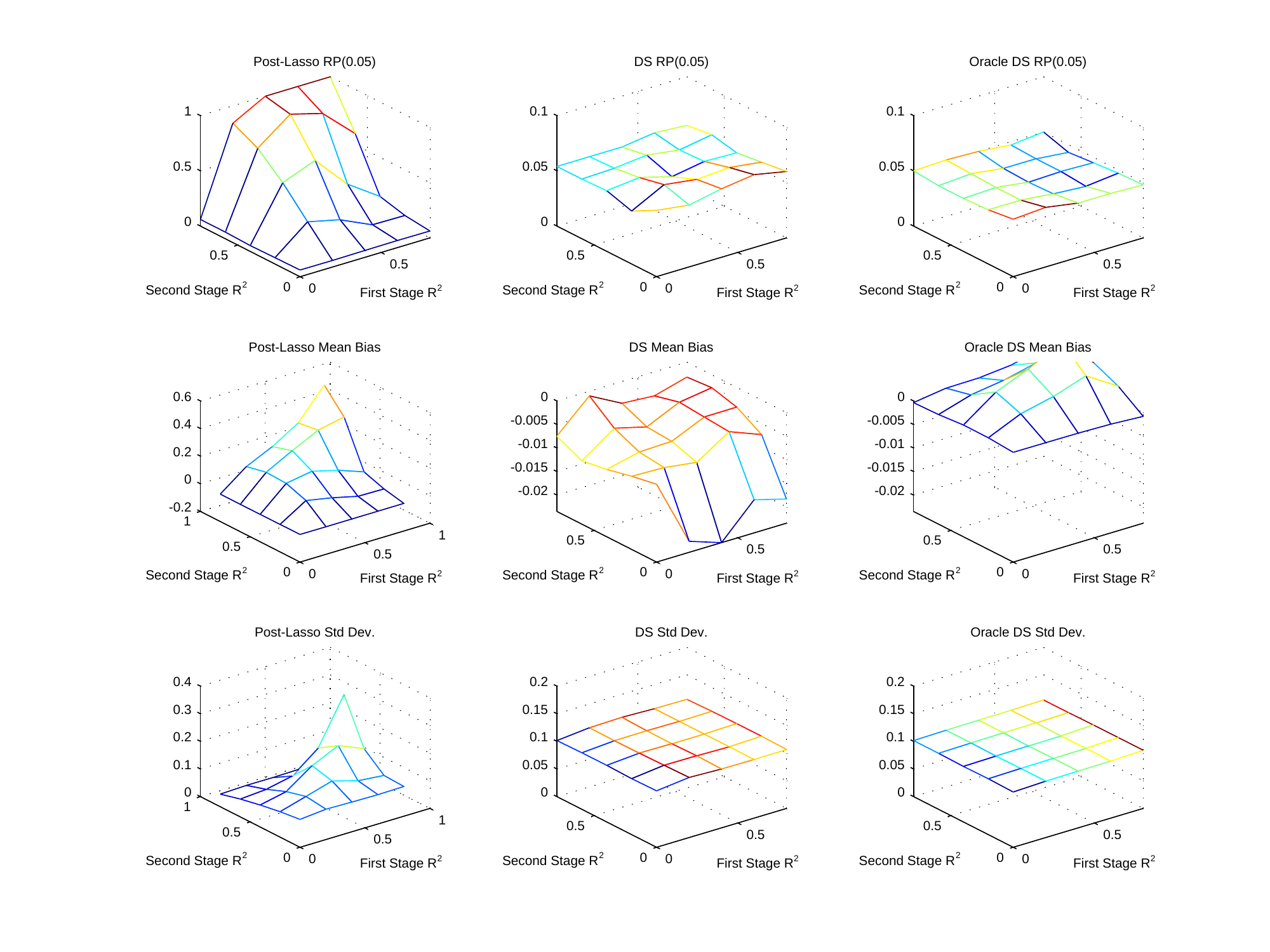}
	\label{fig:figure15}
\caption{Design 2a}
\end{figure}

\pagebreak

\begin{figure}
\includegraphics[width=\textwidth]{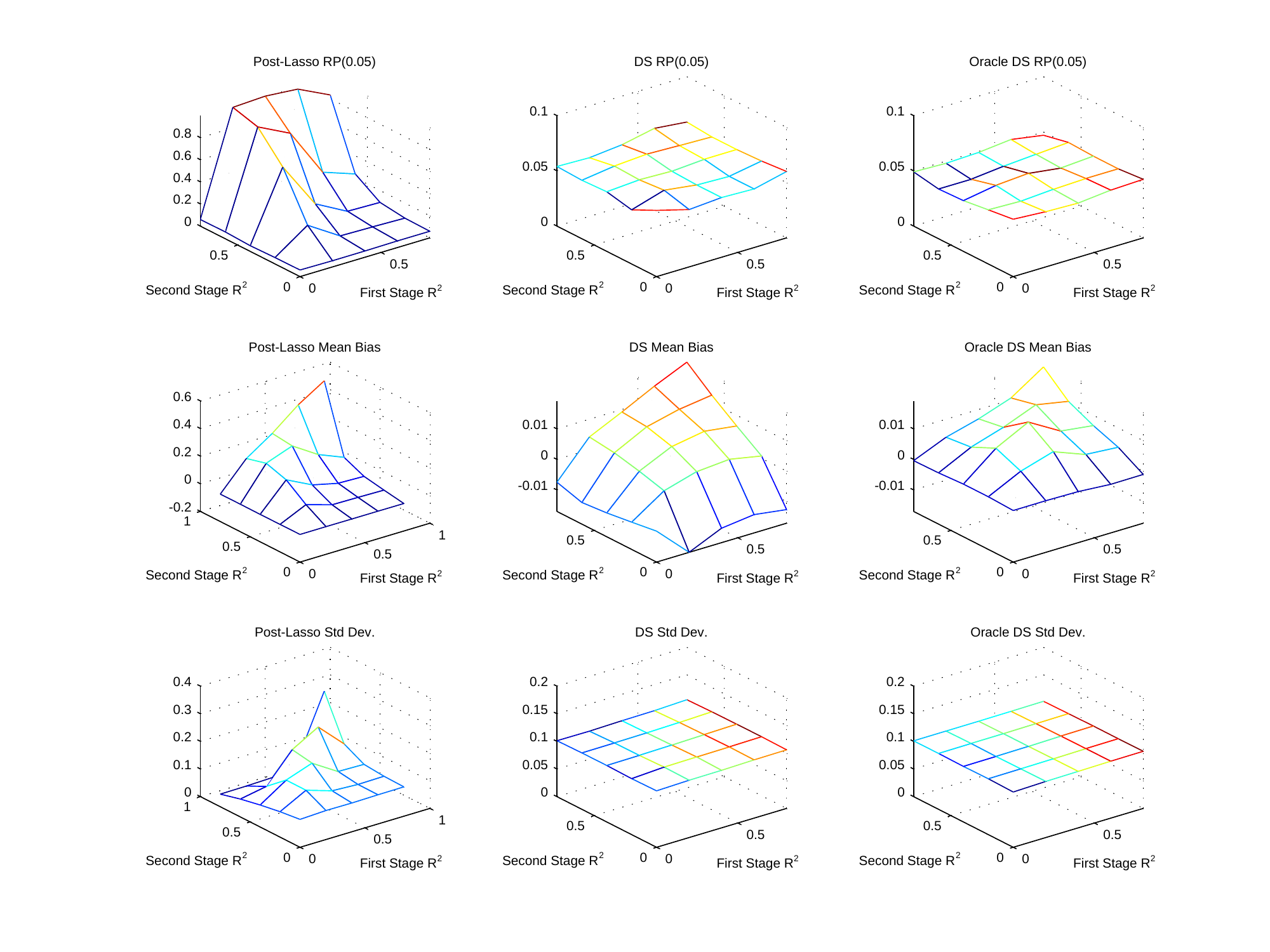}
	\label{fig:figure16}
\caption{Design 22a}
\end{figure}

\pagebreak

\begin{figure}
\includegraphics[width=\textwidth]{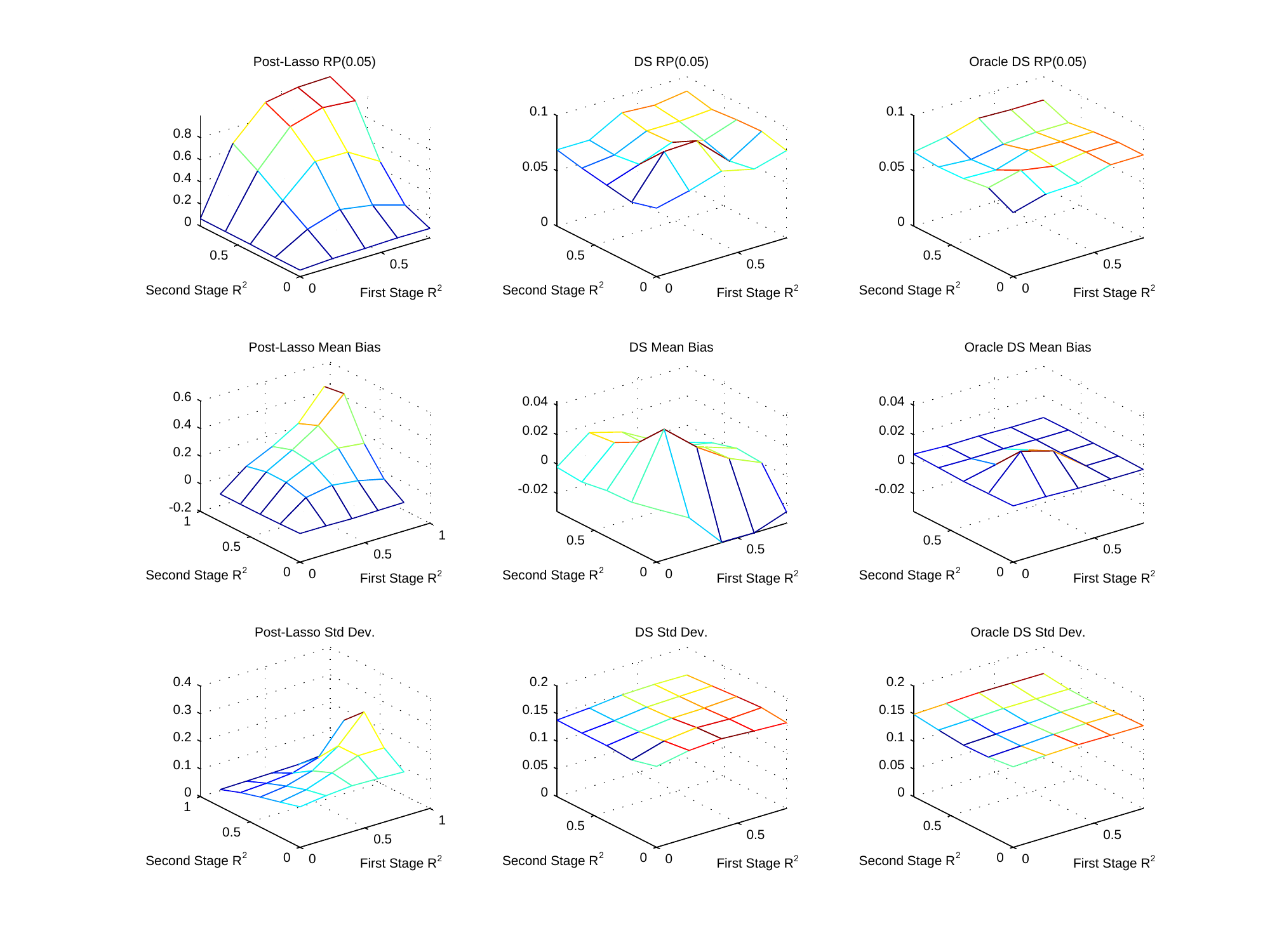}
	\label{fig:figure17}
\caption{Design 3a}
\end{figure}

\pagebreak

\begin{figure}
\includegraphics[width=\textwidth]{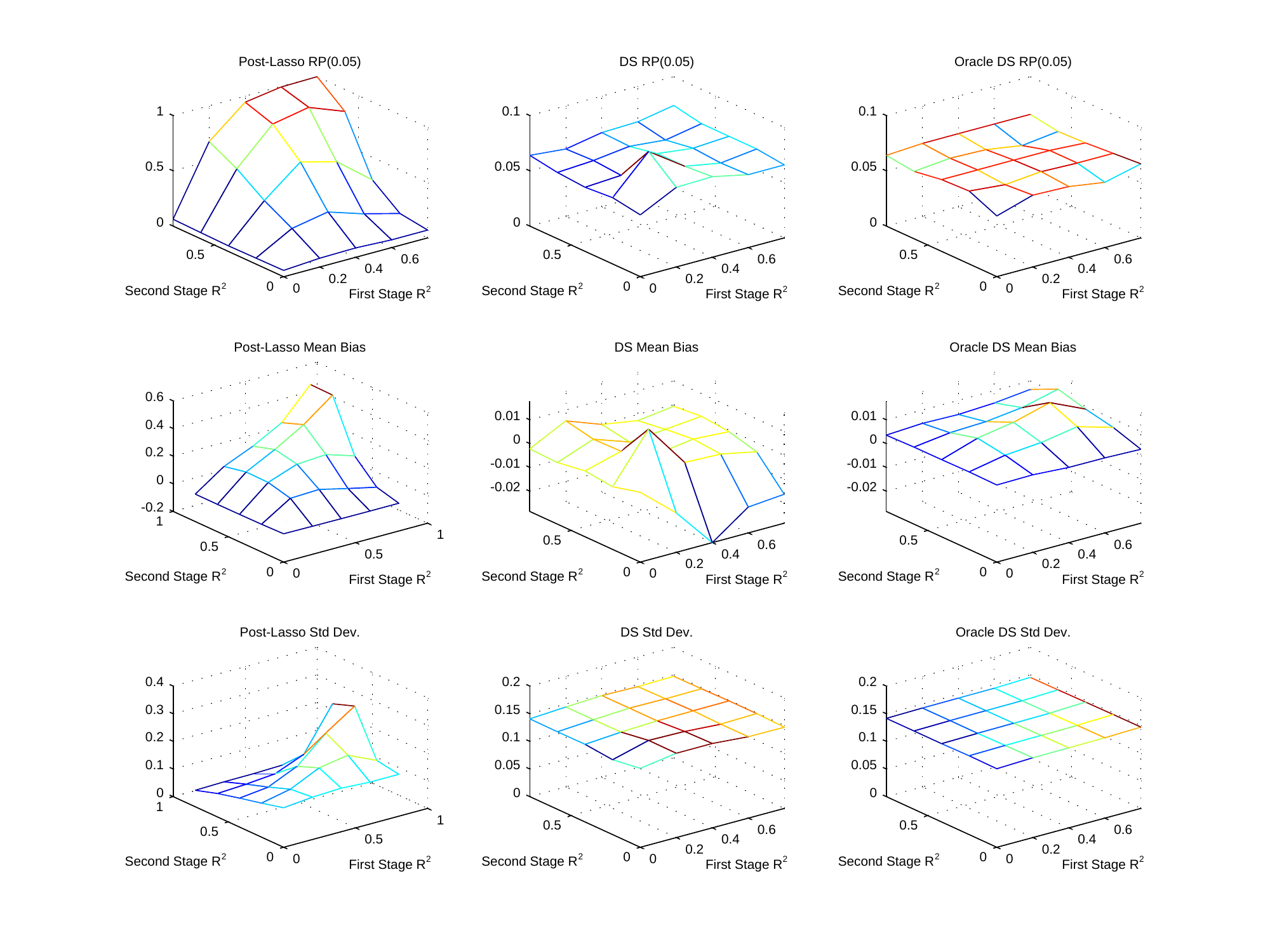}
	\label{fig:figure18}
\caption{Design 4a}
\end{figure}

\pagebreak

\begin{figure}
\includegraphics[width=\textwidth]{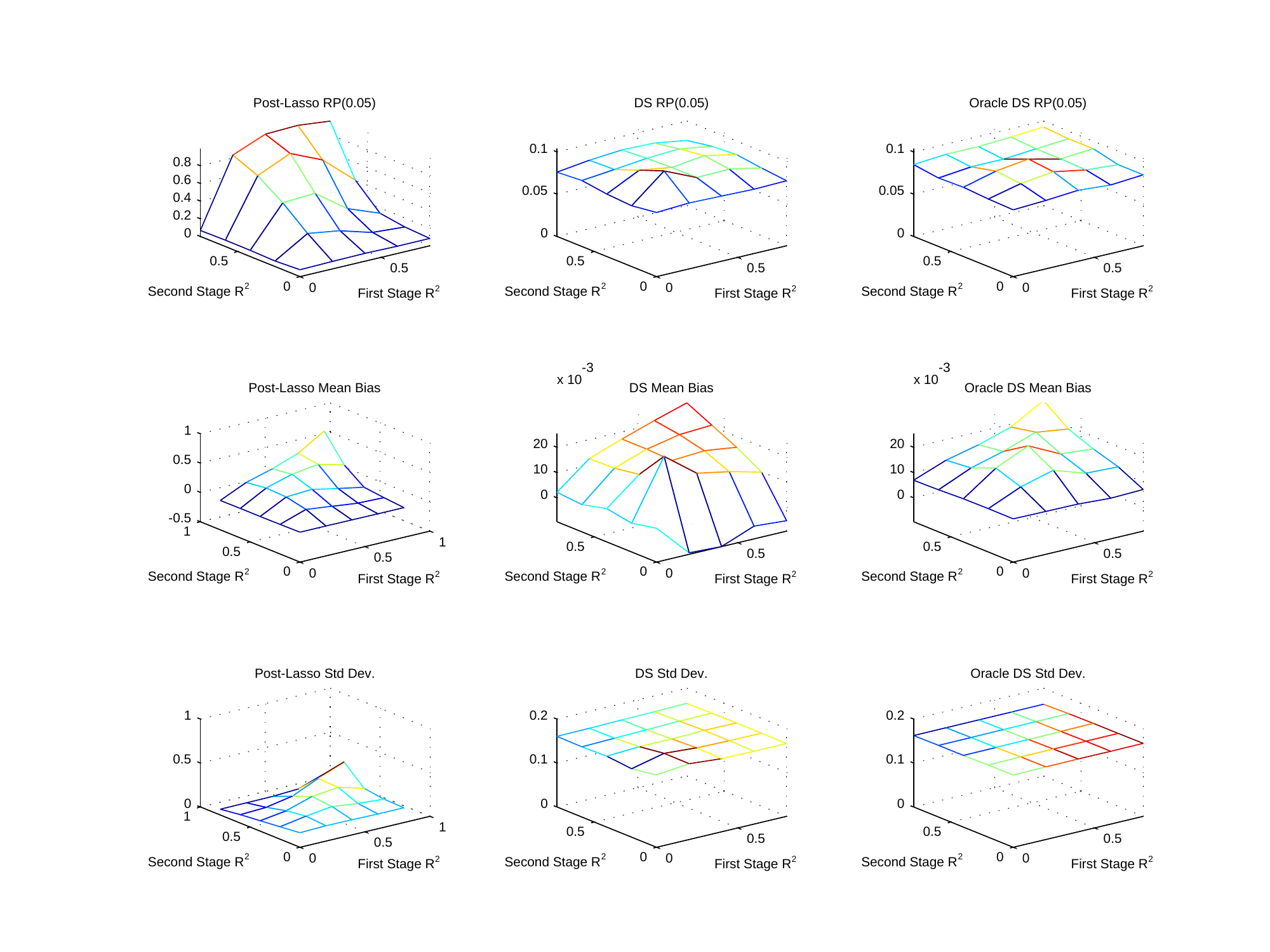}
	\label{fig:figure19}
\caption{Design 44a}
\end{figure}

\pagebreak

\begin{figure}
\includegraphics[width=\textwidth]{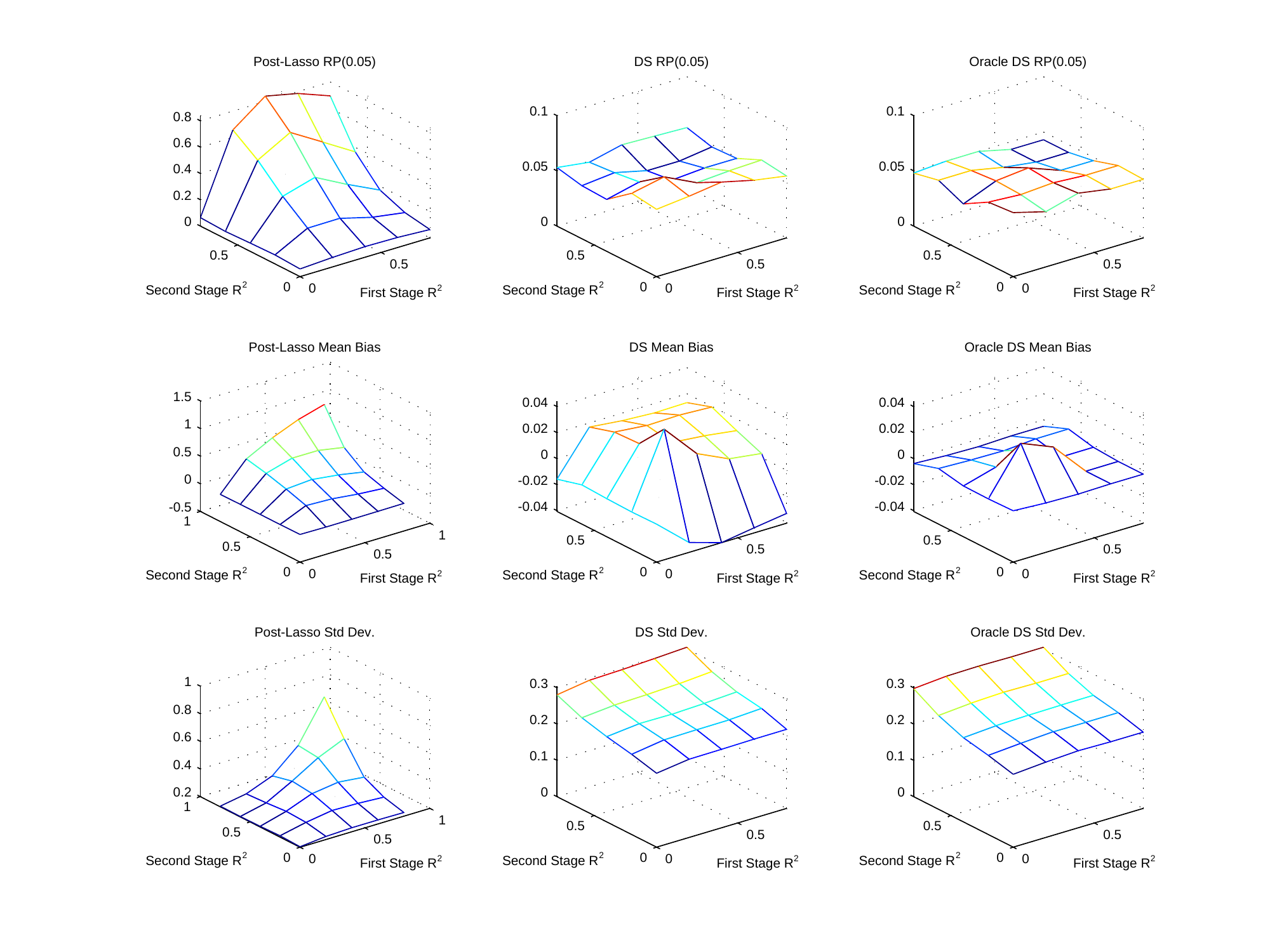}
	\label{fig:figure20}
\caption{Design 5a}
\end{figure}

\pagebreak

\begin{figure}
\includegraphics[width=\textwidth]{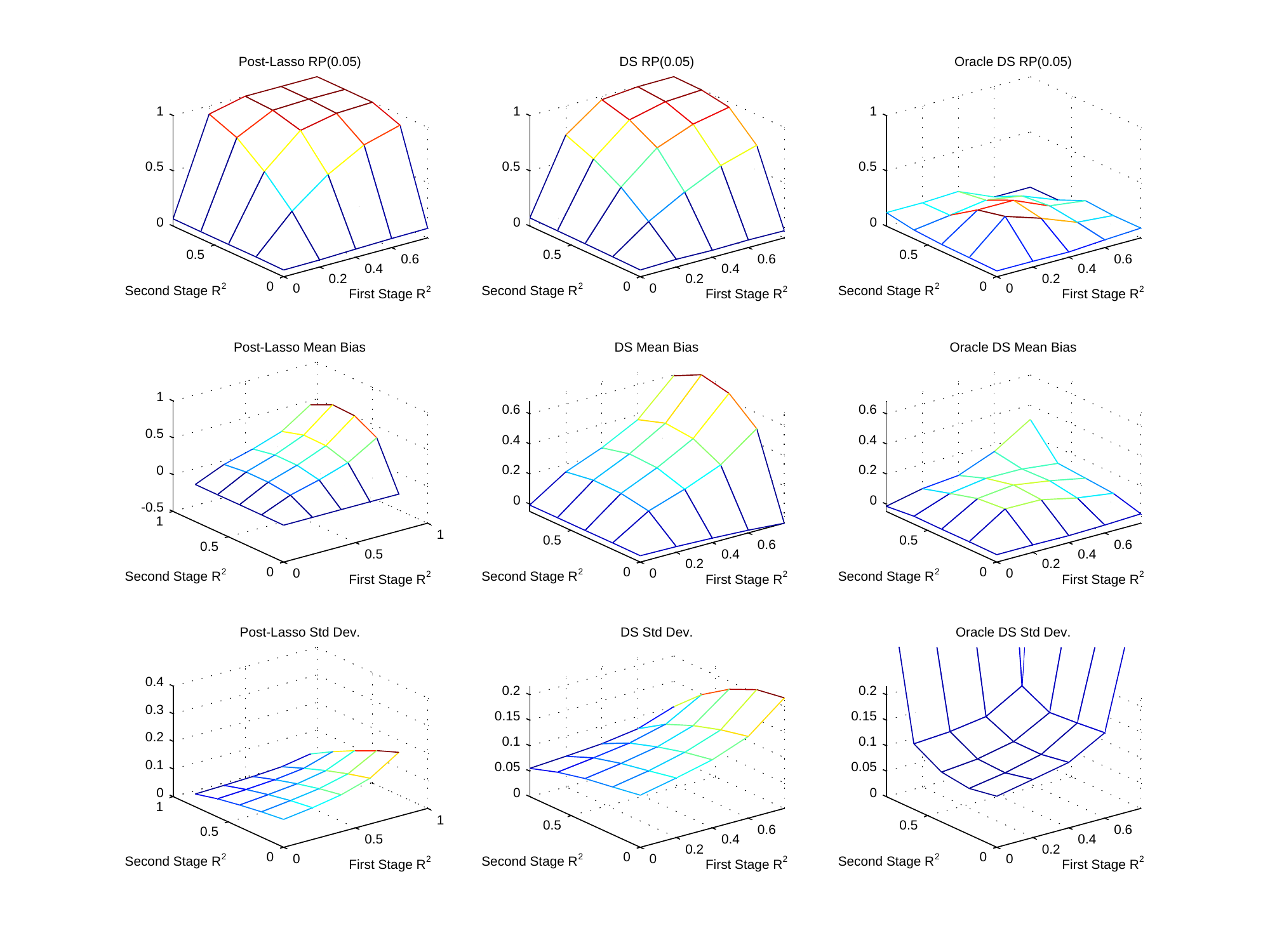}
	\label{fig:figure21}
\caption{Design 6a}
\end{figure}

\pagebreak

\begin{figure}
\includegraphics[width=\textwidth]{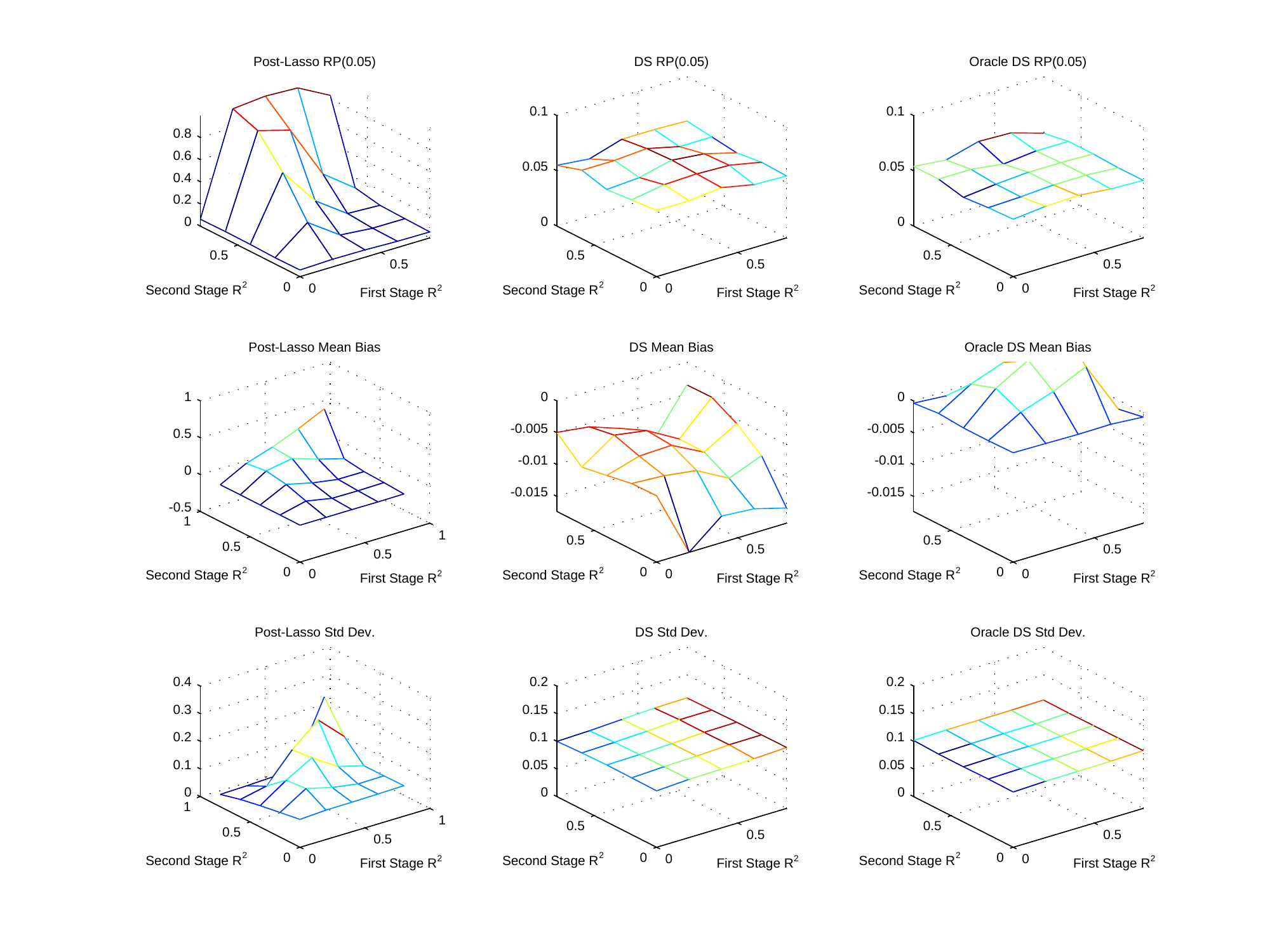}
	\label{fig:figure22}
\caption{Design 7a}
\end{figure}

\pagebreak

\begin{figure}
\includegraphics[width=\textwidth]{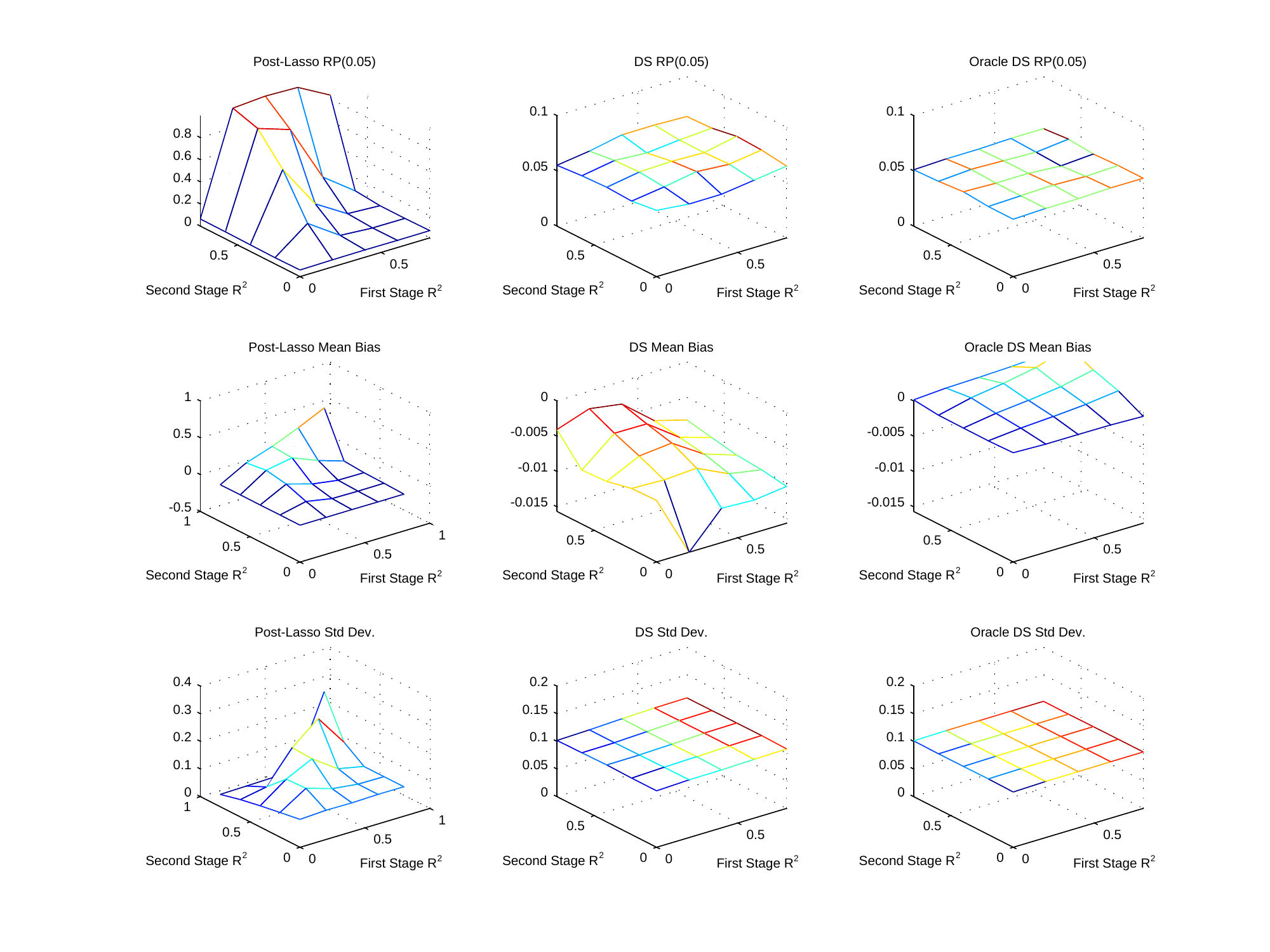}
	\label{fig:figure23}
\caption{Design 72a}
\end{figure}

\pagebreak

\begin{figure}
\includegraphics[width=\textwidth]{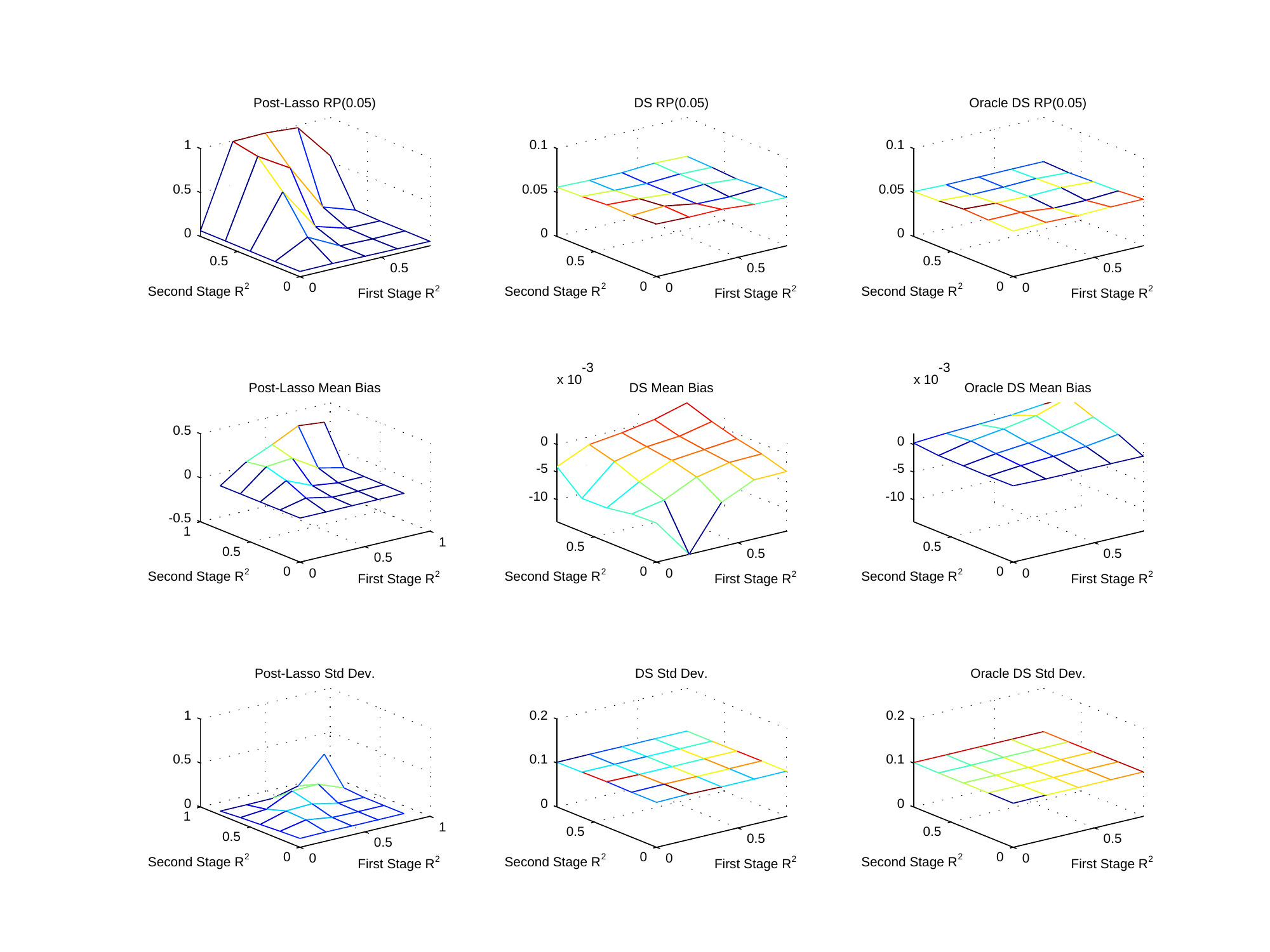}
	\label{fig:figure24}
\caption{Design 722a}
\end{figure}

\pagebreak

\begin{figure}
\includegraphics[width=\textwidth]{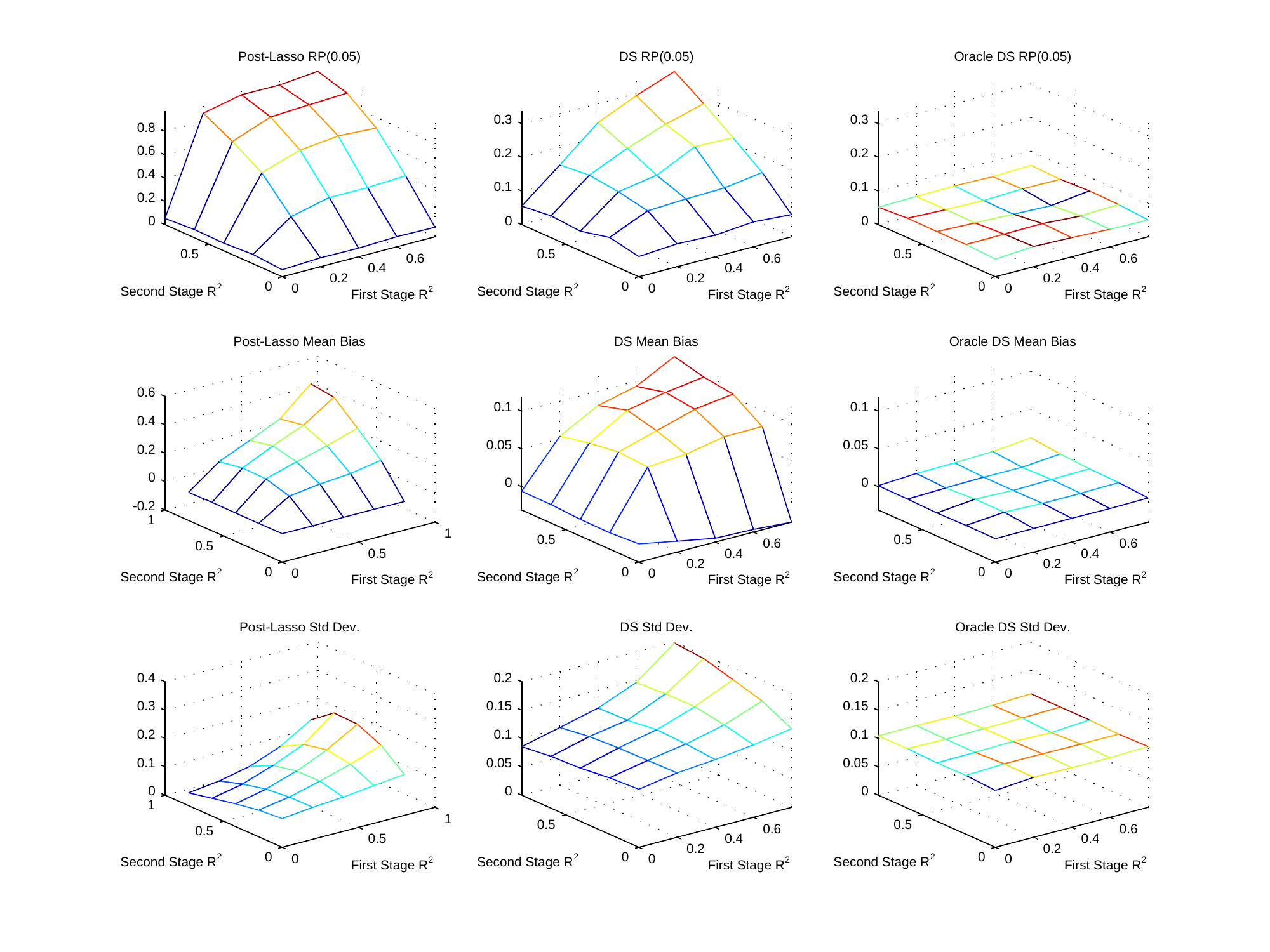}
	\label{fig:figure25}
\caption{Design 8a}
\end{figure}

\pagebreak
\clearpage

\begin{figure}
\includegraphics[width=\textwidth]{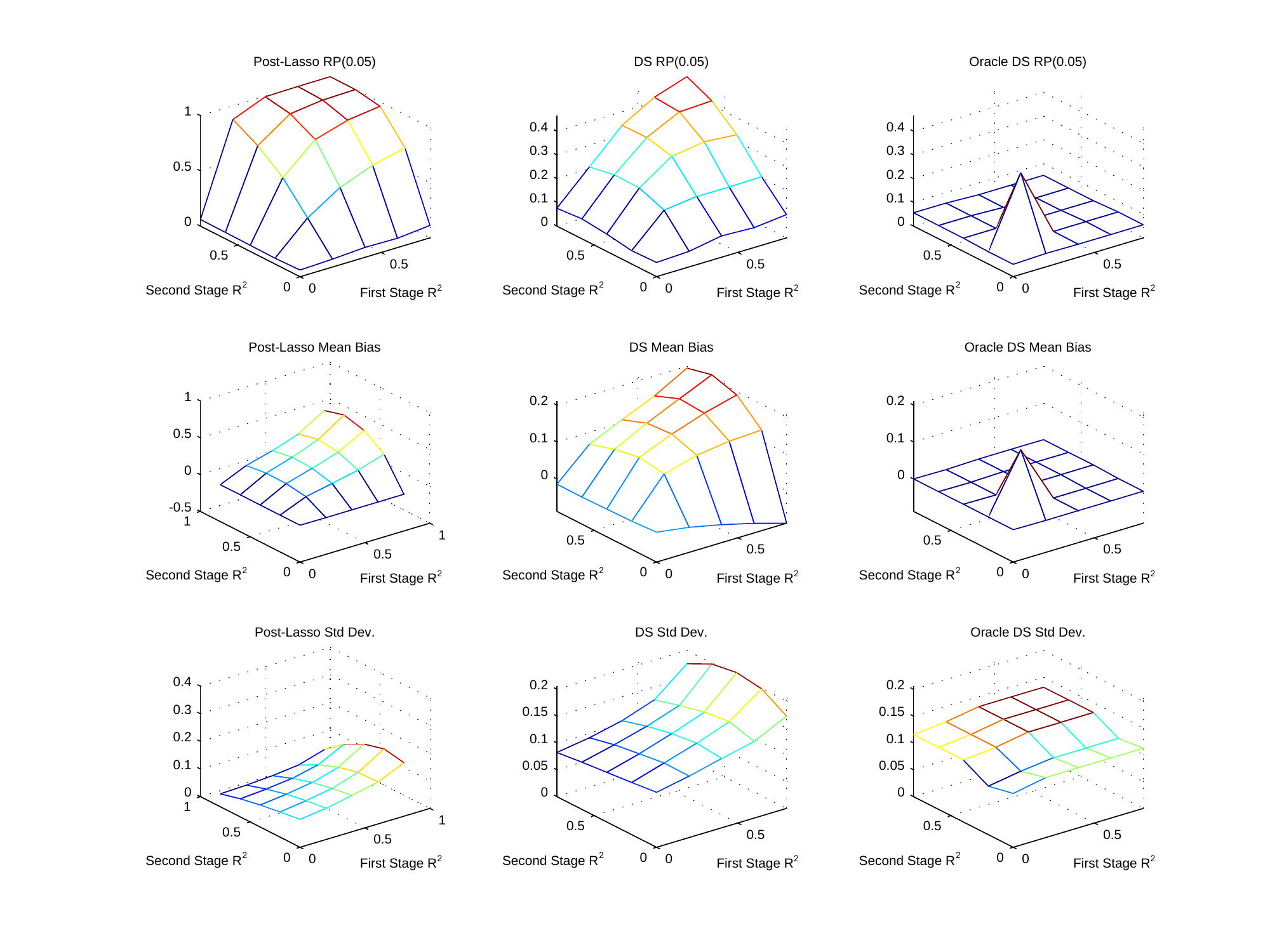}
	\label{fig:figure26}
\caption{Design 1001a}
\end{figure}

\bibliographystyle{econometrica}
\bibliography{mybibVOLUME}

\end{document}